\documentclass[10pt]{article}
\usepackage{amsmath,amsthm,amsfonts,amssymb,epsfig,enumerate}
\usepackage{flafter}

\newtheorem{them}{Theorem}[section]

\newtheorem{rem}[them]{Remark}
\numberwithin{equation}{section}
\newcommand{\norm}[1]{\left\Vert#1\right\Vert}

\usepackage{graphicx}
\usepackage{subfigure}
\usepackage{float}
\usepackage{epstopdf}
\usepackage{color}
\usepackage{colortbl,xcolor}

\begin{document}
\title{Numerical solution of the boundary value problem of elliptic equation by Levi function scheme}
\author{Jinchao Pan$^\dag$ \quad Jijun Liu$^{\dag\;\ddag}$\;\thanks{Corresponding author: Prof. Dr. J.J.Liu, email: jjliu@seu.edu.cn}\\
$^\dag$School of Mathematics/Shing-Tung Yau Center of Southeast University\\ Southeast University\\
Nanjing, 210096, P.R.China\\
$^{\dag\;\ddag}$Nanjing Center for Applied Mathematics\\
Nanjing, 211135, P.R.China
}

\maketitle

%%%% ----------------------------------------------------------------------in
%
\begin{abstract}
For boundary value problem of an elliptic equation with variable coefficients  describing the physical field distribution in inhomogeneous media, the Levi function  can represent the solution in terms of volume and surface potentials, with the drawback that the volume potential involving in the solution expression requires heavy computational costs as well as the solvability of the integral equations with respect to the density pair. We introduce an modified integral expression for the solution to an elliptic equation in divergence form under the Levi function framework. The well-posedness of the linear integral system with respect to the density functions to be determined is rigorously proved. Based on the singularity decomposition for the Levi function, we propose two schemes to deal with the volume integrals so that the density functions can be solved efficiently. One method is an adaptive discretization scheme (ADS) for computing the integrals with continuous integrands, leading to the uniform accuracy of the integrals in the whole domain, and consequently the efficient computations for the density functions. The other  method is the dual reciprocity method (DRM) which is a meshless approach converting the volume integrals into  boundary integrals equivalently by expressing the volume density as the combination of the radial basis functions determined by the interior grids. The proposed schemes are justified numerically to be of satisfactory computation costs. Numerical examples in 2-dimensional and 3-dimensional cases are presented to show the validity of the proposed schemes.
\end{abstract}

\vskip 0.3cm
%{\bf AMS subject classifications:}
%

\vskip 0.3cm


\section{Introduction}

For a bounded smooth domain $\Omega\subset \mathbb{R}^d$ with $d=2,3$, consider the system
\begin{eqnarray}\label{liu11-01}
   \begin{cases}
  -\nabla\cdot(\sigma(x)\nabla u(x))=F(x), &x \in \Omega,\\
  u(x)=f(x), &x\in\partial \Omega
   \end{cases}
\end{eqnarray}
with $0<\sigma_0<\sigma(x)\in C(\overline\Omega)$. It is well-known that
there exists a unique solution $u\in H^1(\Omega)$  for given $F\in L^2(\Omega)$ and $f \in H^{1/2}(\partial\Omega)$, see \cite{Evans}. The above boundary value problem is an mathematical model describing the distribution of electrical potential $u(x)$ with $\sigma(x)$ being the conductivity in conductive media $\Omega\subset \mathbb{R}^d$.

For the boundary value problem \eqref{liu11-01}, it is well-known that the solution can be solved numerically by finite element method (FEM) for the domain $\Omega$ of general shape. Based on the variational form of \eqref{liu11-01}, the solution $u$ is obtained in the whole domain from its base function expansions constructing in terms of the mesh discretizations of the domain $\Omega$. However, there exists another way to solve the solution $u$ for the elliptic equation with constant coefficients, namely, the boundary elements method (BEM), which represents the solution in terms of the surface potential with some density function to be determined. Once the density function in the surface layer potential is determined from the boundary condition, the solution in every point of $\Omega$ is obtained by computing the surface integral defined in $\partial\Omega$. Such techniques have been studied for solving PDEs in different forms such as the Helmholtz equation, the Lame system as well as the biharmoni equations, see \cite{George}. The other applications of BEM are in the areas of inverse and ill-posed problems, where the adjoint operators are relatively easy to be derived due to the integral representations of the solution, when we find the regularizing solutions by minimizing the corresponding cost functional iteratively, see \cite{Chap1, Chap2, Colton1}.

However, in the cases of the PDEs with variable coefficients governing the physical field in inhomogeneous media, the BEM scheme does not work anymore for finding its solution, since the fundamental solution to the differential operator with variable coefficients, although exists, cannot be represented explicitly. In these cases, it is also well-known that,
the Levi function pair $(P(x,y), R(x,y))$  for elliptic operator $A(x,\partial)$ with variable coefficients defined by
$$A(x,\partial)P(x,y)=\delta(x-y)+R(x,y),$$
which is considered as the generalization of the fundamental solution, plays an important rule in theoretical studies for PDEs \cite{Pomp, Mir}.

Although such a generalized expression scheme has been studied thoroughly, the numerical realizations seem not to be implemented efficiently, due to the reason that the solvability of the corresponding integral system with respect to the density functions heavily depends on the potentials, e.g., the single-layer potential or double-layer one, as well as the form of boundary conditions. In \cite{Beshley}, the system  \eqref{liu11-01} with $F(x)=0$ is solved numerically based on the solution representation
\begin{equation}\label{liu11-01-01}
u(x):=\int_{\Omega}\mu(y)P(x,y)dy+\int_{\partial\Omega}\psi(y)P(x,y)ds(y)
\end{equation}
for Levi function $P(x,y)$. The discretization scheme for solving the corresponding integral system is established there, the numerical examples show the potential efficiency of this scheme for \eqref{liu11-01}, see also \cite{Mikha}. However, due to the single-layer representation $\int_{\partial\Omega}\psi(y)P(x,y)ds(y)$ in \eqref{liu11-01-01}, the derived boundary integral equation on $\partial\Omega$ from the Dirichlet boundary condition in \eqref{liu11-01} is of the first kind with respect to the density $\psi$. Consequently, the solvability of the corresponding integral system with respect to the density pair $(\mu,\psi)$ cannot be ensured, which means the Levi function scheme \eqref{liu11-01-01} for solving \eqref{liu11-01} is still questionable.

On the other hand, the domain integral in \eqref{liu11-01-01} involving the density $\mu$ to be  determined is the novel and necessary term of Levi function scheme for solving PDE with variable coefficients. Therefore we need to evaluate the first term in \eqref{liu11-01-01} in an efficient way for obtaining the solution $u$ with satisfactory accuracy. An efficient way to the computation of  this domain integral, as proposed in \cite{Beshley} for $\Omega\subset \mathbb{R}^2$, is to establish the domain transform $\Pi$ to $\Omega$ in terms of known coordinates transform, where $\Pi$ is a standard rectangle domain. Then the first term in \eqref{liu11-01-01} can be discretized in the new coordinate in $\Pi$. However, the coordinates transform, although enables us to compute the integral in a rectangle domain $\Pi$ which is easy to be partitioned, may change the area of each elements in $\Omega$ non-uniformly due to the fact that the Jacobian determinant depends on the position of the elements. This new phenomenon requires some nonuniform partition in $\Pi$ for computing the integral in $\Omega$ with a unform accuracy, otherwise the density function $\mu$ in $\Omega$ will be of different accuracy and consequently contaminates the accuracy of $u$ for \eqref{liu11-01} from the Levi function representation. Moreover, to compute the volume potential accurately, the domain $\Omega$ should be partitioned into elements with very small size, and consequently the evaluations of the unknown $\mu$ at the partition grids lead to the solution of a linear algebra equations of large size, which is of heavy computational cost. To overcome this difficulty, a possible way is to convert the domain integral in $\Omega$ into the boundary integral on $\partial\Omega$ by expanding the density function $\mu$ in terms of some specified base functions, such as the dual reciprocity method (DRM)  \cite{DRM_Nardini} or the radial integration method (RIM) \cite{RIM_Gao}.

Based on the above motivations, we make novel contributions for numerically solving \eqref{liu11-01} under the Levi function framework from three aspects. Firstly, instead of the single layer boundary potential in \eqref{liu11-01-01} for the representation of the solution, we apply the double-layer potential, which is different from those in \cite{Mir} and leads to a linear integral equation of the second kind for the density pair $(\mu,\psi)$.  Secondly, the solvability of this linear system  is established rigorously in two different function spaces, which ensures the applicability of the Levi function scheme for \eqref{liu11-01}. Finally,  after the standard singularity decompositions dealing with the hyper-singularity of the integral term in the linear system coming from the double layer potential, we propose two discrete schemes for computing the volume potential efficiently, namely, the self-adaptive partition scheme (ADS) with explicit quantitative criterion for the non-uniform partition of $\Omega$ and the dual reciprocity method (DRM) leading to a domain discretization-free method.
The proposed two schemes for solving the density function pair provide efficient numerical realizations of solving \eqref{liu11-01-01} in terms of the Levi function framework.
Numerical experiments show that the proposed schemes can yield the solution $u$ on arbitrary internal points with satisfactory accuracy compatible to well-known finite element methods (FEM).

In general, the treatments of the solution of elliptic equation with variable coefficients under the framework of integral equations, as pointed in \cite{Beshley}, are realized in three main ways.  The first scheme is to decompose the variable coefficient into
a constant and a non-constant part, then the term with the variable coefficient
is viewed as a source term, and the direct integral equation approach leads to a boundary-domain integral equation \cite{Ang,Cle,Poz}.  The second scheme
is based on the Green's formula in combination with a Levi function \cite{Jaw, Chk,Cle2}. The third way is what we consider here, namely, the indirect integral equation formulation using potential representation of the solution by Levi functions.

We organize the paper as follows.
In section 2, we establish a boundary integral equation system for the density pair $(\mu,\psi)$ under the Levi function scheme for computing the solution to \eqref{liu11-01}. This new system is different from the cases for the Laplace or the Helmholtz equations with constant coefficients which involve only the surface potentials. The solvability of the coupled system is rigorously proving in two different spaces. Then in section 3, we propose to solve the corresponding linear system with a self-adaptive scheme (ADS) with explicit quantitative criterion for dealing with the smooth integrands, while the singularities in the kernel functions are decomposed from the property of the Levi function, and then can be computed from the well-known quadrature rules.
In section 4, the dual reciprocity method (DRM) which transforms the domain integrals into boundary integrals is applied to compute the volume potential without domain discretization. The choice of interior nodes for constructing the base functions are random and DRM can compute the weak singularity in volume potential explicitly by transforming the volume integrals into boundary integrals. In section 5, numerical experiments of our proposed schemes are implemented for 2-dimensional and 3-dimensional cases, showing the validity by comparing the numerics with finite element methods.

%%%%%%%%%%%%%%%%%%%%%%%%%%%
\section{Integral representation of the solution}

For differential operator $\mathcal{L}\diamond:=-\nabla\cdot(\sigma\nabla\diamond)$ in $\mathbb{R}^d$, we call the function $P(x,y)$ with $x,y\in \mathbb{R}^d$ the Levi function, provided
\begin{equation}\label{liu22-01}
\mathcal{L}_x P(x,y)=\delta(x-y)+R(x,y)
\end{equation}
in the distribution sense,
where $\delta$ is the standard Dirac function, and $R(x,y)$ is a function of weak singularity for $x=y$.
Notice, the Levi function $P(x,y)$ and consequently the reminder $R(x,y)$ are not unique.

Introduce the fundamental solution
\begin{eqnarray*}
\Phi(x,y):=
\begin{cases}
\frac{1}{2\pi}\ln\frac{1}{|x-y|}, &d=2\\
\frac{1}{4\pi}\frac{1}{|x-y|}, &d=3
\end{cases}
\end{eqnarray*}
to the Laplacian operator, i.e., $-\Delta_x \Phi(x,y)=\delta(x-y)$.
It is well known that one of the form of the Levi function pair for $\mathcal{L}$ is
\begin{eqnarray}\label{liu22-3}
(P(x,y),R(x,y))&=&\frac{1}{\sigma(y)}(\Phi(x,y),-\nabla_x\Phi(x,y)\cdot\nabla\sigma(x))
\nonumber\\&=&
\begin{cases}
(\frac{1}{\sigma(y)}\frac{1}{2\pi}\ln\frac{1}{|x-y|}, \frac{1}{\sigma(y)}\frac{(x-y)\cdot \nabla\sigma(x)}{2\pi|x-y|^2}), &d=2\\
(\frac{1}{\sigma(y)}\frac{1}{4\pi|x-y|}, \frac{1}{\sigma(y)} \frac{(x-y)\cdot\nabla\sigma(x)}{4\pi|x-y|^3}), &d=3
\end{cases}
\end{eqnarray}
for $x\not=y$. Note, $R(x,y)$ is of the weak singularity from $|R(x,y)|\le C|x-y|^{-(d-1)}$.

%Taking
%\begin{equation}\label{liu22-02}
%P(x,y):=\frac{-1}{4\pi |x-y|\sigma(y)},\quad x,y\in \mathbb{R}^3, x\not=y,
%\end{equation}
% it follows from
%\begin{eqnarray}\label{liu22-03}
%& &\int_{\mathbb{R}^3}\mathcal{L}_x P(x,y)\varphi(y)dy
%\nonumber\\&=&
%\int_{\mathbb{R}^3}\left(\sigma(x)\Delta_x P(x,y)+\nabla\sigma(x)\cdot\nabla_x P(x,y)\right)\varphi(y)dy
%\nonumber\\&=&
%\int_{\mathbb{R}^3}\left(\sigma(x)\frac{1}{\sigma(y)}\delta(x-y)+
%\nabla\sigma(x)\cdot\frac{x-y}{4\pi|x-y|^3\sigma(y)}\right)\varphi(y)dy
%\nonumber\\&=&
%\varphi(x)+\int_{\mathbb{R}^3}\nabla\sigma(x)\cdot\frac{x-y}{4\pi|x-y|^3\sigma(y)}\varphi(y)dy
%\end{eqnarray}
%for any $\varphi\in C_0^\infty(\mathbb{R}^3)$
%that $P(x,y)$ is a Levi function with the remainder term
%\begin{equation}\label{liu22-04}
%R(x,y):=\frac{\nabla\sigma(x)\cdot(x-y)}{4\pi|x-y|^3\sigma(y)}, \quad x,y\in \mathbb{R}^3, x\not=y.
%\end{equation}

The classical Levi function scheme represents the solution to (\ref{liu11-01}) by \eqref{liu11-01-01},
%\begin{equation}\label{liu11-05}
%u(x):=\int_\Omega\mu(y)P(x,y)dy+\int_{\partial\Omega}\psi(y)P(x,y)ds(y), \quad x\in\Omega,
%\end{equation}
which can be considered as a combination of the volume potential in $\Omega$ and the
single-layer potential in $\partial\Omega$. Then by inserting this representation into both the equation and the boundary condition  in \eqref{liu11-01}, a linear system with respect to the density pair $(\mu|_\Omega,\psi|_{\partial\Omega})$ is derived. However, the corresponding integral equation from the boundary condition is the first kind, and consequently, the solvability of the derived integral system, i.e., the reasonability of the standard   Levi function scheme, cannot be ensured theoretically, although this scheme has been realized efficiently. To overcome such a shortcoming, we propose to modify the representation \eqref{liu11-01-01} as
\begin{equation}\label{liu22-05}
u(x)=\int_{\Omega_{\epsilon_0}}\mu(y)P(x,y)dy+\int_{\partial\Omega}\frac{\psi(y)}{\sigma(y)}\partial_{\nu(y)}\Phi(x,y)ds(y), \quad x\in\Omega
\end{equation}
for some density functions $\mu\in C(\overline\Omega_{\epsilon_0})$ and $\psi\in C(\partial\Omega)$, where $\Omega_{\epsilon_0}\subset\subset\Omega$ is arbitrary specified domain. Notice, the volume potential \eqref{liu22-05} is defined in the interior domain $\Omega_{\epsilon_0}$ instead of the whole domain $\Omega$, which is crucial for us to avoid the hyper-singularity of double-layer potential in \eqref{liu22-05} in the PDE in $\Omega$.
% It is easy to see from $P(x,y)=\frac{1}{\sigma(y)}\Phi(x,y)$ that
% the second term in the right-hand side of \eqref{liu22-05} is the summation of double layer potential with density $\frac{\psi(y)}{\sigma(y)}$ and single layer potential with density $\partial_{\nu(y)}(\frac{1}{\sigma(y)})\psi(y)$.
By the $C^1(\mathbb{R}^d)$ regularity of volume potential (Theorem 8.1, \cite{Colton1}) and the jump relation of the double potential potential as well as the continuity of the single potential (Theorem 2.13, Theorem 2.12, \cite{Colton}), if $u(x)$ represented by (\ref{liu22-05}) solves (\ref{liu11-01}), $(\mu,\psi)\in C(\overline\Omega_{\epsilon_0})\times C(\partial\Omega)$ must solve
\begin{eqnarray}\label{liu22-06}
\begin{cases}
\mu(x)+\int_{\Omega_{\epsilon_0}}\mu(y)R(x,y)dy+\int_{\partial\Omega}\frac{\psi(y)}{\sigma(y)}\partial_{\nu(y)}( \sigma(y)R(x,y))ds(y)=F(x), &x\in\Omega_{\epsilon_0}\\
\int_{\Omega_{\epsilon_0}}\mu(y)P(x,y)dy+\int_{\partial\Omega}\frac{\psi(y)}{\sigma(y)}\partial_{\nu(y)}\Phi(x,y)ds(y)-
\frac{1}{2}\frac{\psi(x)}{\sigma(x)}
=f(x), &x\in\partial\Omega.
\end{cases}
\end{eqnarray}

In case of $(\sigma(x), F(x))\equiv (1, 0)$, it follows from the first equation that $\mu(x)\equiv 0$ and the second equation in \eqref{liu22-06} is just the boundary integral equation of the second kind with respect to the density function $\psi(x)$ for  solving the interior Dirichlet probelm for the Laplace
equation in $\Omega$, noticing that $\Omega_{\epsilon_0}$ is the arbitrary domain in $\Omega$.

In the following we only consider 2-dimensional space. However, all the arguments work for the case $d=3$ with obvious modifications on the fundamental solution.

Define $\tilde\mu(x):=\frac{\mu(x)}{\sigma(x)}$ for $x\in\overline\Omega_{\epsilon_0}$ and $\tilde\psi(x):=\frac{\psi(x)}{\sigma(x)}$ for $x\in\partial\Omega$, and the operators
\begin{align*}
\mathbb{K}_{11}^{\Omega_{\epsilon_0}\to\overline\Omega_{\epsilon_0}}[\tilde\mu](x):
&=\int_{\Omega_{\epsilon_0}} \nabla_x\Phi(x,y)\cdot\nabla\ln\sigma(x)
\tilde\mu(y)dy,&x\in\overline\Omega_{\epsilon_0},\\
\mathbb{K}_{12}^{\partial\Omega\to\overline\Omega_{\epsilon_0}}[\tilde\psi](x):
&=-\int_{\partial\Omega}
\frac{\partial}{\partial\nu(y)}\left(\nabla_x\Phi(x,y)\cdot\nabla\ln\sigma(x)\right)\tilde\psi(y)ds(y),&x\in\overline\Omega_{\epsilon_0},\\
\mathbb{K}_{21}^{\Omega_{\epsilon_0}\to\partial\Omega}[\tilde\mu](x):
&=-2\int_{\Omega_{\epsilon_0}}\Phi(x,y)
\tilde\mu(y)dy,  &x\in\partial\Omega,\\
\mathbb{K}_{22}^{\partial\Omega\to\partial\Omega}[\tilde\psi](x):
&=2\int_{\partial\Omega}
\frac{\partial\Phi(x,y)}{\partial\nu(y)}
\tilde\psi(y)ds(y),      &x\in\partial\Omega.
\end{align*}
%It is easy to see the identity
%$$\nabla_x\mathbb{K}_{22}^{\partial\Omega\to\Omega_{\epsilon_0}}[\tilde\psi](x)\cdot\nabla\ln\sigma(x)\equiv -2\mathbb{K}_{12}^{\partial\Omega\to\Omega_{\epsilon_0}}[\tilde\psi](x),\quad x\in\Omega_{\epsilon_0}.$$
Then we can rewrite (\ref{liu22-06}) in the operator form
\begin{eqnarray}\label{liu22-07}
\begin{cases}
(\mathbb{I}-\mathbb{K}_{11}^{\Omega_{\epsilon_0}\to\Omega_{\epsilon_0}})[\tilde\mu](z)+
\mathbb{K}_{12}^{\partial\Omega\to\Omega_{\epsilon_0}}[\tilde\psi](z)=F(z)\sigma^{-1}(z), \quad z\in \Omega_{\epsilon_0}\\
\mathbb{K}_{21}^{\Omega_{\epsilon_0}\to\partial\Omega}[\tilde\mu](x)+(\mathbb{I}-
\mathbb{K}_{22}^{\partial\Omega\to\partial\Omega})[\tilde\psi](x)=-2f(x),\qquad \quad x\in\partial\Omega.
\end{cases}
\end{eqnarray}
%\yellowB{and the modified solution} representation \eqref{liu22-05} can be rewritten without $\sigma(x)$ as:
%\begin{equation}\label{liu22}
%u(x)=\int_{\Omega_{\epsilon_0}}\tilde\mu(y)\Phi(x,y)dy+\int_{\partial\Omega}
%\tilde \psi(y) \partial_{\nu(y)}\Phi(x,y)ds(y), \quad x\in\Omega
%\end{equation}

We firstly establish the solvability of \eqref{liu22-07} for the density function pair in continuous function space for known $(\sigma, F, f)$ by the following result.

\begin{them}
Assume that the specified conductivity in \eqref{liu11-01} meets $0<\sigma_0\le \sigma\in C^1(\overline\Omega)$ and
\begin{equation}\label{liu22-09}
0\le \|\nabla\ln\sigma\|_{C(\overline\Omega)}\ll 1.
\end{equation}
Then for any  $(F,f)\in C(\overline\Omega)\times C(\partial\Omega)$,
there exists a unique solution $(\tilde\mu,\tilde\psi)\in C(\overline\Omega_{\epsilon_0})\times C(\partial\Omega)$ to \eqref{liu22-07},
where $\Omega_{\epsilon_0}:=\{x\in\Omega:\; \hbox{dist}(x,\partial\Omega)\ge\epsilon_0\}$ with any small $\epsilon_0>0$. Moreover,
$u(x,t)$ defined by \eqref{liu22-05} meets the PDE in $\Omega_{\epsilon_0}$ and the the boundary condition in $\partial\Omega$.
\end{them}

In case of $\nabla\sigma|_{\Omega}\equiv 0$ with $F=0$, i.e., $\sigma$ is a constant in $\Omega$, we can solve the direct problem using single layer potential directly, i.e., we can specify $\mu(x)\equiv 0$ in \eqref{liu22-05}.

\begin{proof}
We prove this result by using the Fredholm Alternative and the contractive mapping recursively.

%Without loss of generality, we assume that
%\begin{equation}\label{liu11-03-01}
%\norm{\nabla\sigma}_{C^1(\overline\Omega)}\ll 1
%\end{equation}
%for $\sigma\in \mathcal{A}_{\epsilon_0}$.
%Otherwise, we can introduce a small constant $\epsilon>0$ and replace $\sigma$ in \eqref{liu11-01} by $\epsilon\sigma$, keeping $\epsilon\sigma\in \mathcal{A}_{\epsilon_0}$ and the solution to \eqref{liu11-01} unchanged.

We rewrite \eqref{liu22-07} equivalently in the vector form
\begin{eqnarray}\label{liu22-07-01}
(\mathbb{I}-\mathbb{K}_\sigma)
\left[
\begin{array}{c}
\tilde\mu\\
\tilde\psi
\end{array}
\right]\left(
\begin{array}{c}
z\\
x
\end{array}
\right)=
\left(
\begin{array}{c}
F(z)\sigma^{-1}(z)\\
-2f(x)
\end{array}
\right), \quad (z,x)\in \Omega_{\epsilon_0}\times\partial\Omega
\end{eqnarray}
with the matrix operator
\begin{eqnarray*}
\mathbb{K}_\sigma:=\left(
\begin{array}{cc}
\mathbb{K}_{11}^{\Omega_{\epsilon_0}\to\Omega_{\epsilon_0}} &-\mathbb{K}_{12}^{\partial\Omega\to\Omega_{\epsilon_0}}\\
-\mathbb{K}_{21}^{\Omega_{\epsilon_0}\to\partial\Omega} &\mathbb{K}_{22}^{\partial\Omega\to\partial\Omega}
\end{array}
\right).
\end{eqnarray*}
By the expressions of four elements of the matrix $\mathbb{K}_\sigma$,  $\mathbb{K}_\sigma$ is compact from $C(\overline\Omega_{\epsilon_0})\times C(\partial\Omega)$ to itself. Notice, both $\mathbb{K}_{21}^{\Omega_{\epsilon_0}\to\partial\Omega}$ and  $\mathbb{K}_{12}^{\partial\Omega\to\Omega_{\epsilon_0}}$ are of continuous kernel due to $\overline\Omega_{\epsilon_0}\subset\subset\Omega$. Therefore, to prove Theorem 2.1, it is enough to prove the uniqueness of the solution to \eqref{liu22-07-01} by the Fredholm Alternative for the linear compact operator equation of the second kind.

The uniqueness of solution to \eqref{liu22-07-01} (equivalently to \eqref{liu22-07}) is established by several steps. Assume that $(\tilde\mu,\tilde\psi)\in C(\overline\Omega_{\epsilon_0})\times C(\partial\Omega)$ meets
\begin{eqnarray}\label{liu22-03}
\begin{cases}
(\mathbb{I}-\mathbb{K}_{11}^{\Omega_{\epsilon_0}\to\Omega_{\epsilon_0}})[\tilde\mu](z)+
\mathbb{K}_{12}^{\partial\Omega\to\Omega_{\epsilon_0}}[\tilde\psi](z)=0, &z\in \Omega_{\epsilon_0},\\
\mathbb{K}_{21}^{\Omega_{\epsilon_0}\to\partial\Omega}[\tilde\mu](x)+(\mathbb{I}-
\mathbb{K}_{22}^{\partial\Omega\to\partial\Omega})[\tilde\psi](x)=0, &x\in\partial\Omega.
\end{cases}
\end{eqnarray}
We need to prove $(\tilde\mu,\tilde\psi)=(0,0)$.

{\bf Step 1:} We prove that the equation
$$(\mathbb{I}-\mathbb{K}_{11}^{\Omega_{\epsilon_0}\to\Omega_{\epsilon_0}})[\tilde\mu](z)=g(z), \quad z\in\Omega_{\epsilon_0}$$
is uniquely solvable for any $g\in C(\Omega_{\epsilon_0})$.

Since $\mathbb{K}_{11}^{\Omega_{\epsilon_0}\to\Omega_{\epsilon_0}}$ is compact from $C(\overline\Omega_{\epsilon_0})$ to itself, it is enough to prove
$$(\mathbb{I}-\mathbb{K}_{11}^{\Omega_{\epsilon_0}\to\Omega_{\epsilon_0}})[\tilde\mu](z)=0$$
has only the trivial solution by the Fredholm Alternative. To this end, let us prove that the map $\mathbb{K}_{11}^{\Omega_{\epsilon_0}\to\Omega_{\epsilon_0}}$ is contractive. By the expression of $\mathbb{K}_{11}^{\Omega_{\epsilon_0}\to\Omega_{\epsilon_0}}$, we know
\begin{equation}\label{liu22-04}
\|\mathbb{K}_{11}^{\Omega_{\epsilon_0}\to\Omega_{\epsilon_0}}
[\tilde\mu]\|_{C(\overline\Omega_{\epsilon_0})}\le
\frac{\norm{\nabla\ln\sigma}_{C(\overline\Omega)}}{2\pi}
\norm{\int_{\Omega_{\epsilon_0}}\left|\nabla_y\ln|y-\cdot|\;\right|dy}_{C(\overline
\Omega_{\epsilon_0})}
\norm{\tilde\mu}_{C(\overline\Omega_{\epsilon_0})}.
\end{equation}
Since $\Omega$ is bounded and $\Omega_{\epsilon_0}\subset\subset\Omega$, $\mathbb{K}_{11}^{\Omega_{\epsilon_0}\to\Omega_{\epsilon_0}}$ is contractive from \eqref{liu22-04} and \eqref{liu22-09}.

{\bf Step 2:} Derive a linear equation for $\tilde \psi$ in fixed point form by \eqref{liu22-03}.

By {\bf Step 1}, we can solve $\tilde\mu$ explicitly from the first equation of \eqref{liu22-03}, that is,
\begin{equation}\label{liu44-08}
\tilde\mu(y)=-(\mathbb{I}-\mathbb{K}_{11}^{\Omega_{\epsilon_0}\to\Omega_{\epsilon_0}})^{-1}
\mathbb{K}_{12}^{\partial\Omega\to\Omega_{\epsilon_0}}[\tilde\psi](y), \quad y\in \Omega_{\epsilon_0}.
\end{equation}
 Inserting this representation into the second equation of \eqref{liu22-03}, we have
\begin{equation}\label{liu33-08}
(\mathbb{I}-
\mathbb{K}_{22}^{\partial\Omega\to\partial\Omega})[\tilde\psi](x)-
\mathbb{K}_{21}^{\Omega_{\epsilon_0}\to\partial\Omega}
(\mathbb{I}-\mathbb{K}_{11}^{\Omega_{\epsilon_0}\to\Omega_{\epsilon_0}})^{-1}
\mathbb{K}_{12}^{\partial\Omega\to\Omega_{\epsilon_0}}[\tilde\psi](x)=0,\quad x\in\partial\Omega.
\end{equation}
Since $\mathbb{K}_{22}^{\partial\Omega\to\partial\Omega}$
% has the decomposition
% $$\mathbb{K}_{22}^{\partial\Omega\to\partial\Omega}[\tilde\psi](x)\equiv
% \mathbb{K}_{22,1}^{\partial\Omega\to\partial\Omega}[\tilde\psi](x)+
% \mathbb{K}_{22,2}^{\partial\Omega\to\partial\Omega}[\tilde\psi](x),\quad x\in\partial\Omega
% $$
% with
% $$\mathbb{K}_{22,1}^{\partial\Omega\to\partial\Omega}[\tilde\psi](x):=2\int_{\partial\Omega}
% \frac{\partial\Phi(x,y)}{\partial\nu(y)}\tilde\psi(y)ds(y), $$
% $$\mathbb{K}_{22,2}^{\partial\Omega\to\partial\Omega}[\tilde\psi](x):=
% 2\int_{\partial\Omega}\Phi(x,y)\frac{\partial}{\partial\nu(y)}\left(
% \ln\frac{1}{\sigma(y)}\right)
% \tilde\psi(y)ds(y),$$
% which are the double- and single-layer
is double-layer potential with respect to the continuous density functions
$\tilde\psi(y)$,
% \frac{\partial}{\partial\nu(y)}\left(
% \ln\frac{1}{\sigma(y)}\right)
% \tilde\psi(y)$, respectively.
it is compact from $C(\partial\Omega)$ to itself (Theorem 2.30 and Theorem 2.31 in \cite{Colton}). So \eqref{liu33-08} leads to
\begin{equation}\label{liu33-09}
(\mathbb{I}-
\mathbb{K}_{22}^{\partial\Omega\to\partial\Omega})[\tilde\psi](x)-
\mathbb{K}^{\partial\Omega\to\partial\Omega}[\tilde\psi](x)=0, \quad x\in\partial\Omega,
\end{equation}
with the operator
$$
\mathbb{K}^{\partial\Omega\to\partial\Omega}[\tilde\psi](x):=
\mathbb{K}_{21}^{\Omega_{\epsilon_0}\to\partial\Omega}
(\mathbb{I}-\mathbb{K}_{11}^{\Omega_{\epsilon_0}\to\Omega_{\epsilon_0}})^{-1}
\mathbb{K}_{12}^{\partial\Omega\to\Omega_{\epsilon_0}}[\tilde\psi](x),\quad x\in\partial\Omega.
$$

% $$
% \mathbb{K}^{\partial\Omega\to\partial\Omega}[\tilde\psi](x):=
% \mathbb{K}_{22,2}^{\partial\Omega\to\partial\Omega}[\tilde\psi](x)+
% \mathbb{K}_{21}^{\Omega_{\epsilon_0}\to\partial\Omega}
% (\mathbb{I}-\mathbb{K}_{11}^{\Omega_{\epsilon_0}\to\Omega_{\epsilon_0}})^{-1}
% \mathbb{K}_{12}^{\partial\Omega\to\Omega_{\epsilon_0}}[\tilde\psi](x),\quad x\in\partial\Omega.
% $$
Notice that the operator $\mathbb{K}_{22}^{\partial\Omega\to\partial\Omega}$ is the double-layer potential for the Laplace operator, the unique solvability of the interior Dirichlet problem for the Laplace equation together with the double-layer potential representation of the solution ensures $\mathbb{I}-
\mathbb{K}_{22}^{\partial\Omega\to\partial\Omega}$ is invertible from $C(\partial\Omega)$ to itself. So  the equation \eqref{liu33-09} has the
fixed-point form
\begin{equation}\label{liu44-01}
\tilde\psi(x)-
(\mathbb{I}-
\mathbb{K}_{22}^{\partial\Omega\to\partial\Omega})^{-1}
\mathbb{K}^{\partial\Omega\to\partial\Omega}[\tilde\psi](x)=0, \quad x\in\partial\Omega.
\end{equation}

{\bf Step 3:} We prove that the equation \eqref{liu22-03} has only the trivial solution.

% Since $\mathbb{K}_{22,2}^{\partial\Omega\to\partial\Omega}$ is the single layer potential, it is obvious that
% \begin{equation}\label{liu11-06}
% \|\mathbb{K}_{22,2}^{\partial\Omega\to\partial\Omega}\|_
% {C(\partial\Omega)\to C(\partial\Omega)}\le C\|\partial_\nu\ln\sigma\|_{C(\partial\Omega)}
% \le C\|\nabla\ln\sigma\|_{C(\overline\Omega)}
% \end{equation}
% for some constant $C=C(\sigma_0, \|\ln\sigma\|_{C(\Omega)})$. On the other hand,
It is obvious that
\begin{equation}\label{liu55-01}
\norm{\mathbb{K}_{21}^{\Omega_{\epsilon_0}\to\partial\Omega}}
_{C(\overline\Omega_{\epsilon_0})\to C(\partial\Omega)}\le C
\end{equation}
due to $\Omega_{\epsilon_0}\subset\subset\Omega$.
%$$|\mathbb{K}_{21}^{\Omega\to\partial\Omega}[\tilde\mu](x)|\le \frac{1}{\pi}\|\tilde \mu\|_{C(\overline\Omega)}\int_\Omega|\ln|x-y||dy\le \frac{1}{\pi}\|\tilde \mu\|_{C(\overline\Omega)}\int_{B(0,R_0)}|\ln|z||dz$$
%for all $x\in\partial\Omega$.
As for the operator $\mathbb{K}_{12}^{\partial\Omega\to\Omega_{\epsilon_0}}$,
the straightforward computations yield
\begin{equation}\label{liu55-02}
\mathbb{K}_{12}^{\partial\Omega\to\Omega_{\epsilon_0}}[\tilde\psi](x)\equiv
\nabla_x\left(\mathbb{K}_{12,0}^{\partial\Omega\to\Omega_{\epsilon_0}}[\tilde\psi](x)\right)\cdot \nabla\ln\sigma(x),
\end{equation}
where
$$\mathbb{K}_{12,0}^{\partial\Omega\to\Omega_{\epsilon_0}}[\tilde\psi](x):=-
\int_{\partial\Omega}\frac{\partial\Phi(x,y)}{\partial\nu(y)}
\tilde\psi(y)ds(y),\quad x\in \Omega_{\epsilon_0},
$$
% $$\mathbb{K}_{12,1}^{\partial\Omega\to\Omega_{\epsilon_0}}[\tilde\psi](x):=-
% \int_{\partial\Omega}\Phi(x,y)
% \frac{\partial}{\partial\nu(y)}\left(\ln\frac{1}{\sigma(y)}\right)\tilde\psi(y)ds(y),
% \quad x\in \Omega_{\epsilon_0},
% $$
which is
the double-layer potential with density function $\tilde\psi$. Since both kernels are smooth, we have from \eqref{liu55-02} that
\begin{eqnarray*}
\norm{\mathbb{K}_{12}^{\partial\Omega\to\Omega_{\epsilon_0}}[\tilde\psi]}_{C(\Omega_{\epsilon_0})}
&=&
\norm{\nabla_x(\mathbb{K}_{12,0}^{\partial\Omega\to\Omega_{\epsilon_0}}[\tilde\psi])\cdot
\nabla\ln\sigma}_{C(\overline\Omega_{\epsilon_0})}
\nonumber\\&\le&
C_{\epsilon_0}\|\tilde\psi\|_{C(\partial\Omega)}
\norm{\nabla\ln\sigma}_{C(\overline\Omega_{\epsilon_0})}
\le
C_{\epsilon_0}
\norm{\nabla\ln\sigma}_{C(\overline\Omega)}\|\tilde\psi\|_{C(\partial\Omega)},
\end{eqnarray*}
which yields
\begin{equation}\label{liu55-03}
\norm{\mathbb{K}_{12}^{\partial\Omega\to\Omega_{\epsilon_0}}}_{C(\partial\Omega)\to C(\overline\Omega_{\epsilon_0})}\le C_{\epsilon_0}\norm{\nabla\ln\sigma}_{C(\overline\Omega)}.
\end{equation}
%Using the jump relations of the derivatives of single- and double-layer potentials, it follows from \eqref{liu55-04} that
%\begin{eqnarray}
%& &\lim_{\Omega\ni x_0\to x\in\partial\Omega}\mathbb{K}_{12}^{\partial\Omega\to\Omega}[\tilde\psi](x_0)
%\nonumber\\
%&=&
%\left(\frac{\partial}{\partial\gamma(x)}\left[\mathbb{K}_{12,0}^{\partial\Omega\to\Omega}[\tilde\psi](x)-
%\mathbb{K}_{12,1}^{\partial\Omega\to\Omega}[\tilde\psi](x)\right]-
%\frac{1}{2}\frac{\partial}{\partial\nu(x)}\left[\ln\frac{1}{\sigma(x)}\right]
%\tilde\psi(x)\right)
%\frac{\partial\sigma(x)}{\partial\gamma(x)}+
%\nonumber\\& &
%\frac{\partial}{\partial\gamma^\perp(x)}
%\left[\mathbb{K}_{12,0}^{\partial\Omega\to\Omega}[\tilde\psi](x)-
%\mathbb{K}_{12,1}^{\partial\Omega\to\Omega}[\tilde\psi](x)\right]
%\frac{\partial\sigma(x)}{\partial\gamma^\perp(x)},
%\end{eqnarray}
%that is, the limitation of $\mathbb{K}_{12}^{\partial\Omega\to\Omega}[\tilde\psi](x)$ as $\Omega\ni x\to\partial\Omega$ exists, which leads to
%$$\|\mathbb{K}_{12}^{\partial\Omega\to\overline\Omega}[\tilde\psi](\cdot)\|_{C(\overline\Omega)}\le C_0\norm{\nabla\sigma}_{C(\overline\Omega)}\norm{\tilde\psi}_{C(\partial\Omega)}.$$
%Therefore, we have
%\begin{equation}
%\norm{\mathbb{K}_{12}^{\partial\Omega\to\overline\Omega}}_{\infty}\le C_0\norm{\nabla\sigma}_{C(\overline\Omega)}.
%\end{equation}

Since both $(\mathbb{I}-
\mathbb{K}_{22}^{\partial\Omega\to\partial\Omega})^{-1}$ and $(\mathbb{I}-\mathbb{K}_{11}^{\Omega\to\Omega})^{-1}$ are bounded, $(\mathbb{I}-
\mathbb{K}_{22}^{\partial\Omega\to\partial\Omega})^{-1}
\mathbb{K}^{\partial\Omega\to\partial\Omega}$ is contractive under \eqref{liu22-04}, by combining \eqref{liu44-01},
% \eqref{liu11-06},
\eqref{liu55-01}  and \eqref{liu55-03} together. So we have $\tilde\psi(x)=0$
from \eqref{liu44-01}. Finally we obtain $\tilde\mu(x)=0$
from \eqref{liu44-08}.

The assertion that $u(x)$ meets PDE in $\Omega_{\epsilon_0}$ comes from \eqref{liu22-06}.
The proof is complete.
\end{proof}

Since $\Omega_{\epsilon_0}$ can approximate $\Omega$ up to any accuracy by taking $\epsilon_0>0$ small enough, we finally get the representation of $u(x)$ in $\Omega$.

\begin{rem}
The importance of this result is that we construct an implementable
Levi-function-based scheme for solving the solution $u(x)$ to \eqref{liu11-01} in any interior domain $\Omega_{\epsilon_0}$ in terms of \eqref{liu22-05}, by proving the solvability of the density pair $(\tilde\mu|_{\Omega_{\epsilon_0}},\tilde\psi|_{\partial\Omega})$ from \eqref{liu22-07} rigorously. The restriction that $u$ satisfies the PDE in $\Omega_{\epsilon_0}$, namely, the requirement that, the first equation of \eqref{liu22-07} hold only in $\Omega_{\epsilon_0}$ instead of $\Omega$, enables us to keep the compact property of the operator $\mathbb{K}_{12}^{\partial\Omega\to\Omega_{\epsilon_0}}$
from its $C^1$-smoothness
% of $$\mathbb{K}_{12,0}^{\partial\Omega\to\Omega_{\epsilon_0}}[\tilde\psi](x),
% \mathbb{K}_{12,1}^{\partial\Omega\to\Omega_{\epsilon_0}}[\tilde\psi](x)$$
for $\tilde\psi\in C(\partial\Omega)$,
while the second equation of \eqref{liu22-07} is a linear integral equation of the second kind.
However, for general $\sigma\in C^1(\overline\Omega)$, if we consider the operators $\mathbb{K}_{12}^{\partial\Omega\to\Omega}$, which is not compact for $\tilde\psi\in C(\partial\Omega)$,  the higher regularity $\tilde\psi\in C^{1,\alpha}(\partial\Omega)$ with $\alpha\in (0,1)$ should be imposed (Theorem 2.23, \cite{Colton}).
\end{rem}

%We rewrite \eqref{liu22-07} in the operator form
%\begin{eqnarray}\label{liu22-08}
%\mathcal{K}_{\sigma}
%\left(
%\begin{array}{c}
%\tilde\mu\\
%\tilde\psi
%\end{array}
%\right):=
%\left(
%\begin{array}{cc}
%\mathbb{I}-\mathbb{K}_{11}^{\Omega\to\Omega} &\mathbb{K}_{12}^{\partial\Omega\to\Omega}\\
%\mathbb{K}_{21}^{\Omega\to\partial\Omega} &\mathbb{I}-
%\mathbb{K}_{22}^{\partial\Omega\to\partial\Omega}
%\end{array}
%\right)
%\left(
%\begin{array}{c}
%\tilde\mu\\
%\tilde\psi
%\end{array}
%\right)
%=
%\left(
%\begin{array}{c}
%0\\
%f
%\end{array}
%\right).
%\end{eqnarray}
Theorem 2.1 establishes the solvability of \eqref{liu22-07-01} in continuous function space for $\sigma$ with slight restrictions, which ensures the reasonability of the Levi function representation
\eqref{liu22-05} for solving \eqref{liu11-01}, by firstly restricting the solution representation in $\Omega_{\epsilon_0}$ and then taking $\epsilon_0>0$ arbitrary small. However, it is well-known that \eqref{liu11-01} is solvable for all $0<\sigma\in L^\infty(\Omega)$. So, an interesting problem is to relax the restriction on $\sigma$ in Theorem 2.1, removing the {\it a-priori} smallness assumption on $\|\nabla\ln\sigma\|_{C(\overline\Omega)}$.  We state this result in $L^2$ space as follows.

For any $(F,f)\in L^2(\Omega)\times L^2(\partial\Omega)$, define the density pair
$(\mu,\psi)\in L^2(\Omega)\times L^2(\partial\Omega)$ from the solution to
\begin{eqnarray}\label{liu22-06-0009}
\begin{cases}
\mu(x)+\int_{\Omega}\mu(y)R(x,y)dy+\int_{\partial\Omega}\frac{\psi(y)}{\sigma(y)}\partial_{\nu(y)} \sigma(y)R(x,y)ds(y)=F(x), &x\in\Omega\\
\int_{\Omega}\mu(y)P(x,y)dy+\int_{\partial\Omega}\psi(y)\partial_{\nu(y)}P(x,y)ds(y)-
\frac{1}{2}\frac{\psi(x)}{\sigma(x)}
=f(x), &x\in\partial\Omega.
\end{cases}
\end{eqnarray}

\begin{them}
There exists a unique solution $(\mu,\psi)\in L^2(\Omega)\times L^2(\partial\Omega)$ to \eqref{liu22-06-0009}. Moreover, the function
\begin{equation}\label{liu44-04-01}
u(x):=\int_\Omega\mu(y)P(x,y)dy+\int_{\partial\Omega}\frac{\psi(y)}{\sigma(y)}\partial_{\nu(y)}\Phi(x,y)ds(y), \quad x\in \mathbb{R}^2
\end{equation}
solves the direct problem \eqref{liu11-01}.
\end{them}

\begin{proof}
We firstly prove the uniqueness of the solution to \eqref{liu22-06-0009}.
It is easy to verify from  the jump relations of the surface potentials that  $u\in H^2(\Omega)$ given by \eqref{liu44-04-01} with density pair $(\mu,\psi)$ satisfying \eqref{liu22-06-0009} for $F=f=0$ solves
\begin{eqnarray*}
\begin{cases}
-\nabla\cdot(\sigma\nabla u)=0, &x\in\Omega\\
u(x)=0, &x\in\partial\Omega,
\end{cases}
\end{eqnarray*}
which yields
\begin{eqnarray}\label{liu44-02-01}
u(x)=\int_\Omega\Phi(x,y)\frac{\mu(y)}{\sigma(y)}dy+
% \nonumber\\& &
\int_{\partial\Omega}\frac{\partial\Phi(x,y)}{\partial\nu(y)}\frac{\psi(y)}{\sigma(y)}ds(y)
% +\int_{\partial\Omega}\Phi(x,y) \psi(y)\partial_{\nu(y)}\frac{1}{\sigma(y)}ds(y)
\equiv 0 \quad\quad
\end{eqnarray}
for $x\in\Omega$. However, $u(x)$ representing by \eqref{liu44-04-01} also satisfies the exterior problem
\begin{eqnarray*}
\begin{cases}
\Delta u=0, &x\in\mathbb{R}^2\setminus\overline\Omega\\
u(x)=0, &x\in\partial\Omega
\end{cases}
\end{eqnarray*}
with zero asymptotic behavior as $|x|\to\infty$.
Consequently we also have
\begin{eqnarray}\label{liu44-03-01}
u(x)=\int_\Omega\Phi(x,y)\frac{\mu(y)}{\sigma(y)}dy+
% \nonumber\\& &
\int_{\partial\Omega}\frac{\partial\Phi(x,y)}{\partial\nu(y)}\frac{\psi(y)}{\sigma(y)}ds(y)
% +\int_{\partial\Omega}\Phi(x,y) \psi(y)\partial_{\nu(y)}\frac{1}{\sigma(y)}ds(y)
\equiv 0 \quad\quad
\end{eqnarray}
for $x\in\mathbb{R}^2\setminus\overline\Omega$. Taking $x\to\partial\Omega$ from $\Omega$ and $\mathbb{R}^2\setminus\overline\Omega$ in \eqref{liu44-02-01} and \eqref{liu44-03-01} respectively, the continuity of single layer potential and the volume potential on $\partial\Omega$ together with the jump relation of double layer potential on $\partial\Omega$ yields
\begin{eqnarray*}
\int_\Omega\Phi(x,y)\frac{\mu(y)}{\sigma(y)}dy+
% \nonumber\\& &
\int_{\partial\Omega}\frac{\partial\Phi(x,y)}{\partial\nu(y)}\frac{\psi(y)}{\sigma(y)}ds(y)
\pm\frac{1}{2}\frac{\psi(x)}{\sigma(x)}
% +\int_{\partial\Omega}\Phi(x,y) \psi(y)\partial_{\nu(y)}\frac{1}{\sigma(y)}ds(y)
\equiv 0, \quad x\in\partial\Omega,
\end{eqnarray*}
%\begin{eqnarray*}
%& &\int_\Omega\Phi(x,y)\frac{\mu(y)}{\sigma(y)}dy+
%\nonumber\\& &
%\int_{\partial\Omega}\frac{\partial\Phi(x,y)}{\partial\nu(y)}\frac{\psi(y)}{\sigma(y)}ds(y)
%-\frac{1}{2}\frac{\psi(x)}{\sigma(x)}+\int_{\partial\Omega}\Phi(x,y) \psi(y)\partial_{\nu(y)}\frac{1}{\sigma(y)}ds(y)\equiv 0, x\in\partial\Omega,
%\end{eqnarray*}
which generate $\frac{\psi(x)}{\sigma(x)}\equiv 0$ in $\partial\Omega$ by subtracting these two identities.
Then \eqref{liu22-06-0009} for $F=f=0$ leads to
\begin{eqnarray}
(\mathbb{I}-\mathbb{K}_{11}^{\Omega\to\Omega})[\frac{\mu}{\sigma}](x)=0, \quad x\in\Omega\\
\mathbb{K}_{21}^{\Omega\to\partial\Omega}[\frac{\mu}{\sigma}](x)=0, \quad x\in\partial\Omega,
\end{eqnarray}
where $(\mathbb{K}_{11}^{\Omega\to\Omega}, \mathbb{K}_{21}^{\Omega\to\partial\Omega})$ is $(\mathbb{K}_{11}^{\Omega_{\epsilon_0}\to\Omega_{\epsilon_0}}, \mathbb{K}_{21}^{\Omega_{\epsilon_0}\to\partial\Omega})$ with $\Omega_{\epsilon_0}$ replaced by $\Omega$.

Define
$w(x)=\int_\Omega\Phi(x,y)\frac{\mu(y)}{\sigma(y)}dy,\;x\in\overline\Omega$
and noticing $-\Delta w=\frac{\mu(x)}{\sigma(x)}$ in $\Omega$,
then the above relations yield
\begin{eqnarray}
-\frac{1}{\sigma(x)}\nabla\cdot(\sigma(x)\nabla w(x))\equiv-\Delta w(x)-\nabla w(x)\cdot\nabla\ln\sigma(x)=0, \quad x\in\Omega\\
w(x)=0, \quad x\in\partial\Omega,
\end{eqnarray}
which leads to $w(x)=0$ in $\Omega$ and then $\mu(x)=0$ in $\Omega$.

Noticing the fact that $\mathbb{K}_{11}^{\Omega\to\Omega}$ is also compact from $L^2(\Omega)$ to $L^2(\Omega)$ (Theorem 8.2 in \cite{Colton1}), \eqref{liu22-06-0009} is a linear operator system of the second kind with compact operator $\mathbb{K}_\sigma$, so the Fredholm Alternative leads to the existence of the solution to \eqref{liu22-06-0009} for any $(F,f)\in L^2(\Omega)\times L^2(\partial\Omega)$.
The proof is complete.
\end{proof}

For known $\sigma(x)$, we can firstly solve the density function $(\tilde\mu,\tilde\psi)$ from the linear integral system \eqref{liu22-07} which is well-posed in terms of Theorem 2.1 and Theorem 2.3, and then determine the solution $u$ to \eqref{liu11-01}
using the representation \eqref{liu22-05}. Since the operators in \eqref{liu22-07} are defined by integrals with singular kernels and also the domain integral, we need to establish the discrete version \eqref{liu22-07}, dealing with the singularities and constructing an efficient scheme for computing the integrals with smooth integrands.

To evaluate the volume potentials numerically, we propose two schemes in the following two sections to handle domain integral in \eqref{liu22-07}. One technique (ADS) is a modification of direct parameterization proposed in \cite{Beshley} for the linear system via adding a self-adaptive partition scheme with explicit quantitative criterion to discrete domain $\Omega$. The other one is based on dual reciprocity method (DRM) \cite{DRM_book} which transforms the domain integrals into equivalent boundary integrals without domain discretization.

\section{Self-adaptive discretization for the volume potential}

The standard technique for dealing with the volume potential in the linear system \eqref{liu22-06-0009} is the domain parameterization which transforms the domain $\Omega$ into a standard rectangle domain $\Pi$, as proposed in \cite{Beshley}. Rather than the nonuniform domain partition of $\Omega$ derived by uniform domain discretization for $\Pi$ in \cite{Beshley}, here we propose a self-adaptive partition scheme (ADS) with explicit quantitative criterion to discrete domain $\Pi$, which ensures the integral in $\Omega$ can be computed numerically with uniform
accuracy. Here we only give the scheme for $\Omega\subset \mathbb{R}^2$, the setting for $\Omega\subset \mathbb{R}^3$ is analogous, see numerics in section 5.

Assume that the domain $\Omega$ is star-like with  the boundary curve $\partial\Omega$ and some center point $P_0\in\Omega$. Then we can represent $\partial\Omega\in \mathbb{R}^2$ by
$$\partial\Omega:=\{x(t)=(x_{1}(t),x_{2}(t))=r(t)(\cos t,\sin t),\; t\in [0,2\pi]\}+P_0,$$
with $2\pi-$periodic radius function $r(t)>0$.
Now we introduce a one-to-one map between  $\Pi:=[0,1]\times [0,2\pi]$ and $\overline\Omega$ by
$$\tilde p(\eta,t):=p(\eta,t)+P_{0}:=(\eta x_{1}(t),\eta x_{2}(t))+P_{0}\in\overline\Omega$$
in the zero measure sense for the coordinate $(\eta,t)\in\Pi$, i.e.,
$$\tilde p|_{\{(\eta,t):(0,1)\times[0,2\pi]\}}=
\Omega\setminus\{P_0\},\quad
\tilde p|_{\{(0,t):\;t\in [0,2\pi]\}}=P_0,\quad \tilde p|_{\{(1,t):\;t\in [0,2\pi]\}}=\partial\Omega.$$

When establishing the solvability of the density pair $(\mu,\psi)$ in continuous function space, we need to restrict both the volume potential and the solution in arbitrarily specified interior domain $\Omega_{\epsilon_0}$ by considering \eqref{liu22-07}. However, the hyper-singularity of the operator $\mathbb{K}_{12}^{\partial\Omega\to\Omega}[\tilde\psi](x)$ for $x\in\Omega$ in discrete version can be removed numerically by taking $x\in\Omega$ in the kernel function with fixed positive distance to $\partial\Omega$. Consequently, instead of \eqref{liu22-07}, we can still consider the system \eqref{liu22-06-0009} directly for the numerical realizations, with the parameterized version
\begin{eqnarray}\label{liu44-01-000}
\begin{cases}
(\mathbb{I}-\mathbb{K}_{11}^{\Omega\to\Omega})[\tilde\mu](\tilde p(\eta,t))+
\mathbb{K}_{12}^{\partial\Omega\to\Omega}[\tilde\psi](\tilde p(\eta,t))=\frac{F(\tilde p(\eta,t))}{\sigma(\tilde p(\eta,t))}, &(\eta,t) \in \Pi \\
\mathbb{K}_{21}^{\Omega\to\partial\Omega}[\tilde\mu](\tilde p(1,t))+(\mathbb{I}-
\mathbb{K}_{22}^{\partial\Omega\to\partial\Omega})[\tilde\psi](\tilde p(1,t))=-2f(\tilde p(1,t)),&t\in[0,2\pi],
\end{cases}
\end{eqnarray}
where the operators have the parametrized representations
$$
\begin{cases}
\mathbb{K}_{11}^{\Omega\to\Omega}[\tilde\mu](\tilde p(\eta,t))&:=\frac{1}{2\pi}\int_{0}^{1}\int_{0}^{2\pi}K_{11}(\eta,t;\xi,\tau)\tilde\mu(\tilde p(\xi,\tau))d\tau d\xi, \qquad (\eta,t)\in \Pi\\
\mathbb{K}_{12}^{\partial\Omega\to\Omega}[\tilde\psi](\tilde p(\eta,t))&:=
\frac{1}{2\pi}\int_{0}^{2\pi}K_{12}(\eta,t;\tau)\tilde\psi(\tilde p(1,\tau))d\tau, \qquad \qquad\quad(\eta,t)\in\Pi\\
\mathbb{K}_{21}^{\Omega\to\partial\Omega}[\tilde\mu](\tilde p(1,t))&:=
\frac{1}{2\pi}\int_{0}^{1}\int_{0}^{2\pi}K_{21}(t;\xi,\tau)
\tilde\mu(\tilde p(\xi ,\tau))d\tau d\xi, \qquad \quad t\in [0,2\pi]\\
\mathbb{K}_{22}^{\partial\Omega\to\partial\Omega}[\tilde\psi](\tilde p(1,t))&:=
\frac{1}{2\pi}\int_{0}^{2\pi}K_{22}(t;\tau)\tilde\psi(\tilde p(1,\tau))d\tau, \qquad \qquad\qquad\;  t \in[0,2\pi]\\
\end{cases}
$$
with the kernels
$$
\begin{cases}
K_{11}(\eta,t;\xi,\tau)&:=\frac{(\tilde p(\xi,\tau)-\tilde p(\eta,t))\cdot\nabla\ln\sigma(\tilde p(\eta,t))}
{|\tilde p(\eta,t)-\tilde p(\xi,\tau)|^2}J(\xi,\tau), \\
K_{12}(\eta,t;\tau)&:=
\frac{\partial}{\partial\nu(\tau)}\left(\frac{(\tilde p(\eta,t)- \tilde p(1,\tau))\cdot\nabla\ln\sigma(\tilde p(\eta,t))}{|\tilde p(\eta,t)- \tilde p(1,\tau)|^2}\right)|x'(\tau)|, \\
K_{21}(t;\xi,\tau)&:=
-\ln\frac{1}{| \tilde p(1,t)-\tilde p(\xi,\tau)|^2}
J(\xi,\tau), \\
K_{22}(t;\tau)&:=
\frac{2\partial}{\partial\nu(\tau)}\left(\ln\frac{1}{|\tilde p(1,t)-\tilde p(1,\tau)|}\right)|x'(\tau)|,
\end{cases}
$$
where
$J(\xi,\tau):=\xi x(\tau)\cdot \nu(\tau)|x'(\tau)|$ is the Jacobian determinant for the coordinated transform $p: \Pi\to\Omega$, with the unit outward  normal direction $\nu(\tau):=\frac{(x_{2}'(\tau),-x_{1}'(\tau))}{|x'(\tau)|}$.

%that is, $K_{22}(t,\tau)$ is of the removable singularity in terms of

% For the surface potential $\mathbb{K}_{22}^{\partial\Omega\to\partial\Omega}$ with weak singularity,
% its kernel $K_{22}(t,\tau)$ has the singularity decomposition by the standard scheme in \cite{Kress, Colton1} that

% \begin{eqnarray}\label{liu44-18}
% K_{22}(t;\tau)=
% K_{22}^{(1)}(t;\tau)+K_{22}^{(2)}(t;\tau)
% +\ln\frac{1}{4\sin^{2}\frac{t-\tau}{2}}K_{22}^{(3)}(\tau),
% \end{eqnarray}

% where $K_{22}^{(1)},K_{22}^{(2)}$ and $K_{22}^{(3)}$ are  continuous functions with the representation

For the surface potential $\mathbb{K}_{22}^{\partial\Omega\to\partial\Omega}$, its
kernel $K_{22}(t,\tau)$ is continuous with the representation
\begin{eqnarray}\label{liu44-18}
K_{22}(t,\tau)=
\begin{cases}
\frac{x''(t)\cdot\nu(t)}{|x'(t))|}, &\quad t=\tau\\
\frac{2(x(t)-x(\tau))\cdot \nu(x(\tau))}{|x(t)-x(\tau)|^2}|x'(\tau)|,&\hbox{elsewhere}.
\end{cases}
\end{eqnarray}

% \begin{eqnarray*}
% K_{22}^{(2)}(t,\tau)=\begin{cases}
% \ln\frac{1}{|x'(t)|^{2}}\;K_{22}^{(3)}(t), &\quad t=\tau\\
% \ln\frac{4\sin^{2}\frac{t-\tau}{2}}{| x(t)- x(\tau)| ^{2}}\;K_{22}^{(3)}(\tau),&\hbox{elsewhere},
% \end{cases}
% \end{eqnarray*}

% $$
% K_{22}^{(3)}(\tau)=\frac{\partial}{\partial\nu(\tau)}\left(\ln\frac{1}{\sigma(\tilde p(1,\tau))}\right)|x'(\tau)|.$$

For this domain transformation, it is easy to see that $\eta x(t)-\xi x(\tau)=0$ if and only if $(\eta,t)=(\xi,\tau)$. Therefore, by introducing
$$G_{1}(\eta,t):=\nu(t)\cdot\nabla\ln\sigma(\tilde p(\eta,t)),\quad G_{2}(\eta,t):=\theta(t)\cdot\nabla\ln\sigma(\tilde p(\eta,t))$$
with unit tangential direction $\theta(t):=\frac{(x_1'(t),x_2'(t))}{|x'(t)|}$, the  kernel $K_{11}(\eta,t;\xi,\tau)$ has the decomposition
\begin{eqnarray}\label{liu44-17}
K_{11}(\eta,t;\xi,\tau)\equiv K^{(11)}(\eta,t;\xi,\tau)+\frac{ J(\xi ,\tau)}{\eta}\frac{G_{2}(\eta,t)}{2|x'(t)|}\cot\frac{\tau-t}{2},
\end{eqnarray}
where $K^{(11)}(\eta,t;\xi,\tau)$ is a continuous function with the representation
\begin{eqnarray}
K^{(11)}(\eta,t;\xi,\tau)
=
\begin{cases}
J(\eta,t)\left[G_{1}(\eta,t)  \frac{x''(t)\cdot\nu(t)}{2\eta|x'(t)|^{2}}+G_{2}(\eta,t)
\frac{x''(t)\cdot \theta(t)}
{2\eta|x'(t)|^{2}}\right], \qquad(\eta,t)=(\xi,\tau)\\
J(\xi,\tau)\left[G_{1}(\eta,t)\kappa _{1}(\eta,t;\xi,\tau)+G_{2}(\eta,t)\kappa _{2}(\eta,t;\xi,\tau )\right ], \;\hbox{elsewhere}
\end{cases}
\end{eqnarray}
with
\begin{eqnarray*}
\begin{cases}
\kappa_{1}(\eta,t;\xi,\tau):=\frac{(\xi x(\tau)-\eta x(t))\cdot\nu(t)}{|\eta x(t)-\xi x(\tau)|^{2}},\\
\kappa_{2}(\eta,t;\xi,\tau):=\frac{(\xi x(\tau)-\eta x(t))\cdot\theta(t)}{|\eta x(t)-\xi  x(\tau)|^{2}}-\frac{1}{2\eta|x'(t)|}\cot\frac{\tau-t}{2}.
\end{cases}
\end{eqnarray*}

%For the kernel in operator $\mathbb{S}_{31}^{\Omega\to\Omega_{0}}$, we also have the decomposition
%\begin{eqnarray}\label{liu44-20}
%S_{31}(\eta,t;\xi,\tau)=
%S^{(31)}(\eta,t;\xi,\tau)
%+\frac{1}{2}\ln\frac{1}{4\sin^{2}\frac{t-\tau}{2}}\;J(\xi,\tau)
%\end{eqnarray}
%with the continuous fucntion
%\begin{eqnarray*}
%S^{(31)}(\eta,t;\xi,\tau)=
%\begin{cases}
%\frac{1}{2}J(\eta,t)\ln\frac{1}{|\eta x'(t)|^2}, &(\eta,t)=(\xi,\tau)\\
%\frac{1}{2}J(\xi,\tau)\ln\frac{4\sin^{2}\frac{t-\tau}{2}}{|\eta x(t)- \xi x(\tau)| ^2},&\hbox{elsewhere}.
%\end{cases}
%\end{eqnarray*}

The weak singular integrals in
% \eqref{liu44-18} and
\eqref{liu44-17}  with respect to $\tau$ can be computed using the standard formula \cite{Colton1}
\begin{eqnarray}\label{liu44-11}
\begin{cases}
\frac{1}{2\pi}\int_{0}^{2\pi}g(\tau)\cot\frac{\tau-t}{2}d\tau \approx \sum_{j=0}^{2n-1} T_{j}(t;n)g(t_{j}),\\
\frac{1}{2\pi}\int_{0}^{2\pi}g(\tau)\ln\left(4\sin^{2}\frac{\tau-t}{2}\right) d\tau \approx \sum_{j=0}^{2n-1} F_{j}(t;n)g(t_{j}),
\end{cases}
\end{eqnarray}
with the weight functions
\begin{eqnarray}\label{liu-Fj_Tj}
\begin{cases}
F_{j}(t;n)=-\frac{1}{n}\left [ \sum_{m=1}^{n-1} \frac{1}{m}\cos m(t-t_{j})+\frac{1}{2n}\cos n(t-t_{j})\right ],\\
T_{j}(t;n)=-\frac{1}{n}\sum_{m=1}^{n-1} \sin m(t-t_{j})-\frac{1}{2n}\sin n(t-t_{j})
\end{cases}
\end{eqnarray}
for $t\in[0,2\pi)$ and $j=0,\cdots,2n-1$.

Now we consider the volume potential with smooth integrands. The special techniques for computing the volume potentials $\mathbb{K}_{11}^{\Omega\to\Omega}, \mathbb{K}_{21}^{\Omega\to\partial\Omega}$  should be developed, notice the fact that, when we decompose the  integral domain $\Omega$ into the summation of $N$-circle ring $$\Omega_i:=\{\tilde p(\xi,\tau): (\xi,\tau)\in [\xi_{i-1},\xi_i]\times [0,2\pi]\}\subset\Omega$$ with $\xi_i:=\frac{i}{N}$ for $i=1,\cdots, N$ along radius direction, the area of $\Omega_i$ will increase for large $i$, see Figure \ref{liu-Pi2O} for the geometric configuration. So we should introduce some self-adaptive rule to compute the integrals in all $\Omega_i$ for keeping the uniform accuracy of the integral in $\Omega$. That is, the integrals in $\Omega_i$ should be computed by different steps $\tau$ for different $i$.

%\begin{figure}[htbp]
%\centering
%\subfigure[domain $\Pi $]{
%\begin{minipage}[t]{0.5\linewidth}
%\centering
%\includegraphics[scale=0.4]{fig1Pi(1).png}
%\caption{fig1}
%\end{minipage}%
%}%
%\subfigure[domian $\Omega $]{
%\begin{minipage}[t]{0.5\linewidth}
%\centering
%\includegraphics[scale=0.4]{fig1O(2).png}
%\caption{fig2}
%\end{minipage}%
%}%
%\centering
%
%\caption{$\Pi_i$ is transformed as $\Omega_i$ by $\Omega_i=\tilde p(\Pi_i)$.}
%\end{figure}

\begin{figure}[htbp]
\centering
\subfigure[domain $\Pi $]
{
\begin{minipage}[t]{0.5\linewidth}
\centering
\includegraphics[scale=0.4]{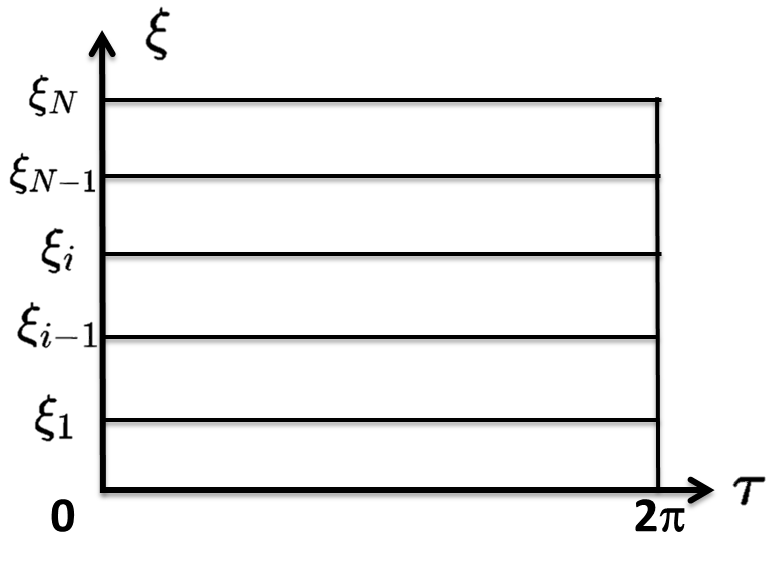}
\end{minipage}
}
\subfigure[domian $\Omega $]
{
\begin{minipage}[t]{0.45\linewidth}
\centering
\includegraphics[scale=0.4]{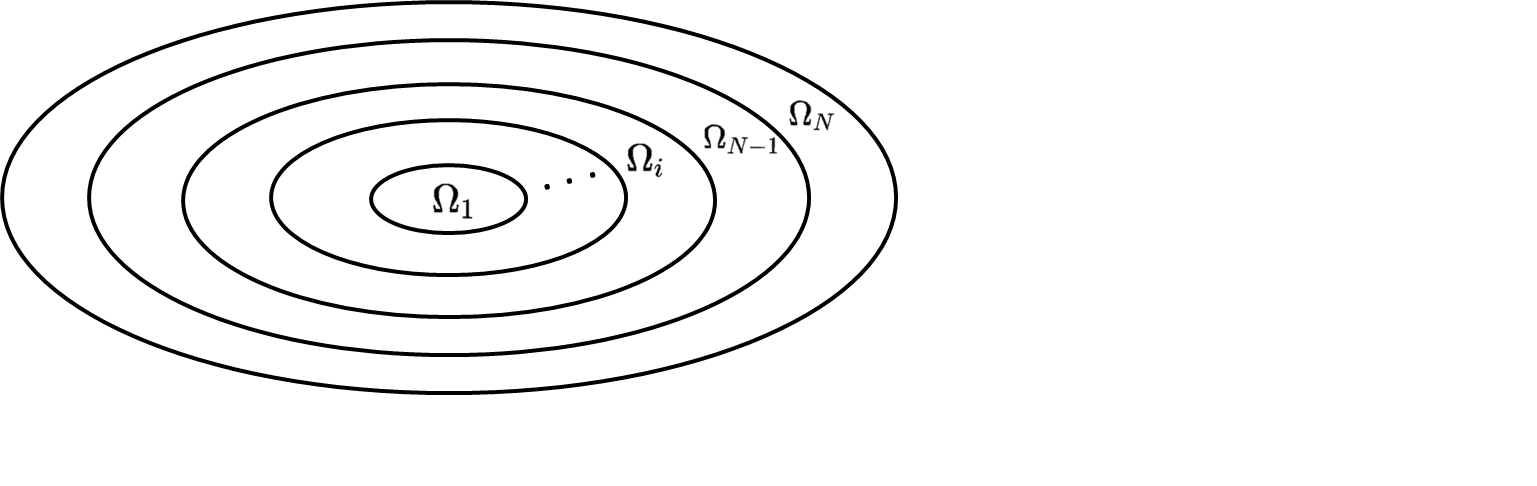}
%\caption{fig2}
\end{minipage}%
}%
\centering
\caption{$\Pi_i$ is transformed as $\Omega_i$ by $\Omega_i=\tilde p(\Pi_i)$.}
\label{liu-Pi2O}
\end{figure}

%
%\begin{figure}[htbp]
%\centering
%\subfigure[domain $\Pi $]{
%\begin{minipage}[t]{0.5\linewidth}
%\centering
%\includegraphics[scale=0.3]{fig2dPi(1).png}
%%\caption{fig1}
%\end{minipage}%
%}%
%\subfigure[domian $\Omega $]{
%\begin{minipage}[t]{0.5\linewidth}
%\centering
%\includegraphics[scale=0.4]{fig2d(2).png}
%%\caption{fig2}
%\end{minipage}%
%}%
%\centering
%\caption{ same number of partitions }
%\end{figure}
%
%From Fig.2,we can know taht  it is  inappropriate to select the same number of partitions $2n$ in $[0.\pi)$ for each $\eta_{i}$,since the difference of partition the  domian $\Omega$ are too large in area.

For the grids $0=\xi_{0}<\xi_1<\cdots<\xi_N=1$, we have the decomposition
$$\overline\Omega\equiv \bigcup_{i=1}^N\overline\Omega_i\equiv \bigcup_{i=1}^N\left\{\tilde p(\xi,\tau):(\xi,\tau)\in\Pi_i:= [\xi_{i-1},\xi_i]\times [0,2\pi]\right\}.$$

\begin{rem}
In the special case of $\Omega$ being an unit cycle, the area of each layer is
\begin{eqnarray}\label{liu44-03-1}
|\Omega_{i}|=(\frac{i}{N})^{2}\pi-(\frac{i-1}{N})^{2}\pi
=\frac{2i-1}{N^{2}}\pi,\quad i=1,\cdots,N,
\end{eqnarray}
which increases linearly with respect to $i$.
\end{rem}

To compute the integral in each $\Omega_i$ with higher accuracy, we introduce the closed curve $\{\tilde p(\xi_{i-1/2},\tau):\tau\in [0,2\pi]\}$  with $\xi_{i-1/2}:=(\xi_{i-1}+\xi_i)/2$ for $i=1,\cdots,N$. Inserting these middle lines
$\xi=\xi_{i-1/2}$ into $\Pi$, the rectangle $\Pi$ is double-partitioned as
$2N$ rectangles with the grids
$$0=\xi_0<\xi_{1/2}<\xi_1<\xi_{3/2}<\xi_2\cdots<\xi_{N-1}<\xi_{N-1/2}<\xi_N=1.$$
Now we apply the middle-rectangle formula in $\Omega_0, \Omega_N$ and the Simpson's formula in $\Omega_i$ for $i=1,\cdots, N-1$ to compute
$$\int_{\Omega_i}s(x)dx=\int_{\Pi_i}s(\tilde p(\xi,\tau))J(\xi,\tau)d\xi d\tau:
=\int_{\xi_{i-1}}^{\xi_i}d\xi\int_0^{2\pi}s_J(\xi,\tau) d\tau$$
with $s_J(\xi,\tau):=s(\tilde p(\xi,\tau))J(\xi,\tau)$.
That is, we have
\begin{eqnarray*}
\int_{\Omega_i}s(x)dx
\approx
\begin{cases}
\int_0^{2\pi}s_J(\xi_{i-1/2},\tau) d\tau\;(\xi_i-\xi_{i-1}), &i=1,N\\
\int_0^{2\pi}\left[s_J(\xi_{i-1},\tau)+4s_J(\xi_{i-1/2},\tau)+s_J(\xi_{i},\tau)\right]
d\tau\;\frac{\xi_i-\xi_{i-1}}{6}, &i=2,\cdots,N-1,
\end{cases}
\end{eqnarray*}
which is rewritten as
\begin{equation}\label{liul11-04-02}
\int_{\Omega_i}s(x)dx\approx \int_0^{2\pi}
\left[\alpha_i^1s_J(\xi_{i-1},\tau)+\alpha_i^2s_J(\xi_{i-1/2},\tau)+
\alpha_i^3s_J(\xi_i,\tau)\right]d\tau
\end{equation}
for $i=1,\cdots, N$ with
\begin{eqnarray}\label{liu11-03-01}
(\alpha_i^1,\alpha_i^2,\alpha_i^3)=
\begin{cases}
(\xi_i-\xi_{i-1})(0,1,0), &i=1,N,\\
\frac{\xi_i-\xi_{i-1}}{6}(1,4,1), &i=2,\cdots, N-1.
\end{cases}
\end{eqnarray}
In terms of \eqref{liul11-04-02}, we need to compute the integrals $\int_0^{2\pi}s_J(\xi_{i},\tau)d\tau, \int_0^{2\pi}s_J(\xi_{i-\frac{1}{2}},\tau)d\tau$ for $i=1,2,\cdots, N$. Define
$\tilde\xi_i:=\xi_{\frac{i}{2}}$ for $i=1,2,\cdots, 2N$, see Figure \ref{liu-Refine}.

\begin{figure}[htbp]
\centering
\subfigure[domain $\Pi $]{
\begin{minipage}[t]{0.35\linewidth}
\centering
\includegraphics[scale=0.4]{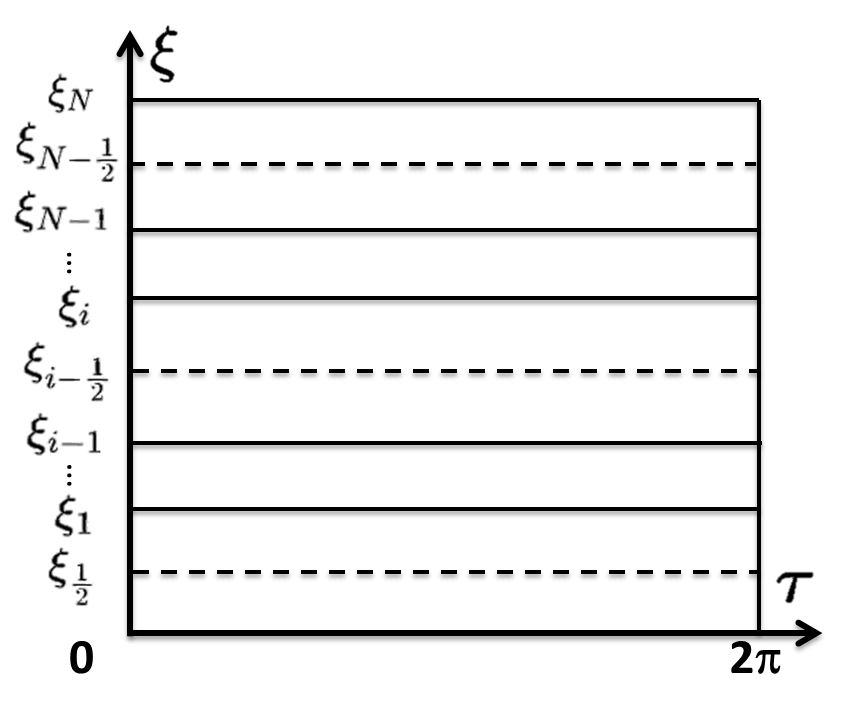}
%\caption{fig1}
\end{minipage}%
}%
\subfigure[domian $\xi $]{
\begin{minipage}[t]{0.7\linewidth}
\centering
\includegraphics[scale=0.4]{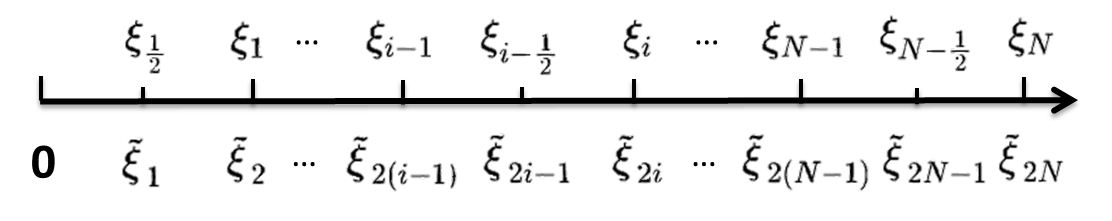}
%\caption{fig2}
\end{minipage}%
}%
\centering
\caption{Refinement of the interval $[0,1]$ as $2N$ subintervals and reordering the partition points.}
\label{liu-Refine}
\end{figure}

Divide the segment $\{(\tilde\xi_i,\tau): \tau\in [0,2\pi]\}$  as $2n_{i}$ equi-intervals for $i=1,\cdots,2N$, where we specify
$n_{1}=2^{k_{1}}$ and $n_{i}=n_{1}2^{k_{i}}$ with $k_{i}\in \mathbb{N}^{+}$ for $i=2,\cdots,2N$.
Then, based on the $2n_i+1$ grids on $\{(\tilde\xi_i,\tau):\tau\in [0,2\pi]\}$, we divide the layer $$\tilde\Omega_i:=\{(\xi,\tau)\in [\tilde\xi_{i-1},\tilde\xi_i]\times [0,2\pi],i=1,\cdots,2N\}$$ as the union of several triangles. In terms of the case either $n_i=n_{i-1}$ or $n_i=2n_{i-1}$, there are two different partitions, namely, either the red partition or the black one, as shown in Figure 3.

\begin{figure}[htbp]
\centering
\subfigure[domain $\Pi $]{
\begin{minipage}[t]{0.5\linewidth}
\centering
\includegraphics[scale=0.4]{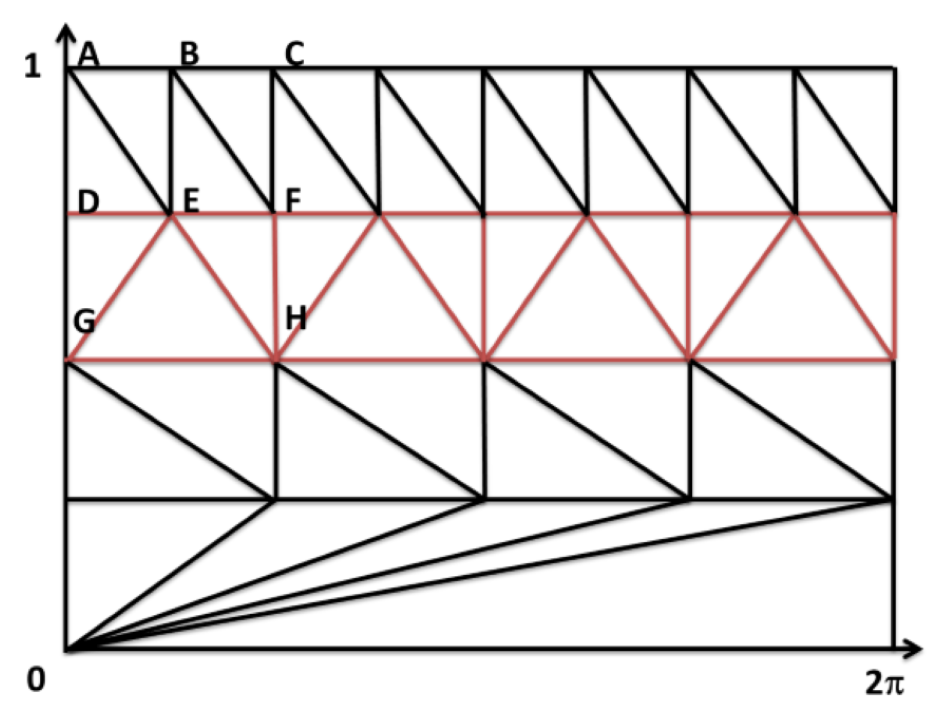}
%\caption{fig1}
\end{minipage}%
}%
\subfigure[domian $\Omega $]{
\begin{minipage}[t]{0.5\linewidth}
\centering
\includegraphics[scale=0.4]{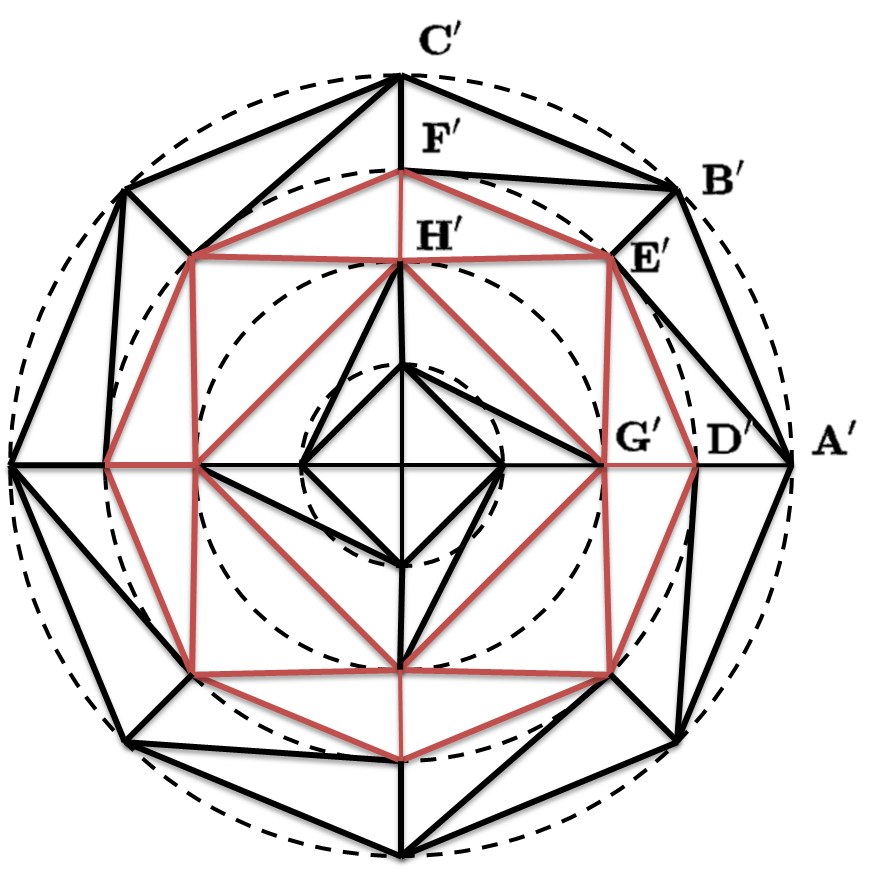}
%\caption{fig2}
\end{minipage}%
}%
\centering
\caption{The adaptive partition scheme at each layer.}
\end{figure}

Our criteria for specifying $k_i$ is that the area of each triangle in $\tilde\Omega_i$ for $i=2,\cdots,2N$ has almost the same size as the triangle in $\tilde\Omega_1$, see Figure 3(b). More precisely, we take the rule
\begin{eqnarray}\label{liu44-03-2}
|s_i|:=\frac{|\tilde\Omega_{i}|}{2\cdot2n_{i}}\approx \frac{|\tilde\Omega_{1}|}{2n_{1}}:=|s_1| \hbox{ or }
|s_i|:=\frac{|\tilde\Omega_{i}|}{3\cdot2n_{i}} \approx \frac{|\tilde\Omega_{1}|}{2n_{1}}:=|s_1|,
\end{eqnarray}
which is specified for $i=2,\cdots,2N$ from the quantitative rule
\begin{eqnarray}\label{liu44-03-3}
\frac{1}{3}\le\frac{|s_{i}|}{|s_{1}|}\le 3.
\end{eqnarray}

Using the grids $\{\tilde\xi_i: i=1,\cdots, 2N\}$ with $\tilde\xi_1=\xi_{\frac{1}{2}}$ and $\tilde\xi_{2N}=\xi_N$ in \eqref{liul11-04-02}
together with \eqref{liu11-03-01}, it follows from the straightforward but lengthy computations that
\begin{equation}\label{liu44-13}
\int_\Omega s(x)dx\equiv\sum_{i=1}^N\int_{\Omega_i} s(x)dx\approx
\sum_{k=1}^{2N-1}c_k\int_0^{2\pi}s_J(\tilde\xi_k,\tau)d\tau\approx
\sum_{k=1}^{2N-1}c_k\frac{2\pi}{2n_k}\sum_{j=0}^{2n_k-1}s_J(\tilde\xi_k,t_j^k),
\end{equation}
with $t_j^k:=j\frac{\pi}{n_k}$ and  the weights
\begin{eqnarray}\label{liu44-14}
c_{2i-1}=
\begin{cases}
\frac{1}{N}, &i=1,N,\\
\frac{2}{3N},&i=2,\cdots,N-1,
\end{cases}
\qquad
c_{2i}=
\begin{cases}
\frac{1}{6N},&i=1,N,\\
\frac{1}{3N},&i=2,\cdots,N-2.
\end{cases}
\end{eqnarray}

\begin{table}
\begin{center}
\caption{The partition numbers at different layers from our self-adaptive strategy.}
\label{table1}
\begin{tabular}{|c |c| c |c |c| c |c |c |c |c|}
\hline
$n_1$&$n_2$&$n_3$&$n_4$&$n_5$&$n_6$&$n_7$&$n_8$&$n_9$&$n_{10}$\\
\hline
$2^3$&$2^4$&$2^5$&$2^5$&$2^6$&$2^6$&$2^6$&$2^6$&$2^7$&$2^7$\\
\hline
$n_{11}$&$n_{12}$&$n_{13}$&$n_{14}$&$n_{15}$&$n_{16}$&$n_{17}$&$n_{18}$&$n_{19}$&$n_{20}$\\
\hline
$2^7$&$2^7$&$2^7$&$2^7$&$2^7$&$2^7$&$2^8$&$2^8$&$2^8$&$2^8$\\
\hline
\end{tabular}
\label{Table-03-1}
\end{center}
\end{table}

Since the area $\tilde\Omega_i$ is hard to compute for $\Omega$ of general shape, by employing \eqref{liu44-03-1} as a benchmark for $\tilde\Omega_i$ in \eqref{liu44-03-2} and \eqref{liu44-03-3}, we finally take $k_{i}=\left \lfloor \log_{2}{(2i-1)} \right \rfloor$
for $i=2,\cdots,2N$ with specified $k_1$. For the concrete case $N=10,k_1=3$, the step sizes with respect to $\tau$ from this self-adaptive strategy for different $\tilde \xi_i$ are presented in Table 1. It can be seen that partition numbers  from the inner curve to the outer curve will become large.

%
%Introducing some interpolation quadrature rules:
%
%\begin{eqnarray}\label{liu44-13}
%\frac{1}{2\pi}\int_{0}^{2\pi}g(\tau)d\tau \approx
%\frac{1}{2n}\sum_{j=0}^{2n-1} g(t_{j})
%\end{eqnarray}
%
%\begin{eqnarray}\label{liu44-14}
%\int_{0}^{1}\frac{1}{2\pi}\int_{0}^{2\pi}g(\xi,\tau)d\tau d\xi \approx
%\sum_{i=1}^{2N-1} \frac{\alpha _{i}}{2n_{i}}\sum_{j=0}^{2n_{i}-1} g(\tilde\xi_{i},t_{j}^i)
%\end{eqnarray}
%with quadrature weights $\alpha_{i}\in \mathbb{R}$ and quadrature points $t_{j}^i=\frac{j\pi}{n_i}$.

Based on the decomposition
% \eqref{liu44-18} and
\eqref{liu44-17}, we apply \eqref{liu44-11} to compute the singular integrals and
\eqref{liu44-13}-\eqref{liu44-14} to compute the regular ones in \eqref{liu44-17}, which lead to the linear system
\begin{eqnarray}\label{liu44-26}
\begin{cases}
\tilde\mu(\tilde p(\eta,t))- \sum\limits_{i=1}^{2N-1}c_i \sum\limits_{j=0}^{2n_{i}-1}\Hat{K}_{11}(\eta,t;\tilde\xi_{i},t_{j}^i)
\tilde\mu_{ij}
+
\sum\limits_{j=0}^{2n_{2N}-1}\frac{1}{2n_{2N}}K_{12}(\eta,t;t_{j}^{2N})\tilde\psi_{j}=\frac{F(\tilde p(\eta,t))}{\sigma(\tilde p(\eta,t))},  \\
\sum\limits_{i=1}^{2N-1}c_i\frac{1}{2n_i} \sum\limits_{j=0}^{2n_{i}-1}K_{21}(t;\tilde\xi_{i},t_j^i)\tilde\mu_{ij}+\tilde\psi(\tilde p(1,t))-\sum\limits_{j=0}^{2n_{2N}-1}\frac{1}{2n_{2N}}K_{22}(t;t_j^{2N})\tilde\psi_{j}=-2f(\tilde p(1,t)),
\end{cases}
\end{eqnarray}
with the known coefficients
\begin{eqnarray*}
\Hat{K}_{11}(\eta,t;\tilde\xi_{i},t_j^i):=\frac{1}{2n_{i}}
K^{(11)}(\eta,t;\tilde\xi_{i},t_j^i)+\frac{ J(\tilde\xi_{i} ,t_j^i)}{\eta}\frac{G_2(\eta,t)}{2|x'(t)|}T_{j}(t;n_{i})
\end{eqnarray*}
and the unknowns $\tilde \mu_{ij}:=\tilde \mu(\tilde p(\tilde\xi_{i},t_j^i)),\; \tilde \psi_{j}:=\tilde \psi(\tilde p(1,t_j^{2N}))$.
% \begin{eqnarray*}
% \Hat{K}_{11}(\eta,t;\tilde\xi_{i},t_j^i):=\frac{1}{2n_{i}}
% K^{(11)}(\eta,t;\tilde\xi_{i},t_j^i)+\frac{ J(\tilde\xi_{i} ,t_j^i)}{\eta}\frac{G_2(\eta,t)}{2|x'(t)|}T_{j}(t;n_{i}),\\
% \Hat{K}_{22}(t;t_j^{2N}):=\frac{1}{2n_{2N}}[K_{22}^{(1)}(t;t_j^{2N})+K_{22}^{(2)}(t;t_j^{2N})]+
% K_{22}^{(3)}(t_j^{2N})F_{j}(t;n_{2N})
% \end{eqnarray*}
% and the unknowns $\tilde \mu_{ij}:=\tilde \mu(\tilde p(\tilde\xi_{i},t_j^i)),\; \tilde \psi_{j}:=\tilde \psi(\tilde p(1,t_j^{2N}))$.

By specifying $(\eta,t)$ in \eqref{liu44-26} at the collocation points $(\tilde\xi_i, t_j^i)$ with $j= 0,\cdots, 2n_i-1, i=1,\cdots,2N-1$ in the first equation and $t=t_j^{2N}$ in the second equation, we can finally solve $\tilde \mu_{ij}$ for $j= 0,\cdots,2n_i-1, \;i=1,\cdots,2N-1$ and $\tilde\psi_{j}$ for $j=0,1,\cdots,2N-1$.

%Applying this solution to the following equation,we can get the solvability of \eqref{liu11-01}
%\begin{eqnarray*}
%\sum_{i}^{N} \sum_{j=0}^{2n_{i}-1}\alpha_{k}\Hat{S}_{31}(\eta_{m},t_{k};\xi_{i},\tau_{j})\tilde\mu_{ij}+
%\sum_{j=0}^{2n_{N}-1}\frac{1}{2n_{N}}S_{32}(\eta_{m},t_{k};\tau_{j})\tilde\psi_{j}=U(\tilde p(\eta_{m},t_{k})),
%\end{eqnarray*}
%where
%\begin{eqnarray*}
%\Hat{S}_{31}(\eta_{m},t_{k};\xi_{i},\tau_{j})=  \frac{1}{2n_{i}}S^{(31)}(\eta_{m},t_{k};\xi_{i},\tau_{j})
%-\frac{J(\xi_{i},\tau_{j})}{2}F_{j}(t_{k};n_{i})
%\end{eqnarray*}
%$ m = 1,\dots, N, k = 0,\dots, 2n_{m} -1$

\section{The dual reciprocity method}

In the above section, we already propose a scheme dealing with the volume potential \eqref{liu22-07} by partitioning the domain $\Omega$ in some uniform way, which is still the realization of domain integral and consequently suffers from the large number of the unknowns $\{\tilde\mu_{i,j}: j=0,1,\cdots, 2n_i-1, i=1,\cdots, 2N-1\}$ from the partition of $\Omega$. In this section, we apply the dual reciprocity method (DRM) to transform the domain integral into boundary integral. One of the attractive feature of DRM is that the choice of interior nodes can be random and compute the weak singularity in the domain integral explicitly.

The essence of DRM is to transform the volume integral with fundamental solution as kernel function into a surface integral by expanding the integrand in terms of some base functions $\{f_k: k=1,\cdots,M\}$, and consequently decreases the number of grids for integration \cite{DRM_book}.  More precisely, we expand
\begin{eqnarray}\label{liu23-02-18-01}
\tilde\mu(x)\approx\sum_{k=1}^{M}\alpha_k f_k(x),
\end{eqnarray}
where $\{f_k(x):k=1,\cdots,M\}$ constitutes the basis functions in the form
\begin{eqnarray}\label{liu23-02-18-02}
f_k(x)=\Delta \hat f_k(x)
\end{eqnarray}
for some known $\hat f_k(x)$
depending on interior nodes $x_k\in\Omega$ and $\alpha_k$ is the expansion coefficients to be determined.
Substituting \eqref{liu23-02-18-02} and \eqref{liu23-02-18-01} into   $\mathbb{K}_{11}^{\Omega\to\overline\Omega}[\tilde\mu](x)$ in \eqref{liu22-07} yields
\begin{align}\label{liu23-02-18-03}
\mathbb{K}_{11}^{\Omega \to \Omega}[\tilde\mu](x)
&=\int_\Omega \nabla_x\Phi(x,y)\cdot\nabla\ln\sigma(x)\tilde\mu(y)dy\nonumber\\
&\approx
\sum_{k=1}^{M}\alpha_k\nabla_x\left ( \int_\Omega \Phi(x,y)\Delta_y\hat f_k(y) dy\right ) \cdot\nabla\ln\sigma(x),  \quad x\in\Omega.
\end{align}
Defining $D_k(x):=\int_\Omega \Phi(x,y)\Delta_y\hat f_k(y) dy$  and integrating by parts, we have
% \begin{align}\label{liu23-02-18-04}
% \mathbb{K}_{11}^{\Omega \to \Omega}[\tilde\mu](x) \approx
% \sum_{k=1}^{M}\alpha_k\nabla_x D_k (x)\cdot\nabla\ln\sigma(x),  \qquad x\in\Omega
% \end{align}
% where $D_k(x)$ can be expressed as :
\begin{align}\label{liu23-02-18-03-01}
D_k(x)&=  \int_{\partial \Omega} [\Phi(x,y)\partial_{\nu(y)} \hat f_k(y)
- \hat f_k(y)\partial_{\nu(y)}\Phi(x,y)] ds(y)
+\int_\Omega\hat f_k(y)\Delta_y\Phi(x,y)dy\nonumber\\
&=\int_{\partial \Omega}[\Phi(x,y)\partial_{\nu(y)} \hat f_k(y)
- \hat f_k(y)\partial_{\nu(y)}\Phi(x,y)] ds(y)
-\hat f_k(x), \quad x\in\Omega.
\end{align}

It is noted that $D_k(x)$ involves only boundary integral and the singularity in $\mathbb{K}_{11}^{\Omega\to\Omega}[\tilde\mu](x)$ has been integrated explicitly. Thus it is convenient to handle $\mathbb{K}_{11}^{\Omega\to\Omega}[\tilde\mu](x)$ numerically, since only the boundary discretization is required for specified  $\hat f_k$ in \eqref{liu23-02-18-02}.

Similarly, the domain integral $\mathbb{K}_{21}^{\Omega\to\partial\Omega}[\tilde\mu](x)$ can also be converted into the boundary integral
\begin{align}\label{liu23-02-18-05}
\mathbb{K}_{21}^{\Omega \to \partial\Omega}[\tilde\mu](x)
=-2\int_\Omega \Phi(x,y)\tilde\mu(y)dy
\approx
-2\sum_{k=1}^{M}\alpha_k D_k(x),  \quad x\in\partial\Omega,
\end{align}
with $D_k(x)$ for $x\in\partial\Omega$ is defined by
$D_k(x):=\lim_{z\in\Omega, z\to x}D_k(z)$,
which is of the expression
\begin{align}\label{liu23-Dk-boundary}
D_k(x)&=\int_{\partial \Omega} [\Phi(x,y)\partial_{\nu(y)} \hat f_k(y)
- \hat f_k(y)\partial_{\nu(y)}\Phi(x,y)] ds(y)
+\int_\Omega \hat f_k(y)\Delta_y\Phi(x,y)dy\nonumber\\
&=\int_{\partial \Omega} [\Phi(x,y)\partial_{\nu(y)} \hat f_k(y)
- \hat f_k(y)\partial_{\nu(y)}\Phi(x,y)] ds(y)
-\frac{1}{2}\hat f_k(x),\quad x\in\partial\Omega
\end{align}
due to the jump relation of surface potentials.
The boundary integrals with $\Phi(x,y)$ and  $\partial_{\nu(y)}\Phi(x,y)$ being kernels can be computed from standard formulas.

Since we transform all the domain integrals in the linear system into boundary integrals via DRM, the numerical solution of the boundary value problem \eqref{liu11-01} can be generated conveniently, since the linear system for the density functions  involves only the values on the boundary grids. Such a scheme will decrease the computational cost, especially in 3-dimensional cases.

%\subsection{Discrete version for linear system by DRM in 2D}

Now we compute $\mathbb{K}_{11}^{\Omega \to \Omega}[\tilde\mu](x)$ for $x\in\Omega$
and $\mathbb{K}_{21}^{\Omega \to \partial\Omega}[\tilde\mu](x)$ for $x\in\partial\Omega$ by  \eqref{liu23-02-18-03} and \eqref{liu23-02-18-05} in terms of the expression $D(x)$.

Here we only give the scheme for the case $\Omega\subset \mathbb{R}^2$, where the
Nystr$\ddot{\text{o}}$m scheme can be used to construct the quadrature rules for boundary integrals. As for the case $\Omega\subset \mathbb{R}^3$,  the discrete scheme stated in \cite{Colton2} (Chapter 3.7) can be applied  to handle the singularities and build an efficient scheme for computing the integrals with smooth integrals. In section 5, a numerical example showing this methodology will be presented.

Assume that the boundary curve $\partial \Omega$ has a $2\pi$-periodic parametric representation
$$\tilde x(t):=x(t)+P_0 \in \partial \Omega, \quad 0\le t \le 2\pi.$$

Introduce the quadrature points $t_j=\pi j/n,j=0,1,\dots,2n-1$ for boundary curve $\{\tilde x(t): t\in [0,2\pi]\}$. Since $\nabla D_k(x)\cdot \nabla \ln \sigma(x)$ with $D_k(x)$ given by \eqref{liu23-02-18-03-01} has smooth kernel function, the discrete version of $\mathbb{K}_{11}^{\Omega\to\Omega}[\tilde\mu](x)$ for $x\in\Omega$ in terms of Nystr$\ddot{\text{o}}$m is \cite{Colton2}
\begin{align}\label{liu23-k11_discrete}
\mathbb{K}_{11}^{\Omega\to\Omega}[\tilde\mu](x)
&\approx \sum_{k=1}^{M}\alpha_k \left [ \sum_{j=1}^{2n}\frac{1}{2n} \left( \tilde G_{j}(x)\hat q_{jk}-\tilde H_{j}(x)|x'(t_j)|\hat f_{jk}\right)-\nabla\hat f_{k}(x)\cdot \nabla \ln \sigma(x)\right ]
\end{align}
for $x\in \Omega$, where $\hat f_{jk}=\hat f_k(y_j)$,
 $\hat q_{j k}= |x'(t_j)|\left.\partial_{\nu(y)} \hat f_k(y)\right | _{y=y_j} $ with $y_j=\tilde x(t_j)$ and
\begin{align*}
    \tilde G_{j}(x)=\frac{(x-y_j)\cdot \nabla \ln \sigma(x)}{|x-y_j|^2} \qquad
    \tilde H_{j}(x)=\left.\frac{\partial}{\partial\nu(y)}\left(\frac{(x-y)\cdot \nabla \ln \sigma(x)}{|x-y|^2}\right)\right | _{y=y_j}
\end{align*}

For computing \eqref{liu23-02-18-05} with logarithmic singularity in $D_k(x)$ for $x\in \partial\Omega$, we firstly make the decomposition
\begin{align*}
    &-2D_k(\tilde x(t))\\
    =&\frac{1}{2\pi}\int_{0}^{2\pi}\left[\ln| x(t)-x(\tau)|^2\partial_{\nu(\tilde x(\tau))}\left(\hat{f}_k(\tilde x(\tau))\right)|x'(\tau)|+\hat{f}_k(\tilde x(\tau))K_{22}(t;\tau)\right]d\tau+\hat{f}_k(\tilde x(t))\\
    =&\frac{1}{2\pi}\int_{0}^{2\pi}\left[\ln\frac{| x(t)-x(\tau)|^2}{4\sin^2\frac{t-\tau}{2}}+\ln\left(4\sin^2\frac{t-\tau}{2}\right)\right]|x'(\tau)|\partial_{\nu(\tilde x(\tau))}\left(\hat{f}_k(\tilde x(\tau))\right)d\tau+ \\
    &\frac{1}{2\pi}\int_{0}^{2\pi}
    K_{22}(t,\tau)\hat{f}_k (\tilde x(\tau)) d\tau+\hat{f}_k(\tilde x(t)), \quad 0\le t \le 2\pi,
\end{align*}
which leads to for $\tilde x\in\partial\Omega$ that
\begin{align}
    \mathbb{K}_{21}^{\Omega \to \partial\Omega}[\tilde\mu](\tilde x(t))
%& \approx
%-2\sum_{k=1}^{M}\alpha_k D_k(\tilde x(t))  \nonumber\\
& \approx \sum_{k=1}^{M}\alpha_k \left [\sum_{j=1}^{2n}\left(G_{j}(\tilde x(t))\hat q_{jk}+
\frac{1}{2n}K_{22}(t,t_j)\hat f_{jk}\right)+\hat f _{k}(\tilde x(t))\right],
\end{align}
where $G_j(\tilde x(t))$ is a continuous function
\begin{align*}
    G_{j}(\tilde x(t))=F_j(t;n)+
\begin{cases}
\frac{1}{2n}\ln(|x'(t)|^2), &t=t_j
\\
\frac{1}{2n}\ln\frac{|x(t)-x(t_j)|^2}{4\sin^2\frac{t-t_j}{2}}, &\hbox{elsewhere}
\end{cases}
\end{align*}
with $F_j(t;n)$ defined in \eqref{liu-Fj_Tj}.
Finally we yield the following approximate version of \eqref{liu22-07}:
\begin{align}\label{liu-23-discrete}
    \begin{cases}
    \sum\limits_{k=1}^{M} \alpha_k f_k(x)-\sum\limits_{k=1}^{M}\alpha_k DK_{11}(x;k)-\sum\limits_{j=1}^{2n}\frac{1}{2n}\tilde H_j(x)|x'(t_j)|\tilde \psi_j=\frac{F( x)}{\sigma(x)},\\
    \sum\limits_{k=1}^{M}\alpha_k DK_{21}(t;k)+\tilde \psi (\tilde x(t))-\sum\limits_{j=1}^{2n}\frac{1}{2n}K_{22}(t;t_j)\tilde \psi_j =-2f(\tilde x(t))
    \end{cases}
\end{align}
with respect to the unknowns $\alpha_k (k=1,\cdots,M)$ and $\tilde \psi_j:=\tilde \psi(\tilde x(t_j)) (j=1,\cdots,2n)$, where the known coefficients
\begin{align*}
    DK_{11}(x;k)&=\sum_{j=1}^{2n}\frac{1}{2n} \left( \tilde G_{j}(x)\hat q_{jk}-\tilde H_{j}(x)|x'(t_j)|\hat f_{jk}\right)-\nabla\hat f_{k}(x)\cdot \nabla \ln \sigma(x),\\
    DK_{21}(t;k)&=\sum_{j=1}^{2n}\left(G_{j}(\tilde x(t))\hat q_{jk}+
\frac{1}{2n}K_{22}(t,t_j)\hat f_{jk}\right)+\hat f _{k}(\tilde x(t)).
\end{align*}

By specifying $(x,t)$ in \eqref{liu-23-discrete} at the collocation points $(x
_k,t_j)$ with $k=1,\cdots,M$ and $j=1,\cdots,2n$, we can finally solve $\alpha_k$ for $k=1,\cdots,M$ and $\tilde \psi_j$ for $j=1,\cdots,2n$. Finally the density function $\tilde \mu(x)$ can be approximated by \eqref{liu23-02-18-01}.

The expansion of density $\tilde\mu(x)$ needs to specify the basis function $f_k(x)$.  Since $\tilde f_k(x)$ is required to meet $\Delta \tilde f_k(x)=f_k(x)$, it is convenient to choose $f_k(x)$ in some special form. For specified internal grid $x_k\in\Omega\subset \mathbb{R}^d$, as recommended in \cite{DRM_Brunton}, a typical form is
\begin{equation}\label{liu11-01-02}
f_k(x)=1+r_k(x):=1+|x-x_k|, \quad k=1,\cdots,M,
\end{equation}
where $r_k(x)$ is the distance between the field point $x$ and the internal grid $x_k$.

For radial basis function $f_k$ in this form,
we also find the radial basis function $\hat f_k(x):=\hat f(r_k)$ to  \eqref{liu23-02-18-02}, which satisfies the ordinary differential equation \cite{DRM_Brunton}
$$
\frac{\mathrm{d}^2 \hat{f} }{\mathrm{d} r_k^2} +\frac{d-1}{r_k}\frac{\mathrm{d} \hat{f} }{\mathrm{d}r_k}=1+r_k.
$$
One solution to this equation is
\begin{eqnarray}
\hat f_k(x)=
\begin{cases}
\frac{r_k^2(x)}{4}+\frac{r_k^3(x)}{9}, &d=2,\\
\frac{r_k^2(x)}{6}+\frac{r_k^3(x)}{12}, &d=3.
\end{cases}
\end{eqnarray}
Then the derivatives $\partial_{\nu(x)}\hat f_k(x), \nabla \hat f_k(x)$ can be computed easily.

\section{Numerical experiments }

Now we do numerics for solving \eqref{liu11-01} by our proposed schemes for two examples in $\Omega\subset \mathbb{R}^d$, showing the validity of our proposed schemes. Moreover, we also compare the numerical performances of two schemes.

{\bf Example 1.} Consider  two different domains $\Omega\subset \mathbb{R}^2$. One is a heart-shaped domain with
\begin{equation}\label{nu1}
\partial\Omega:=\Gamma\equiv\{x(t)=(0.2\cos t,0.4\sin t-0.3 \sin^{2}t),\; t\in [0,2\pi]\}+(0.5,1),
\end{equation}
and the other one is an elliptical domain with
\begin{equation}\label{nu2}
\partial\Omega:=\Gamma\equiv\{x(t)=(\cos t,0.5 \sin t),\; t\in [0,2\pi]\}+(0.5,1),
\end{equation}
that is, we always take $P_{0}=(0.5,1)$ as the center of domain $\Omega$.
For the conductivity
$$\sigma(x,y)=2+\frac{1}{5}\sin(25x)+\frac{1}{5}\cos(25y)$$ and the source function $$F(x,y)=-10[x\cos(25x)+\sin(25y)]-[4+\frac{2}{5}\sin(25x)+\frac{2}{5}\cos(25y)]$$
in $\Omega\subset \mathbb{R}^2$,  it is easy to verify that
$u_{ex}(x,y)=x^2-2y+3$
is the exact solution to the PDE in \eqref{liu11-01}. Correspondingly, we take $f(x,y)=u_{ex}(x,y)|_{\partial\Omega}$ as the boundary value in \eqref{liu11-01}.

\vskip 0.3cm

{\bf 1A: Realization of ADS in section 3.}

For the discretization of $\Omega$ along radial direction, we divide $[0,1]$ for $\xi$ as $20$ subintervals in $(\xi,\tau)$ coordinates by setting $N=10$. Then the outer boundary of $\tilde\Omega_i$ has the representation
\begin{eqnarray}\label{liu44-27}
\Gamma_i: =\{\tilde\xi_i x(t): t\in [0,2\pi]\}+(0.5,1)
\end{eqnarray}
with $\tilde\xi_i=i\times\frac{1}{20}$ for $i=1,\cdots,20$. To check the numerical performances, we consider the error of numerical solution in each curve $\Gamma_i\in\Omega$ (local error) and the whole domain $\Omega$ (average error) by introducing the relative error functions
\begin{eqnarray}\label{liu44-28}
\frac{\left \| u_{Nn}-u_{ex} \right \|_{L^2(\Gamma_i)} }{\left \| u_{ex} \right \|_{L^2(\Gamma_i)}}\approx Err_L^i&:\equiv&\left (\frac{  \sum_{j=0}^{2n_i-1} (u_{Nn}-u_{ex})^2(\tilde\xi_i,t_j^i)|\tilde\xi_i x'(t_j^i)|}{\sum_{j=0}^{2n_i-1} u_{ex}^2(\tilde\xi_i,t_j^i)|\tilde\xi_i x'(t_j^i)|}\right )^{1/2},
\\
\frac{\left \| u_{Nn}-u_{ex} \right \|_{L^2(\Omega)} }{\left \| u_{ex} \right \|_{L^2(\Omega)}}\approx Err_A&:\equiv&\left (\frac{  \sum_{i=1}^{2N-1}\frac{c_i}{n_i}\sum_{j=0}^{2n_i-1} (u_{Nn}-u_{ex})^2(\tilde\xi_i,t_i^j)J(\tilde\xi_i,t_i^j)}
{\sum_{i=1}^{2N-1}\frac{c_i}{n_i}\sum_{j=0}^{2n_i-1} u_{ex}^2(\tilde\xi_i,t_i^j)J(\tilde\xi_i,t_i^j)}\right )^{1/2}
\end{eqnarray}
respectively.

\begin{figure}[htbp]
\centering
\subfigure[Exact solution]{
\begin{minipage}[t]{0.31\linewidth}
\centering
\includegraphics[scale=0.3]{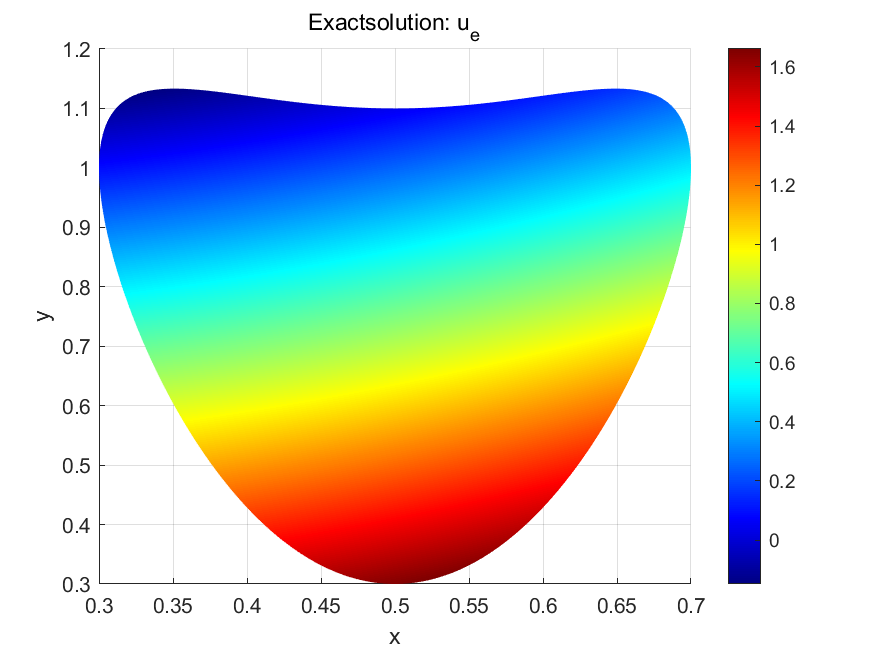}
%\caption{fig1}
\end{minipage}%
}%
\subfigure[Numerical solution]{
\begin{minipage}[t]{0.31\linewidth}
\centering
\includegraphics[scale=0.2]{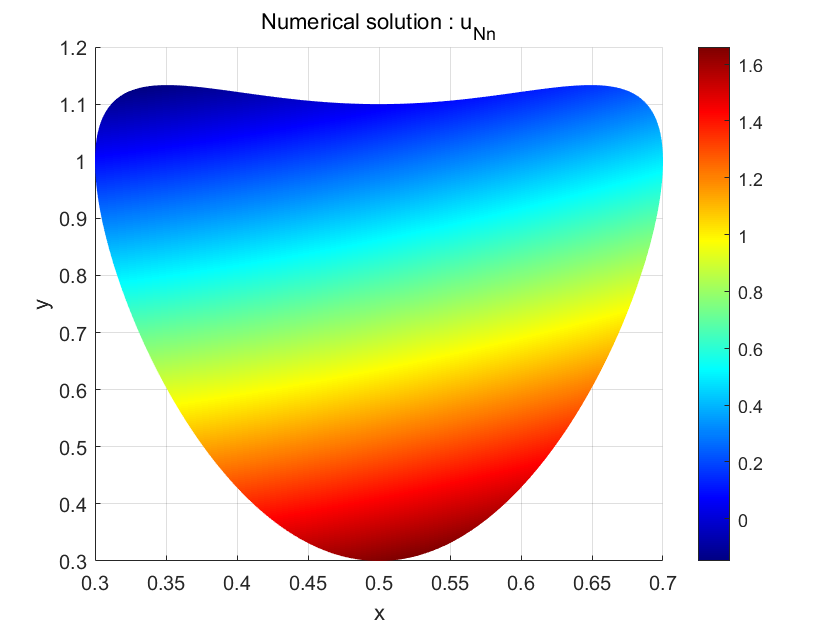}
%\caption{fig2}
\end{minipage}%
}%
\subfigure[Absolute error]{
\begin{minipage}[t]{0.31\linewidth}
\centering
\includegraphics[scale=0.3]{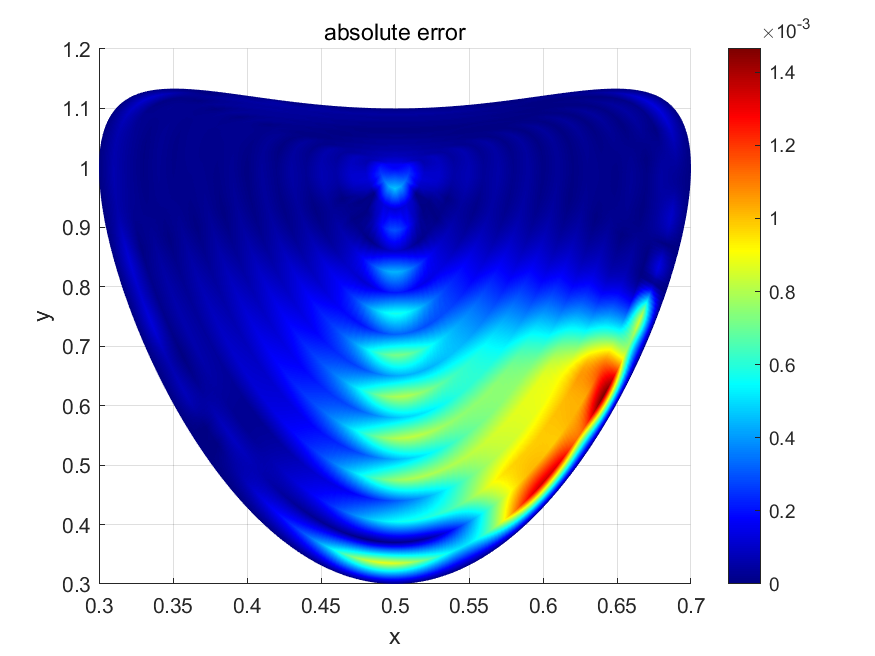}
\end{minipage}
}
\centering
\caption{Numerical result for Example 1 in heart-shaped domain with $N=10,k_1=3$.}
\end{figure}

\begin{figure}[htbp]
\centering
\subfigure[Exact solution]{
\begin{minipage}[t]{0.31\linewidth}
\centering
\includegraphics[scale=0.3]{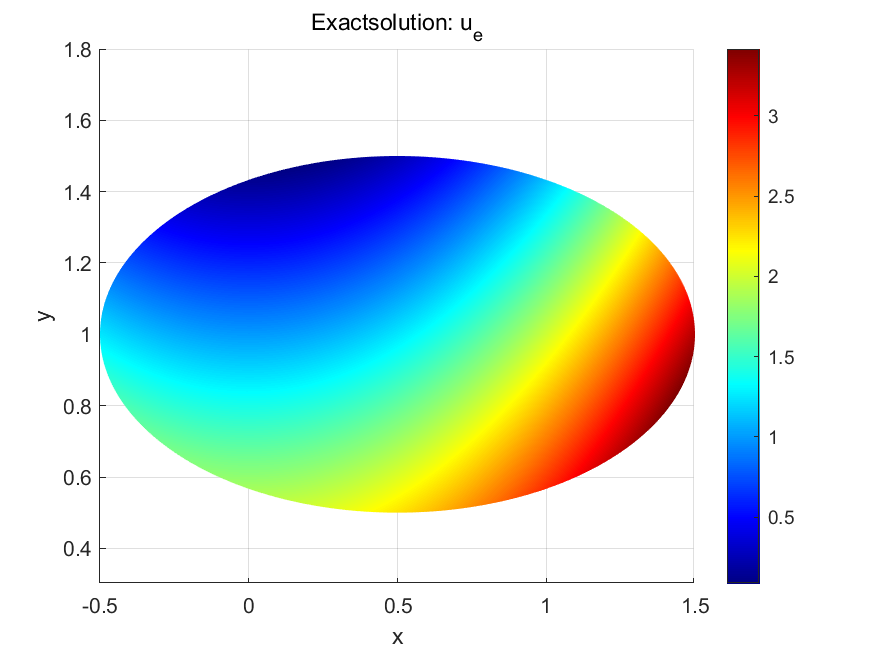}
\end{minipage}
}
\subfigure[Numerical solution]{
\begin{minipage}[t]{0.31\linewidth}
\centering
\includegraphics[scale=0.3]{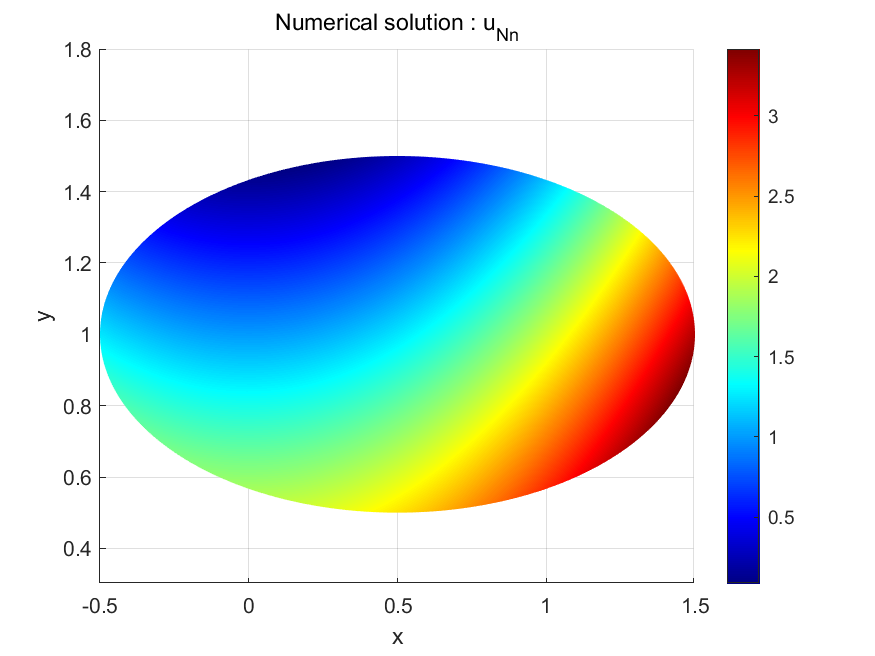}
\end{minipage}
}
\subfigure[Absolute error]{
\begin{minipage}[t]{0.31\linewidth}
\centering
\includegraphics[scale=0.3]{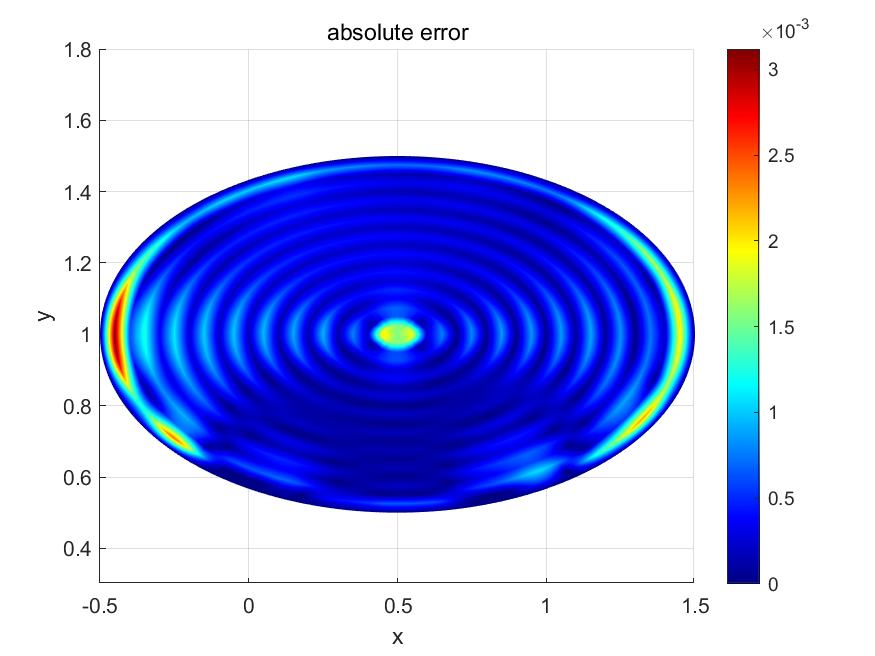}
\end{minipage}
}
\centering
\caption{Numerical result for Example 1 in elliptical domain with $N=10,k_1=3$.}
\end{figure}

The results in two different domains specified by \eqref{nu1} and \eqref{nu2} are illustrated in Figure 4 and Figure 5, where the columns (a),(b),(c) show the exact solution, numerical solution from our ADS, as well as the point-wise error distributions, respectively. It can be seen that ADS proposed in section 3 can yield numerical solution to a very satisfactory level. An interesting observation is that the maximum error always appears in the center area of $\Omega$ near $P_0$ and boundary curve $\Gamma$, the reason is that we always take a cycle $\Omega_1$ without any further partition inside $\Omega_1$ and $\Omega_{2N}$. Also, when we compute the integrals in $\Omega_1$ and $\Omega_{2N}$, as shown in \eqref{liu11-03-01} for $i=1,N$, only the middle-point rectangle quadrature formula are applied, rather than the Simpson's formula.

To describe the error distributions in $\Omega$ precisely, which cannot be observed from Figure 4 and Figure 5, we give a quantitative description in Table 2 and Table 3 for two different domains, where we choose four different curves $\Gamma_2, \Gamma_7,\Gamma_{12},\Gamma_{17}$ inside two domains to show the local error distributions, and also the relative error in $\Omega$ for the average error. For these numerics, we test our scheme for $k_1=1,2,3$, which means the curve $\Gamma_1$ is divided as $2n_1=2\times 2^{k_1}$ sub-intervals along $\tau$ directions. For other curves $\Gamma_i$  with $i=1,\cdots,20$, the step sizes with respect to $\tau$ are $2n_i=2\times n_12^{k_i}$, where we take $k_i=\left \lfloor \log_{2}{(2i-1)} \right \rfloor$.

\begin{table}[h]
\begin{center}
\caption{Relative error for Example 1 in heart-shaped domain. }
\label{table2}
\renewcommand{\arraystretch}{1.2}
\begin{tabular}{|c|c|c|c|}
\hline
$k_1$ &1&2&3\\
\hline
$Err_L^2$ &$4.56500\times 10^{-3}$&$1.14630\times 10^{-3}$&$0.32475\times 10^{-3}$\\
\hline
$Err_L^7$&$6.62500\times 10^{-3}$&$1.67670\times 10^{-3}$&$0.38854\times 10^{-3}$\\
\hline
$Err_L^{12}$&$7.03540\times 10^{-3}$&$1.57700\times 10^{-3}$&$0.38128\times 10^{-3}$\\
\hline
$Err_L^{17}$&$12.4350\times 10^{-3}$&$1.89370\times 10^{-3}$&$0.51314\times 10^{-3}$\\
\hline
$Err_A$&$10.3731\times 10^{-3}$&$4.52940\times 10^{-3}$&$0.51072\times 10^{-3}$\\
\hline
\end{tabular}
\label{Table-03-2}
\end{center}
\end{table}
\begin{table}[h]
\begin{center}
\caption{Relative error for Example 1 in elliptical domain. }
\label{table3}
\renewcommand{\arraystretch}{1.2}
\begin{tabular}{|c|c|c|c|}
\hline
$k_1$ &1&2&3\\
\hline
$Err_L^2$&$1.63660\times 10^{-3}$&$1.68310\times 10^{-3}$&$1.63660\times 10^{-3}$\\
\hline
$Err_L^7$&$4.12310\times 10^{-3}$&$1.22140\times 10^{-3}$&$1.21280\times 10^{-3}$\\
\hline
$Err_L^{12}$&$3.59170\times 10^{-3}$&$0.39568\times 10^{-3}$&$0.25058\times 10^{-3}$\\
\hline
$Err_L^{17}$&$4.99230\times 10^{-3}$&$0.66747\times 10^{-3}$&$0.51239\times 10^{-3}$\\
\hline
$Err_A$&$1.52940\times 10^{-3}$&$1.10763\times 10^{-3}$&$0.86416\times 10^{-3}$\\
\hline
\end{tabular}
\label{Table-03-3}
\end{center}
\end{table}

\begin{figure}[H]
\centering
\subfigure[Error for domain \eqref{nu1}.]
{
\begin{minipage}[t]{0.4\linewidth}
\centering
\includegraphics[scale=0.25]{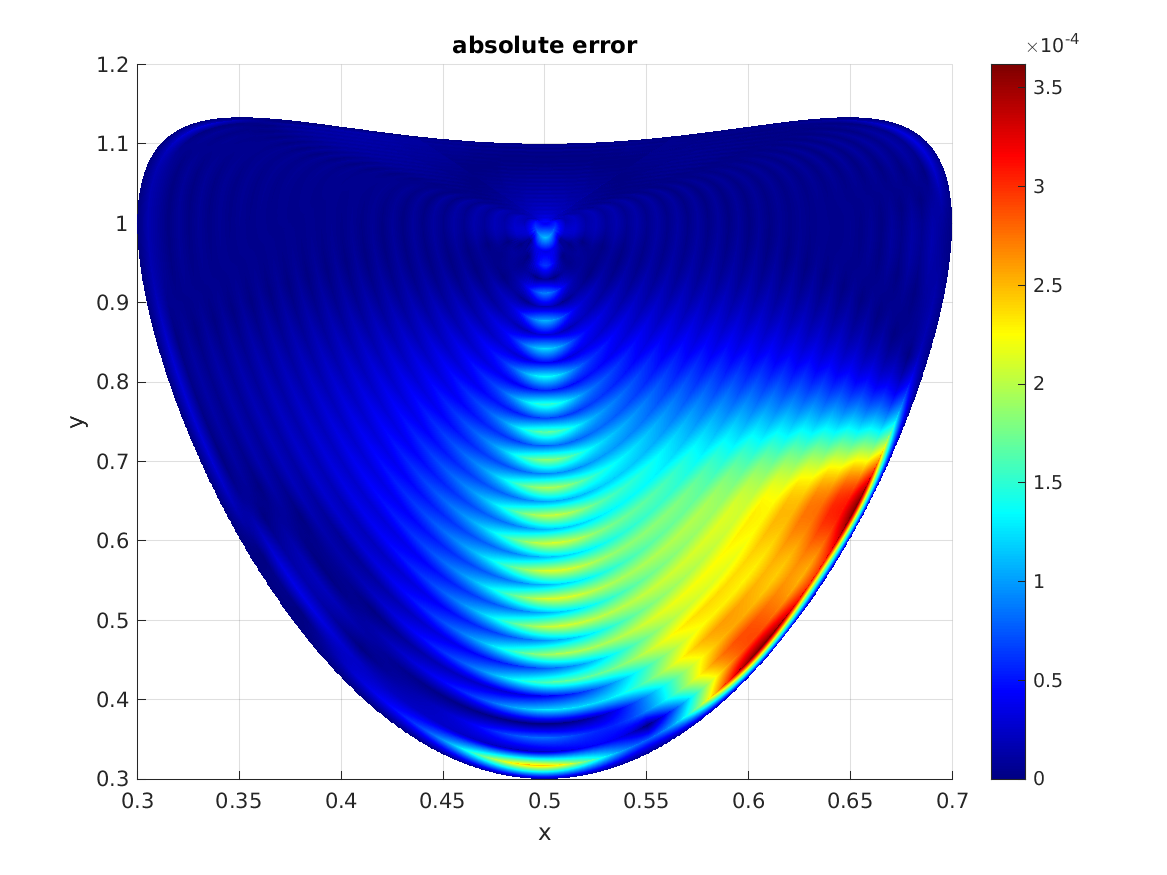}
\end{minipage}
}
\subfigure[Error for domain \eqref{nu2}.]
{
\begin{minipage}[t]{0.4\linewidth}
\centering
\includegraphics[scale=0.25]{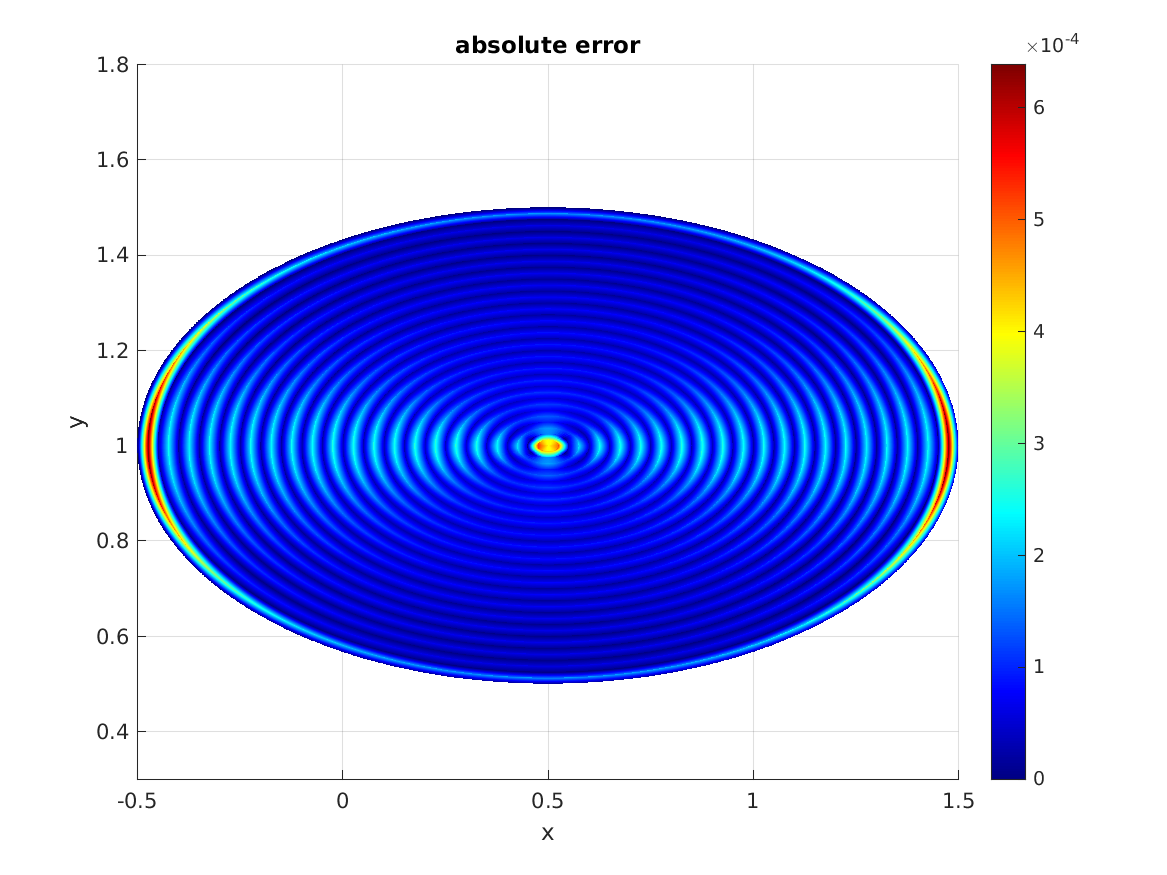}
%\caption{fig2}
\end{minipage}%
}%
\centering
\caption{Error distributions for Example 1 by refinement of $\Omega$  with $N=20,k_1=4$.}
\end{figure}

When we check the error distributions in $\Omega$ shown in Table 2 and Table 3, it can be seen that the error in all different closed curves $\Gamma_i$ and also the average error in the whole domain is always of the same amplitude $O(10^{-3})$, which reveal our uniform accuracy in the whole domain due to our self-adaptive partition strategy ADS in different $\Gamma_i$.

In the above implementations, we divide the domain $\Omega$ by the $\xi$-interval $[0,1]$  for fixed $N=10$ and $k_1=1,2,3$ dividing $\partial\Omega_1$ which also determines the partition size of $\Omega$. It is imaginable that the computation accuracy will be improved if we refine the grids of the domain $\Omega$, as shown in Figure 6 for two domains specified by \eqref{nu1} and \eqref{nu2}, where the absolute errors are improved to $O(10^{-4})$.

%Next we consider the second example with general $\sigma$ where the exact solution has no explicit expression different from Example 1. For checking the numerical performances, we compare our results with those obtained from the "pdetool" in MATLAB.

\vskip 0.3cm

{\bf 1B: Realization of DRM in section 4.}

In order to represent the accuracy of the solution by the DRM proposed in section 4, we define a mean absolute error and a  mean root square relative error by
\begin{align}
    Err_m :&\equiv \frac{1}{J}\sum\limits_{j=1}^{J}\left|u_{Nn}(x_j)-u_{ex}(x_j)\right|,\label{liu-23-DRMErr_m}\\
    Err_s :&\equiv \left (\frac{\sum\limits_{j=1}^{J}\left | u_{Nn}(x_j)-u_{ex}(x_j) \right |^2 }{\sum\limits_{j=1}^{J}\left | u_{ex}(x_j) \right |^2 }\right)^{1/2}\label{liu-23-DRMErr_s},
\end{align}
respectively,
where $\{x_j:j=1,\cdots,J\}\subset \Omega$ is the set of internal nodes, $u_{Nn}(x_j)$ and $u_{ex}(x_j)$ are the numerical solution and exact one in the grid $x_j$, respectively.  The distribution of internal nodes and boundary ones are presented in Figure \ref{liu-23-Nodes}.
\begin{figure}[htbp]
\centering
\subfigure[Node locations for domain \eqref{nu1}.]
{
\begin{minipage}[t]{0.45\linewidth}
\centering
\includegraphics[scale=0.32]{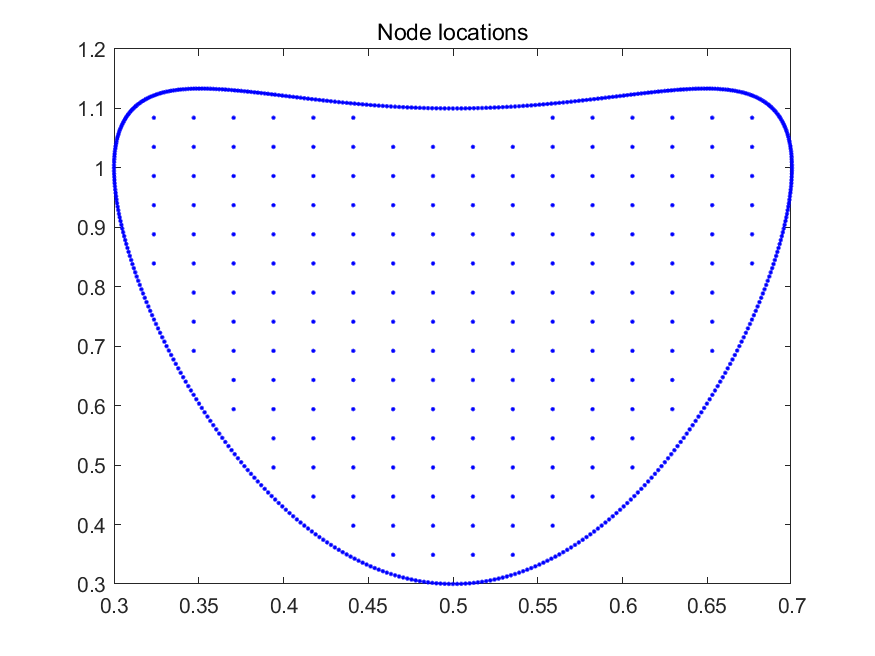}
\end{minipage}
}
\subfigure[Node locations for domain \eqref{nu2}.]
{
\begin{minipage}[t]{0.45\linewidth}
\centering
\includegraphics[scale=0.32]{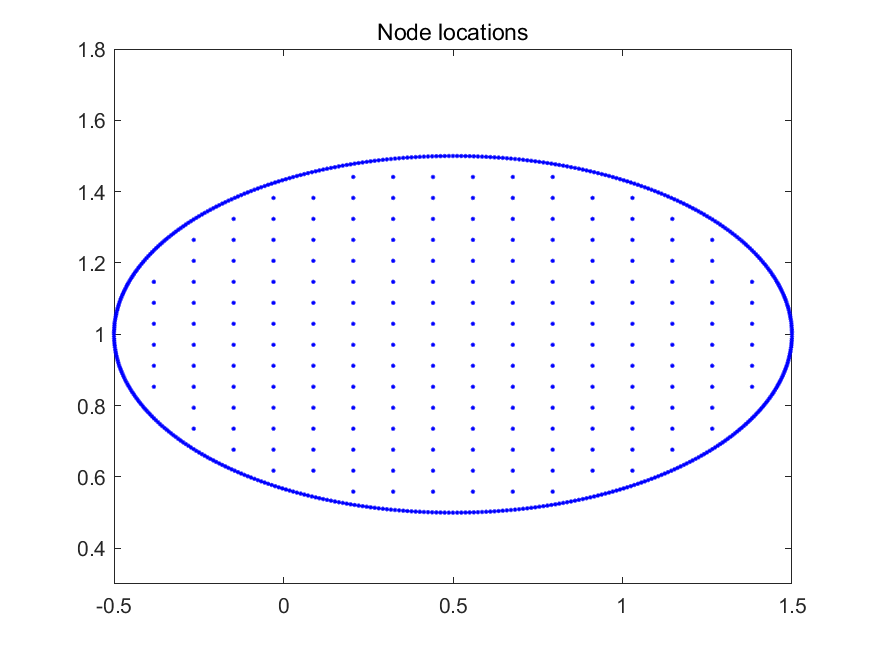}
\end{minipage}%
}%
\centering
\caption{Nodes distribution: (a) 196 internal nodes and 512 boundary nodes for domain \eqref{nu1}; (b) 208 internal nodes and 512 boundary nodes for domain \eqref{nu2}. }
\label{liu-23-Nodes}
\end{figure}

We compare the numerical performances of the DRM with the well-known FEM scheme. Since our proposed scheme is of 708 collocation nodes for
domain \eqref{nu1} and 720 collocation nodes for
domain \eqref{nu2} shown in Figure \ref{liu-23-Nodes}, we apply 717 nodes for domain \eqref{nu1} and 731 nodes for \eqref{nu2} to generate meshes in PDETOOL by Matlab for using FEM scheme, see (b) and (e) in Figure \ref{liu-23-result1}.

Under these discretizations with almost the same grids, our proposed DRM scheme is compatible with FEM.
Figure \ref{liu-23-result1} illustrates the results for two domains by DRM and the FEM in MATALB.
The left column of Figure \ref{liu-23-result1} shows the point-wise error distributions from our scheme, while the point-wise error for FEM is presented in the right column of Figure \ref{liu-23-result1}. It can be seen that the maximum errors of our scheme for two domains are always smaller than those of FEM for this example.

\begin{figure}[htbp]
\centering
\subfigure[Error for DRM in \eqref{nu1}.]{
\begin{minipage}[t]{0.3\linewidth}
\centering
\includegraphics[scale=0.28]{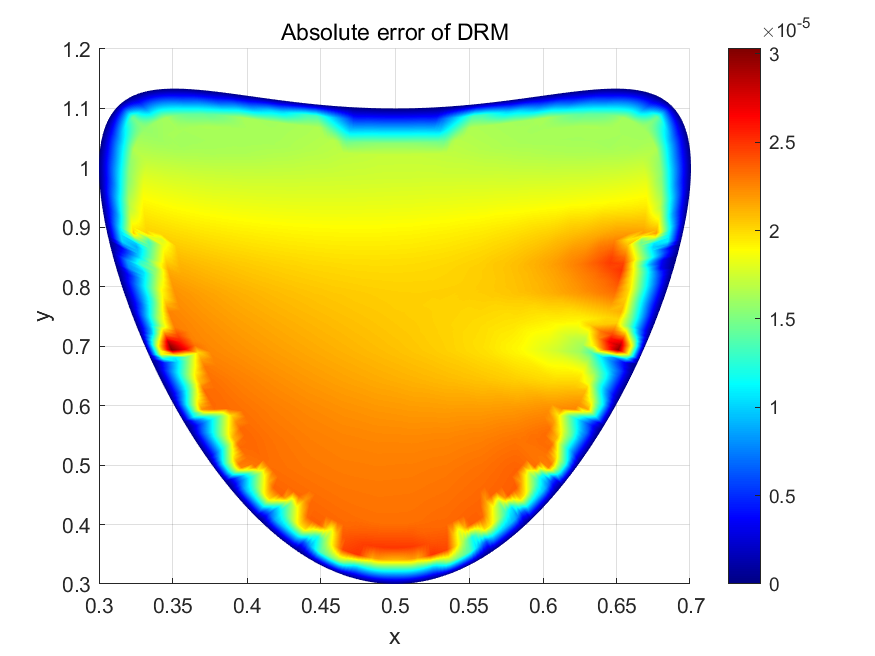}
\end{minipage}
}
\subfigure[Mesh of domain \eqref{nu1} by FEM.]{
\begin{minipage}[t]{0.33\linewidth}
\centering
\includegraphics[scale=0.3]{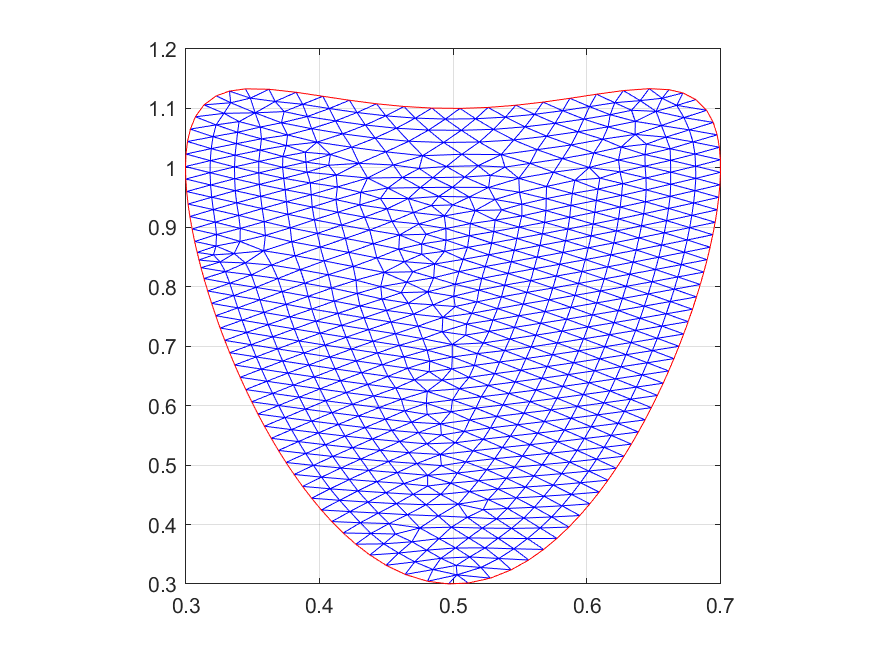}
\end{minipage}%
}%
\subfigure[Error for FEM in \eqref{nu1}.]{
\begin{minipage}[t]{0.33\linewidth}
\centering
\includegraphics[scale=0.28]{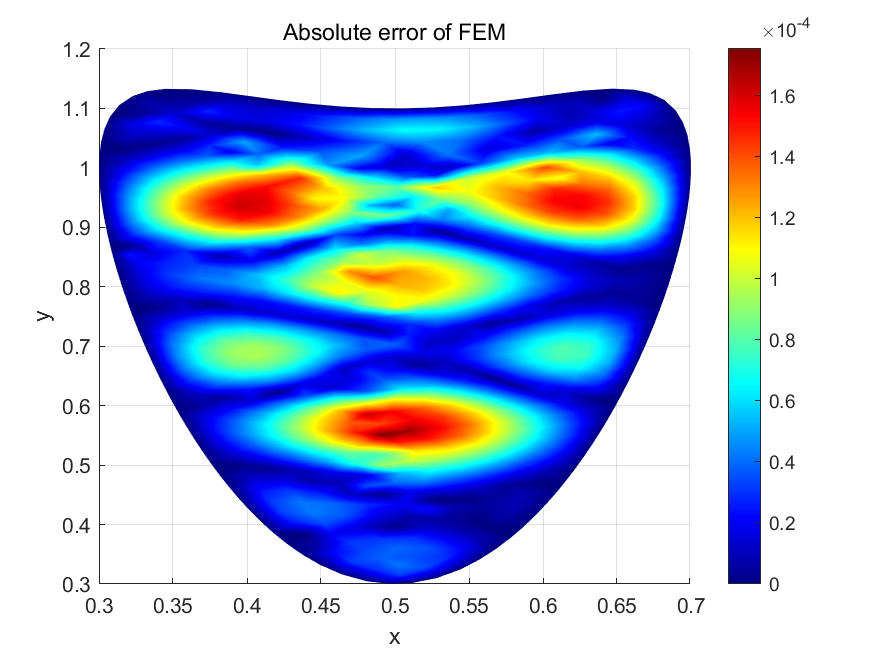}
%\caption{fig1}
\end{minipage}%
}
\\
\subfigure[Error for DRM in \eqref{nu2}.]{
\begin{minipage}[t]{0.3\linewidth}
\centering
\includegraphics[scale=0.28]{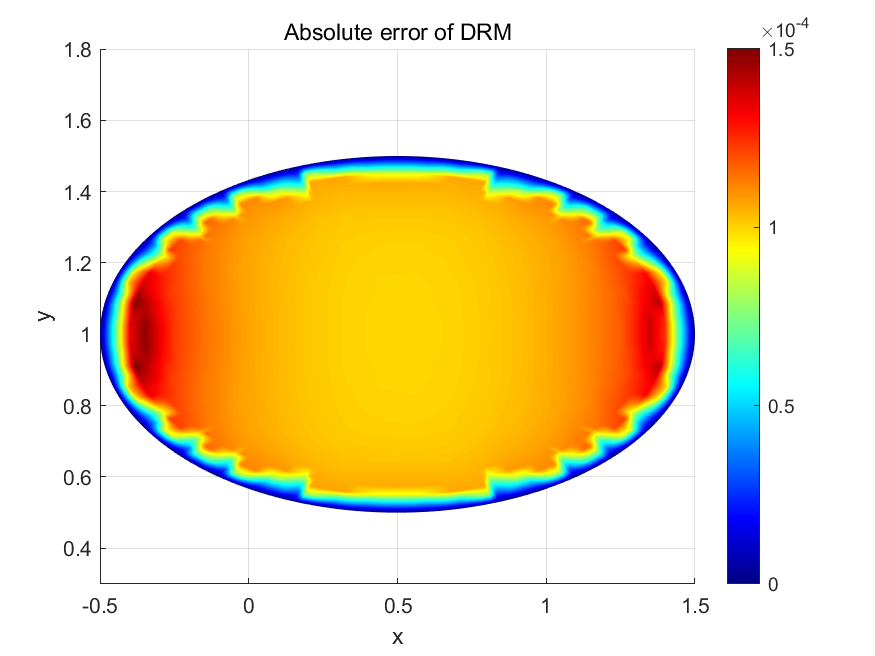}
%\caption{fig1}
\end{minipage}%
}%
\subfigure[Mesh of domain \eqref{nu2} by FEM.]{
\begin{minipage}[t]{0.33\linewidth}
\centering
\includegraphics[scale=0.28]{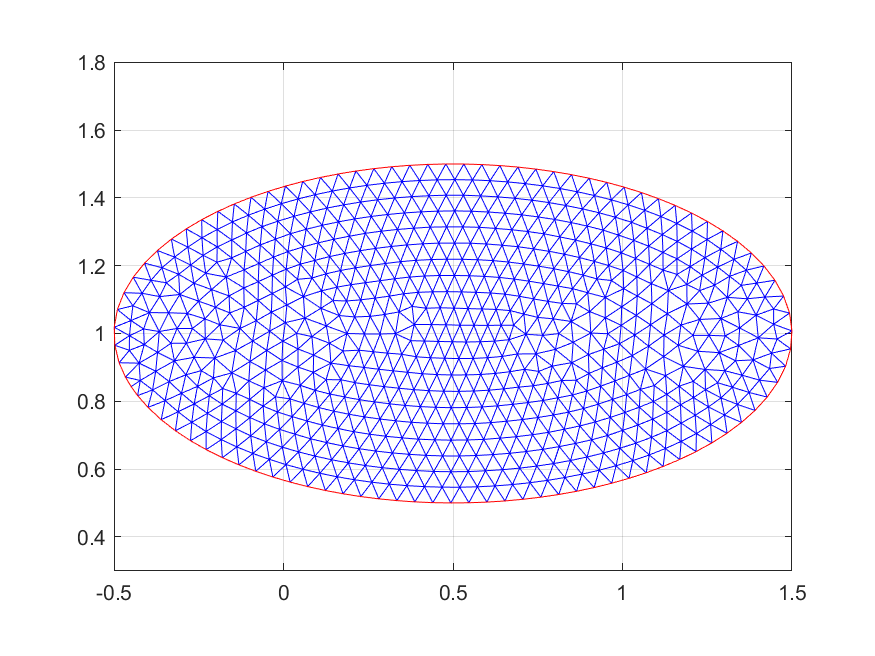}
%\caption{fig2}
\end{minipage}%
}%
\subfigure[Error for FEM in \eqref{nu2}.]{
\begin{minipage}[t]{0.33\linewidth}
\centering
\includegraphics[scale=0.28]{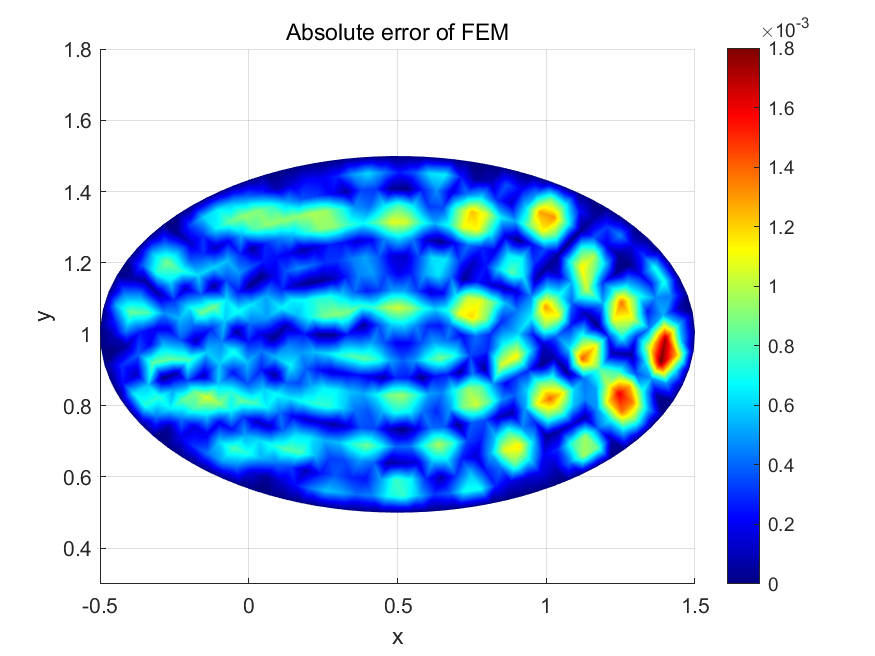}
%\caption{fig1}
\end{minipage}%
}
\centering
\caption{Numerical results of DRM and FEM in two domains.}
\label{liu-23-result1}
\end{figure}

\begin{table}[htbp]
\begin{center}
\caption{Numerics comparison for Example 1 in heart-shaped domain. }
\renewcommand{\arraystretch}{1.2}
\begin{tabular}{|c|c|c|c|c|}
\hline
&$Num_C$&$Err_m$&$Err_s$&$Time$\\
\hline
ADS &$856$&$2.3569\times 10^{-3}$&$3.9884\times 10^{-3}$&1.28768 s\\
\hline
DRM
&$708$&$2.0243\times 10^{-5}$&$1.1663\times 10^{-5}$&0.03191 s\\
\hline
FEM &$717$&$5.3564\times 10^{-5}$&$8.1742\times 10^{-5}$&0.48435 s\\
\hline
\end{tabular}
\label{liu23-DRM-Table-Ex1_x}
\end{center}
\end{table}

\begin{table}[htbp]
\begin{center}
\caption{Numerics comparison for Example 1 in elliptical domain. }
\renewcommand{\arraystretch}{1.2}
\begin{tabular}{|c|c|c|c|c|}
\hline
&$Num_C$&$Err_m$&$Err_s$&$Time$\\
\hline
ADS
&$856$&$4.3969\times 10^{-3}$&$3.6247\times 10^{-3}$&1.36882 s\\
\hline
DRM
&$720$&$1.0902\times 10^{-4}$&$6.6762\times 10^{-5}$&0.03326 s\\
\hline
FEM &$731$&$5.0242\times 10^{-4}$&$6.6272\times 10^{-4}$&0.47564 s\\
\hline
\end{tabular}
\label{liu23-DRM-Table-Ex1_p}
\end{center}
\end{table}

Now we compare the numerical performances of our proposed schemes ADS and DRM with FEM in the whole domain $\Omega$ in terms of average errors \eqref{liu-23-DRMErr_m} and \eqref{liu-23-DRMErr_s}, instead of the point-wise error shown in Figure \ref{liu-23-result1}. Table \ref{liu23-DRM-Table-Ex1_x} and Table \ref{liu23-DRM-Table-Ex1_p} present  $(Err_m,Err_s)$ of our schemes and FEM for domain \eqref{nu1}, domain \eqref{nu2}, respectively, where $Num_C$ is the number of collocation nodes.
It can be revealed that, even if  ADS chooses
more nodes ($856$), the accuracy ($O(10^{-3})$) is still lower than DRM and FEM, with longer computational time. However, our second scheme DRM is of the same error order as FEM for compatible collocation nodes. For domain $\eqref{nu2}$,  $Err_s$ of DRM ($O(10^{-5})$) is even lower than that for FEM ($O(10^{-4})$), as displayed in Table \ref{liu23-DRM-Table-Ex1_p}. It should be noted that our proposed scheme DRM takes less computational time than FEM, since mesh-generation of FEM is also time-consuming. However, DRM is a mesh-free scheme which chooses  internal nodes randomly.

We further investigate the computational accuracy of DRM  in terms of the number of internal grids. To this end, we keep the boundary grids unchanged but take different internal grids, as shown in Figure \ref{liu-23-examine-internal-nodes1} and Figure \ref{liu-23-examine-internal-nodes2}, where there are approximately $10, 15, 20, 30$ and $40$ internal grids for two domains. Such a configuration approximates density function $\tilde \mu$ by very few grids. We present the errors for $u$ in Table \ref{Table-DRM-examine-Ex1_x} and Table \ref{Table-DRM-examine-Ex1_p}, where $Num_I$ is the number of internal nodes.
\begin{figure}[htbp]
% \subfigure[Node location for domain\eqref{nu1}.]{
\subfigure{
\includegraphics[scale=0.2]{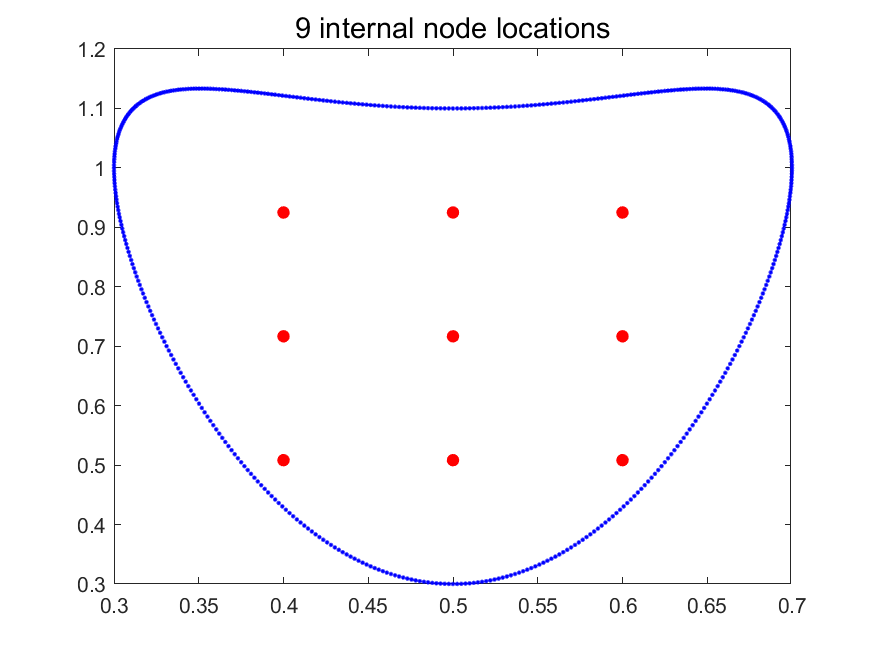}
\hspace{-0.2in}
}
\subfigure{
\includegraphics[scale=0.2]{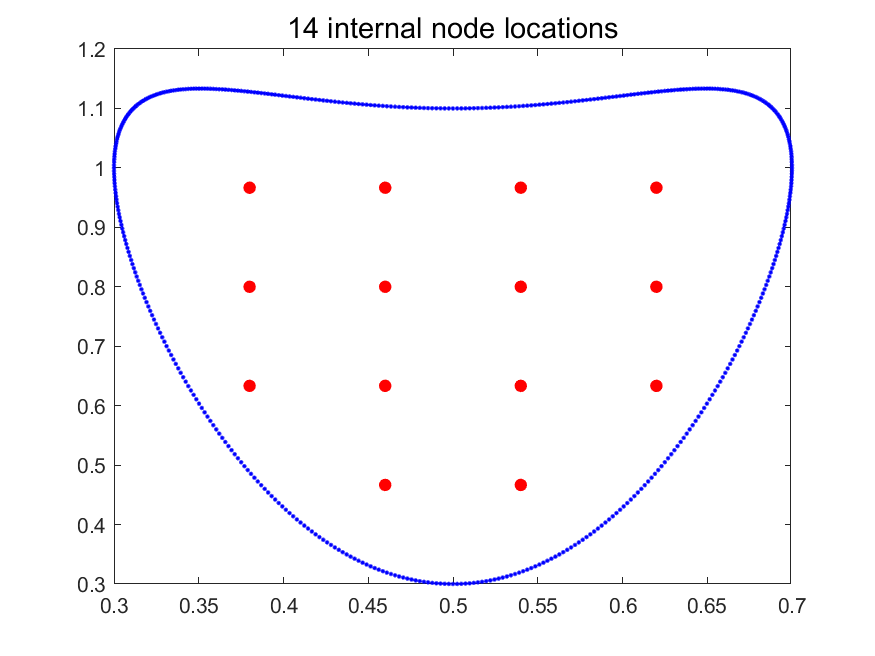}
\hspace{-0.2in}
}%
\subfigure{
\includegraphics[scale=0.2]{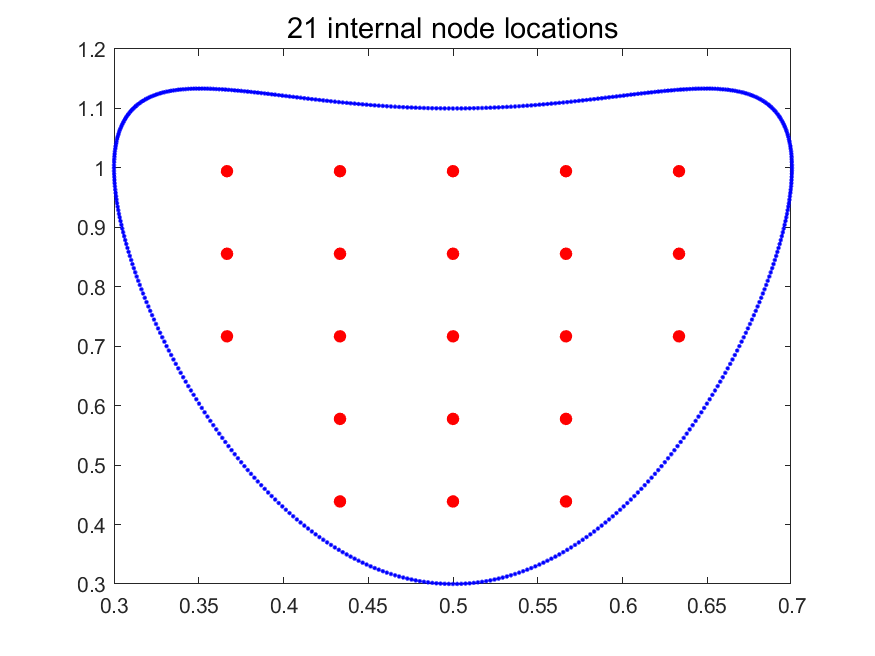}
%\caption{fig1}
\hspace{-0.25in}
}
\subfigure{
\includegraphics[scale=0.2]{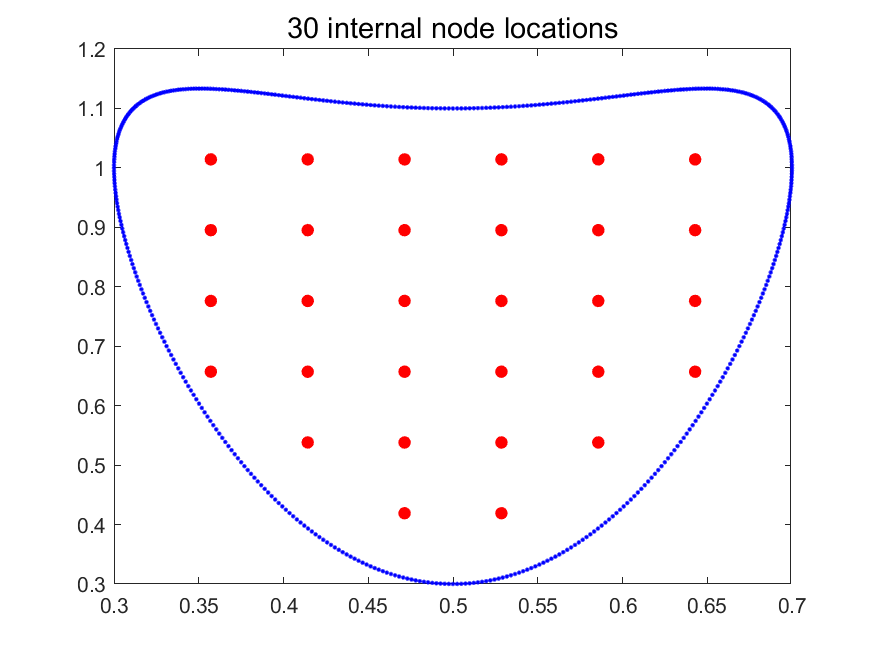}
\hspace{-0.25in}
}
\subfigure{
\includegraphics[scale=0.2]{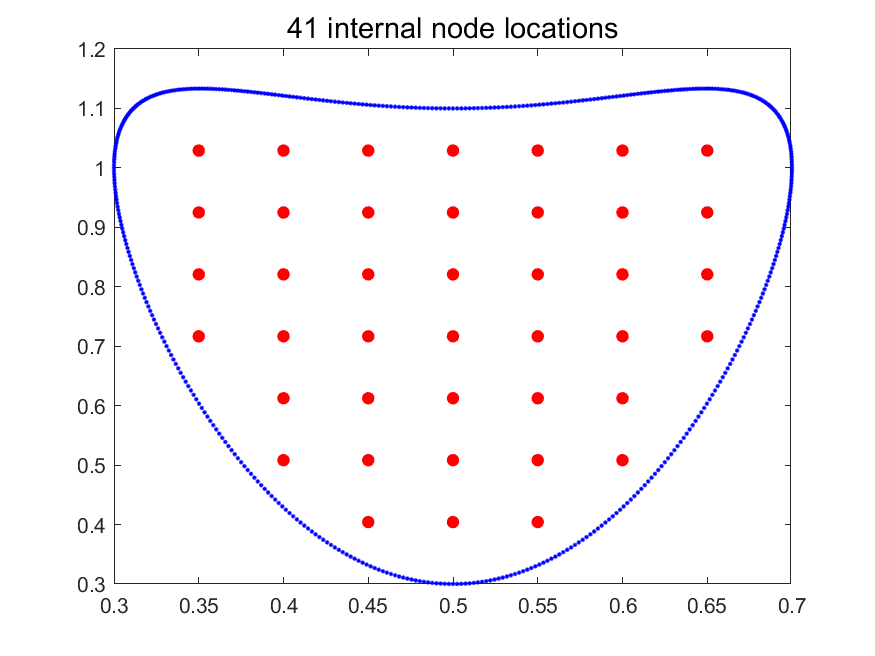}
}%
\caption{Different internal nodes distributions for domain \eqref{nu1}. }
\label{liu-23-examine-internal-nodes1}
\end{figure}
\begin{figure}[H]
% \subfigure[Node location for domain\eqref{nu1}.]{
\subfigure{
\includegraphics[scale=0.2]{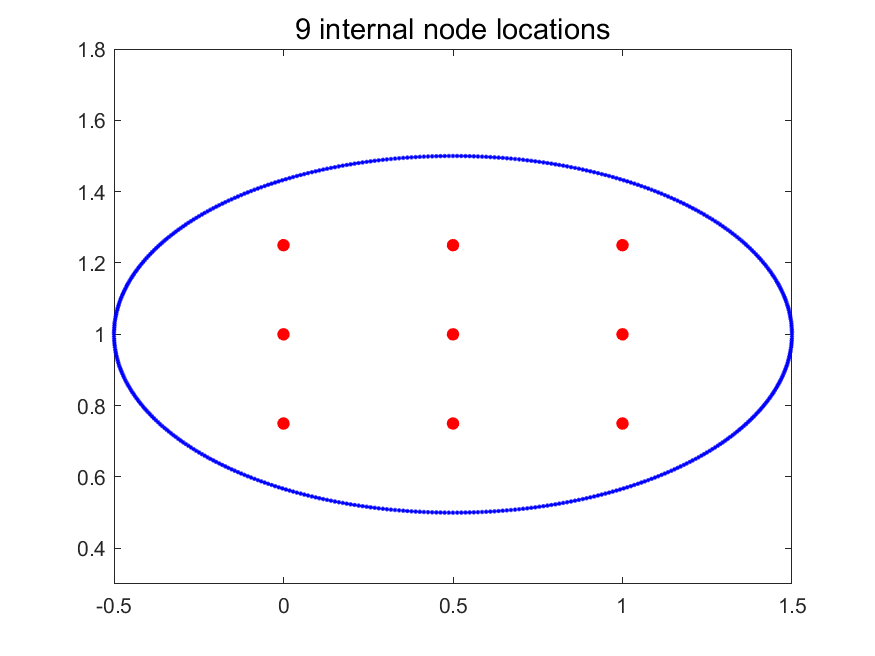}
\hspace{-0.2in}
}
\subfigure{
\includegraphics[scale=0.2]{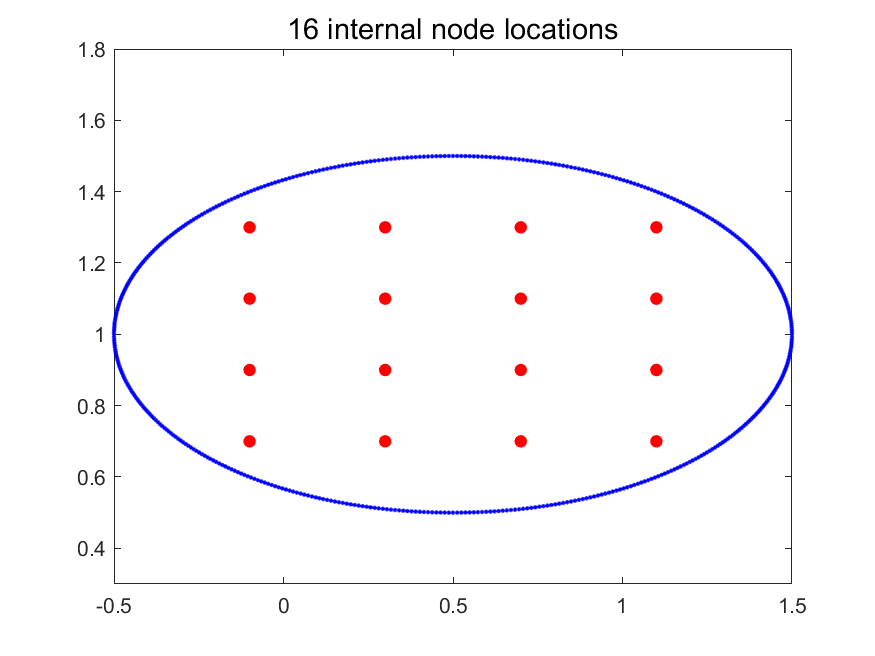}
\hspace{-0.2in}
}%
\subfigure{
\includegraphics[scale=0.2]{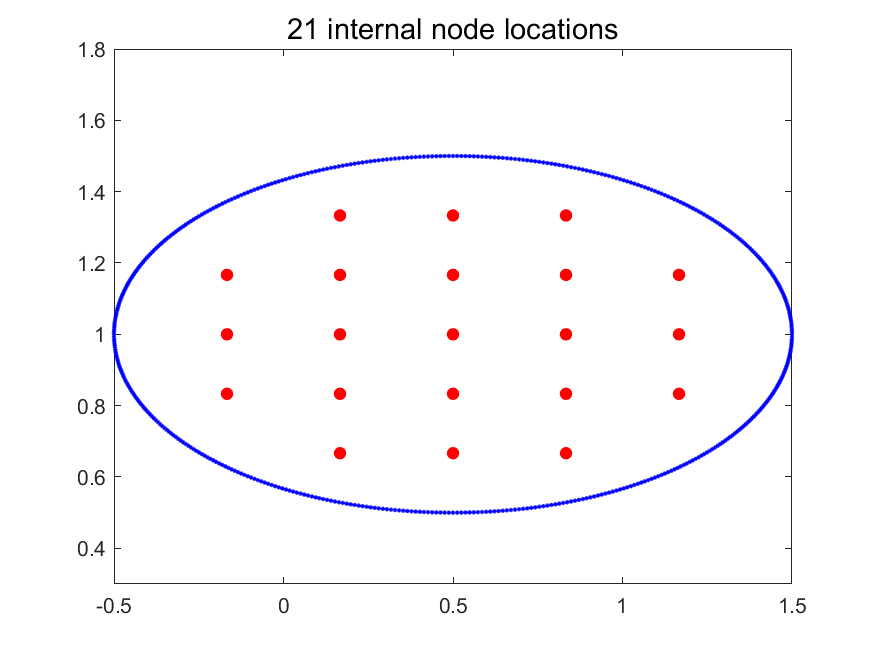}
%\caption{fig1}
\hspace{-0.25in}
}
\subfigure{
\includegraphics[scale=0.2]{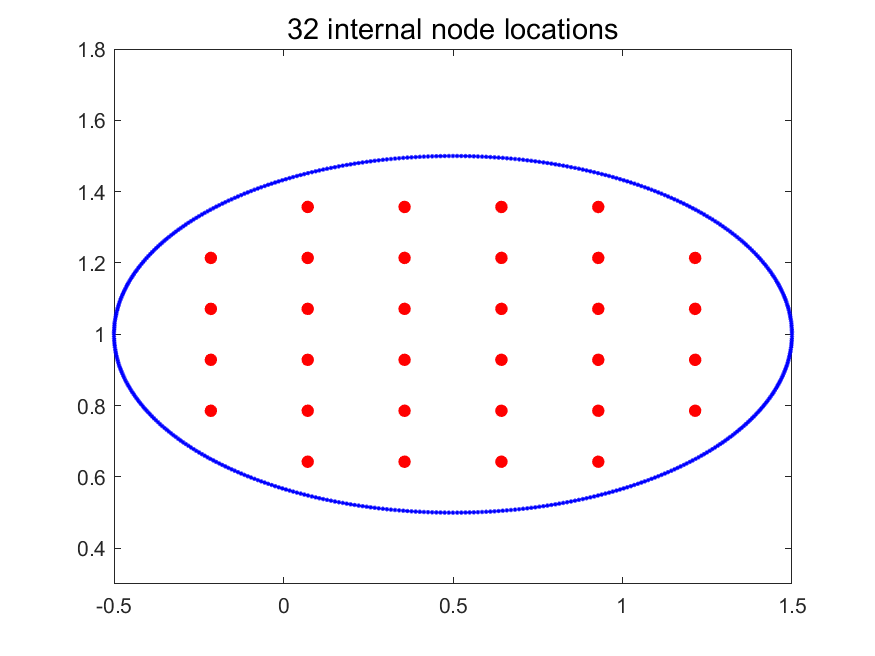}
\hspace{-0.25in}
}
\subfigure{
\includegraphics[scale=0.2]{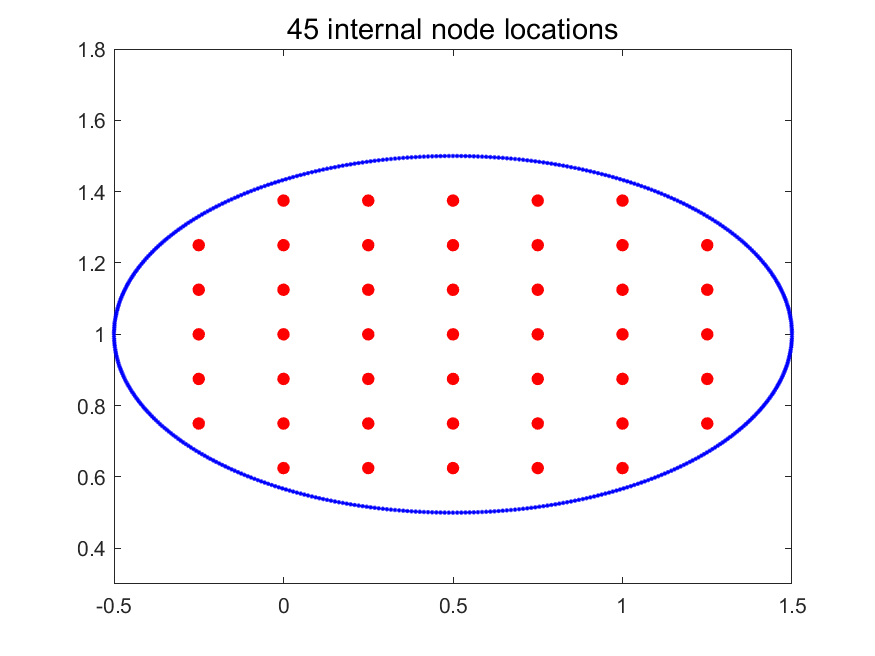}
}%
\caption{Different internal nodes distributions for domain \eqref{nu2}. }
\label{liu-23-examine-internal-nodes2}
\end{figure}

\begin{table}[ht]
\begin{center}
\caption{Error for Example 1 in heart-shaped domain. }
\renewcommand{\arraystretch}{1.2}
\begin{tabular}{|c|c|c|c|c|c|}
\hline
$Num_I$&9&14&21&30&41\\
\hline
$Err_m$ &$2.007\times 10^{-4}$&$9.061\times 10^{-5}$&$4.702\times 10^{-5}$&$3.323\times 10^{-5}$&$2.496\times 10^{-5}$\\
\hline
$Err_s$&$6.213\times 10^{-4}$&$6.818\times 10^{-5}$&$3.380\times 10^{-5}$&$2.234\times 10^{-5}$&$1.564\times 10^{-5}$\\
\hline
\end{tabular}
\label{Table-DRM-examine-Ex1_x}
\end{center}
\end{table}

\begin{table}[ht]
\begin{center}
\caption{Error for Example 1 in elliptical domain.}
\renewcommand{\arraystretch}{1.2}
\begin{tabular}{|c|c|c|c|c|c|}
\hline
$Num_I$&9&16&21&32&45\\
\hline
$Err_m$ &$1.570\times 10^{-3}$&$5.539\times 10^{-4}$&$4.683\times 10^{-4}$&$2.008\times 10^{-4}$&$1.396\times 10^{-4}$\\
\hline
$Err_s$&$1.231\times 10^{-3}$&$4.970\times 10^{-4}$&$3.672\times 10^{-4}$&$1.762\times 10^{-4}$&$1.095\times 10^{-4}$\\
\hline
\end{tabular}
\label{Table-DRM-examine-Ex1_p}
\end{center}
\end{table}

From Table \ref{Table-DRM-examine-Ex1_x} and Table \ref{Table-DRM-examine-Ex1_p}, it can be seen that our proposed DRM scheme can also yield satisfactory results which are good approximations to the exact solutions, even if we apply very few internal grids. For example, for 9 internal grids, the mean absolute errors are
$O(10^{-4})$ and $O(10^{-3})$ for  \eqref{nu1} and \eqref{nu2} respectively,  while the root mean square relative errors are $O(10^{-4})$ and $O(10^{-3})$, which are satisfactory for the numerics of our boundary value problem.

From Table \ref{liu23-DRM-Table-Ex1_x} and Table \ref{liu23-DRM-Table-Ex1_p} for 2-dimensional domain $\Omega$,  ADS needs to apply a large number of points and consequently spends longer time  to achieve satisfactory accuracy. Therefore, in the following 3-dimensional example, we only consider DRM for numerical realizations.

\vskip 0.3cm

{\bf Example 2. A 3-dimensional model.}

We take a pinched ball  to be $\Omega\subset \mathbb{R}^3$ for finding our solution $u$, with the boundary surface
$$
\partial \Omega=\left \{ r(\theta,\varphi)x(\theta,\varphi):\; x(\theta,\varphi)
=(\sin\theta\cos\varphi, \sin\theta\sin\varphi, \cos\theta)\in \mathbb{S}^2, \theta \in [0,\pi], \varphi \in [0,2\pi]\right\},
$$
where $r(\theta,\varphi)=\sqrt{(1.44+0.5\cos2\varphi(\cos2\theta-1)}$, see Figure \ref{liu-23-Ex3_surface} for the geometric configuration.

\begin{figure}[htbp]
\centering
\includegraphics[scale=0.43]{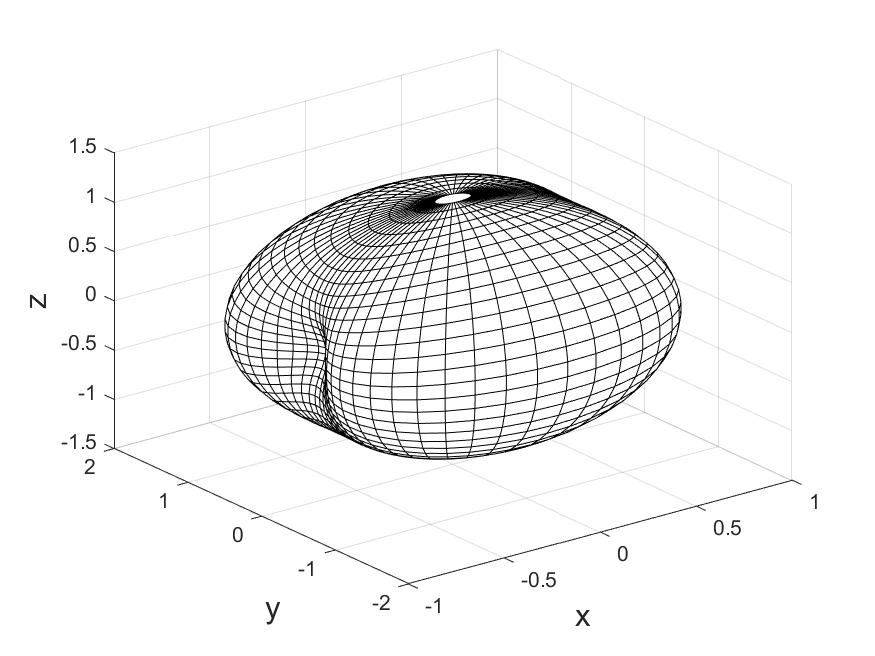}
\centering
\caption{Geometric shape for the pinched ball $\Omega$.}
\label{liu-23-Ex3_surface}
\end{figure}

We give the conductivity in $\Omega$ as
$$\sigma(x,y,z)=x^2y+2(y+z^2)+2$$
and the internal source function
$$F(x,y,z)=-(6y+2z)x^2-8yz-4y-20z-4z^2-4.$$
Then $u_{ex}(x,y,z)=x^2+2(y+2)z+1$
is the exact solution to \eqref{liu11-01} for corresponding $f\equiv u|_{\partial\Omega}$.

\begin{figure}[htbp]
\centering
\subfigure[Front view]{
\centering
\includegraphics[scale=0.3]{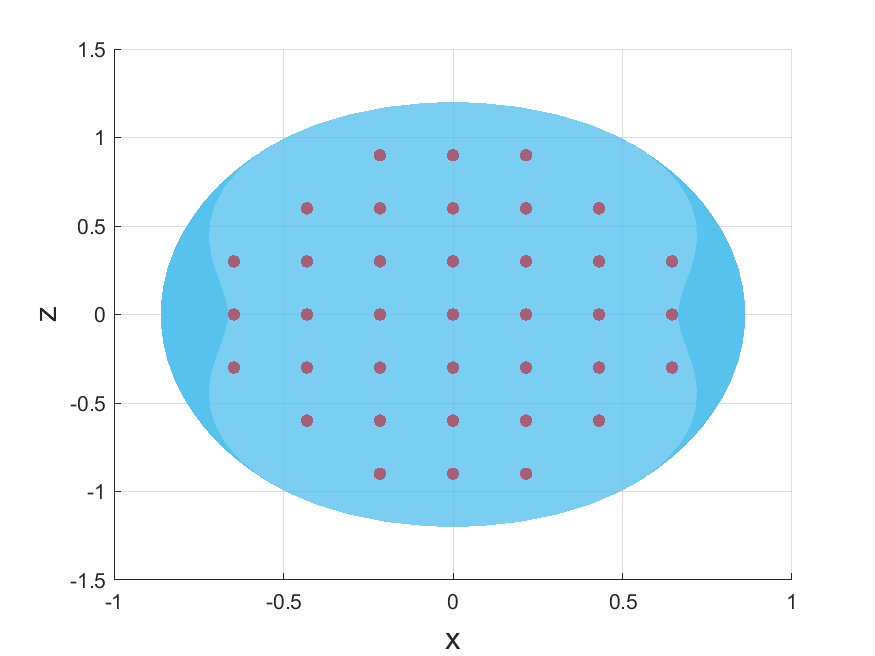}
}%
\subfigure[Right side view]{
\centering
\includegraphics[scale=0.3]{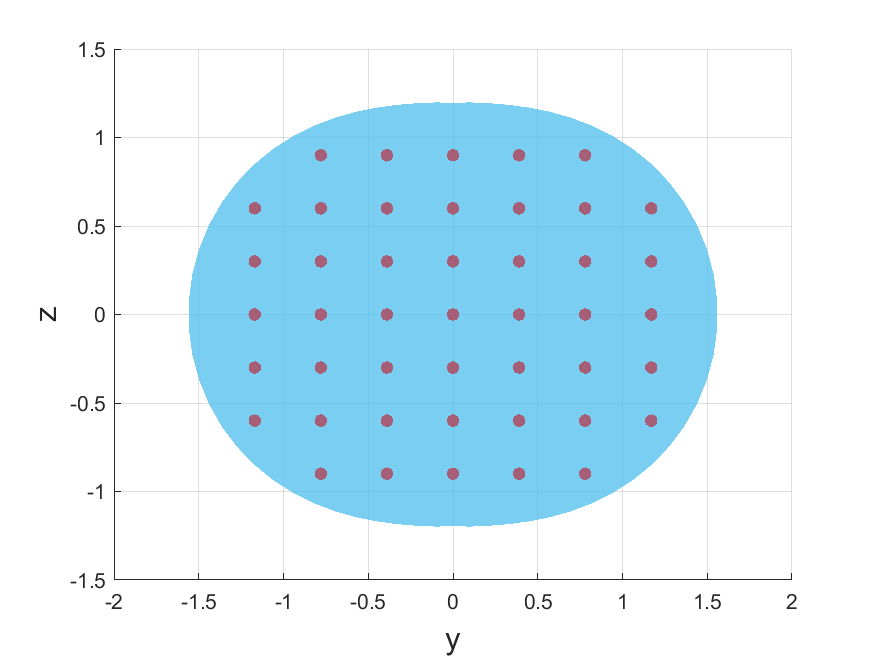}
%\caption{fig1}
}
\subfigure[Top view]{
\centering
\includegraphics[scale=0.3]{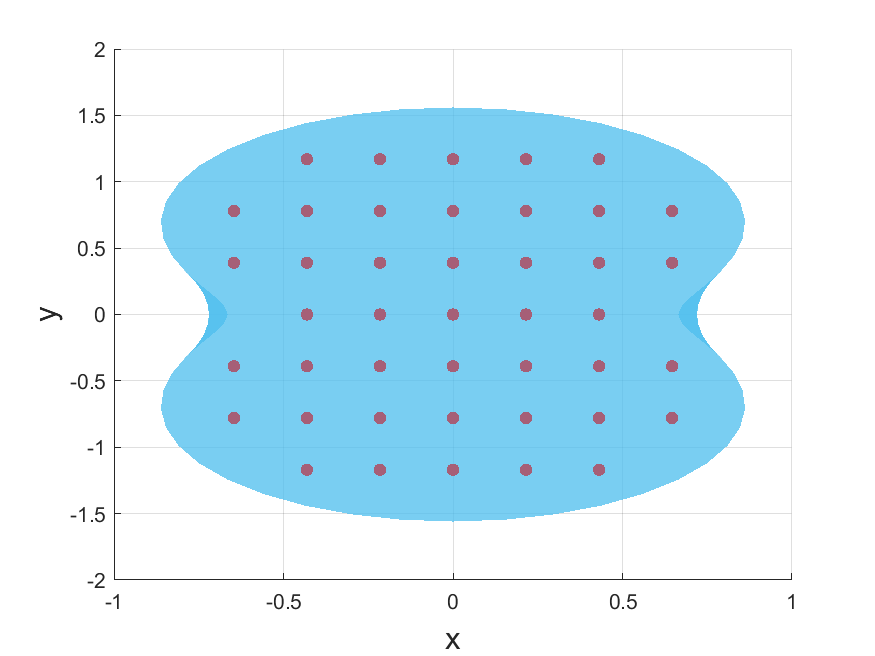}
%\caption{fig1}
}
\centering
\caption{Three views of 197 equally spaced internal nodes for $\Omega$.}
\label{liu-23-Ex3-internal nodes}
\end{figure}
We choose a set of quadrature nodes $x_{jk}$  on the unit sphere $\mathbb{S}^2$ in terms of the polar coordinates by \cite{Colton2}
$$
x_{jk}:\equiv(\sin \theta_j \cos \varphi_k, \sin \theta_j \sin \varphi_k, \cos \theta_j),\quad j=1,\cdots,N,\;  k=0,\cdots,2N-1,
$$
where $\theta_j = \arccos t_j, \varphi_k=\pi k/N$ for  the Gauss points $t_j$ in $[-1,1]$. Then we obtain $2N^2$ surface nodes $q_{jk}:\equiv r(\theta_j,\varphi_k)x_{jk}\in \partial\Omega$.
%When $N=32$, $2N^2=2048$ surface nodes $q_{jk}$ will discrete the surface of the unit sphere $\Omega$ into a set of boundary element as showed in Figure \ref{liu-23-Ex3_surface}.
 %The distribution of 197 equally spaced internal configuration nodes in $\Omega$ from the three views can be shown in Figure \ref{liu-23-Ex3-internal nodes}, and then
As for internal  nodes for $\Omega$, we take $197$ points  equally spaced in $\Omega$,  the distribution in $\Omega$ from  three views is shown in Figure \ref{liu-23-Ex3-internal nodes}.

\begin{figure}[htbp]
\centering
\subfigure[Exact solution]{
\centering
\includegraphics[scale=0.22]{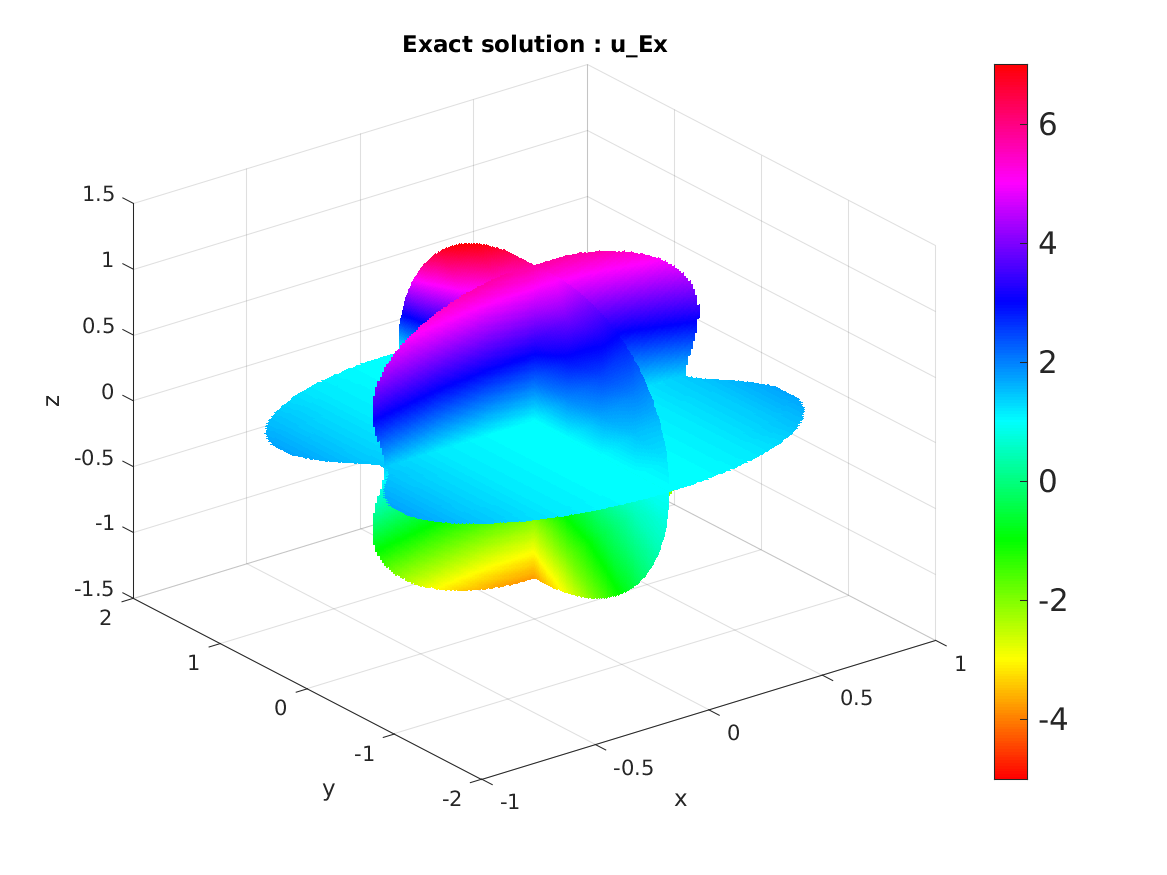}
}
\subfigure[Numerical solution]{
\centering
\includegraphics[scale=0.22]{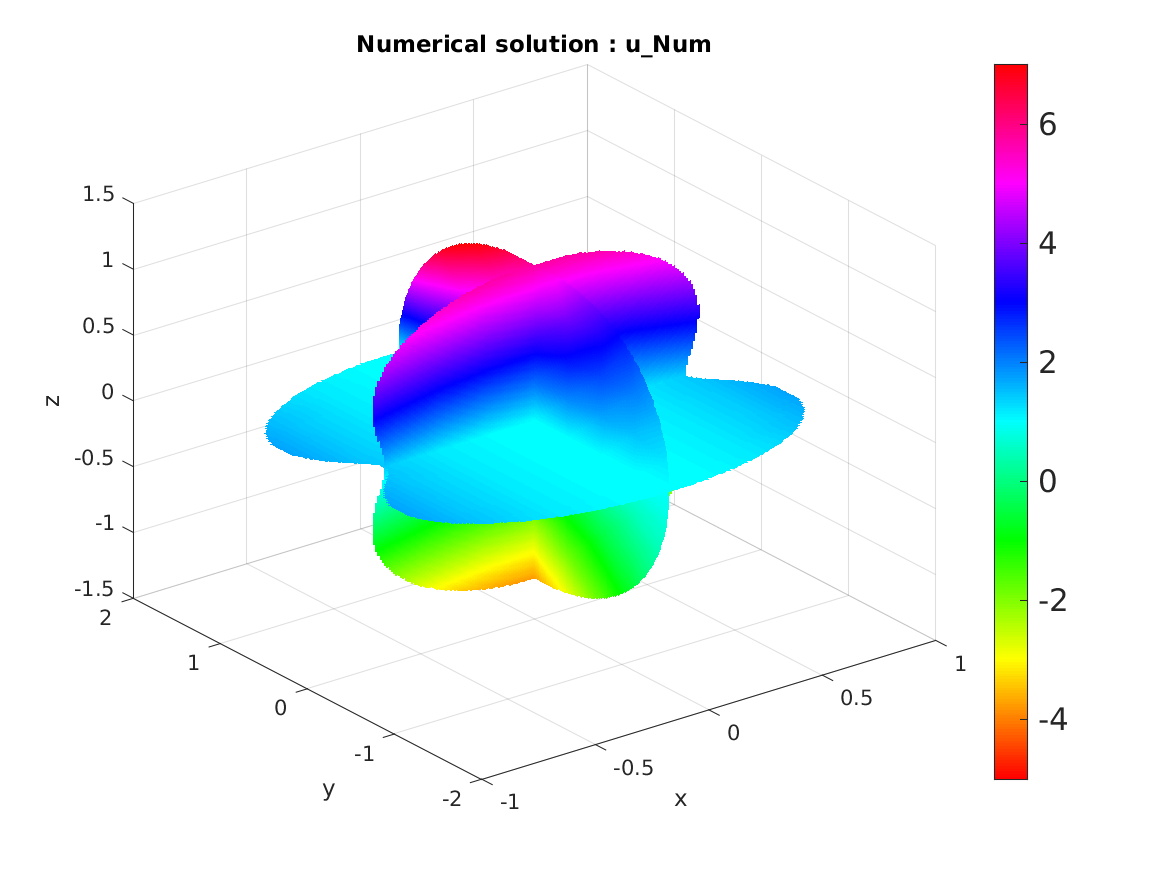}
}
\subfigure[Absolute error]{
\centering
\includegraphics[scale=0.22]{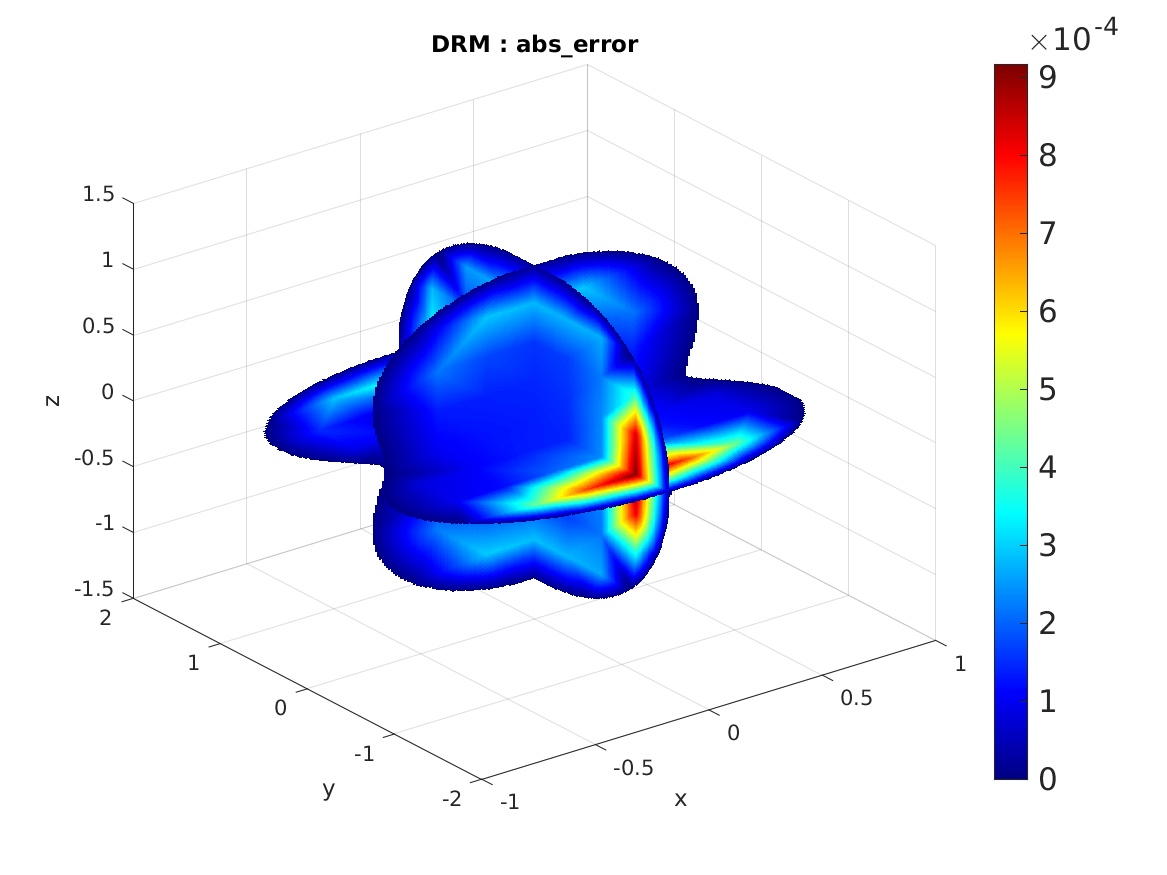}
}
\centering
\caption{Numerical results for Example 2 with $N=64$.}
\label{liu-23-Numerical result3}
\end{figure}

In Figure \ref{liu-23-Numerical result3}, we show the numerical performances of our DRM scheme with $N=64$ at three slice planes orthogonal to $x, y, z$ axis, respectively. The maximum absolute error is $9.1935\times 10^{-4}$, while the mean absolute error $Err_m$ and mean root square relative error $Err_s$ defined by \eqref{liu-23-DRMErr_m}, \eqref{liu-23-DRMErr_s} are $2.3467\times 10^{-4}$ and $1.2059\times 10^{-4}$. It can be observed that our DRM scheme based on the Levi function representation can solve 3-dimensional problems with satisfactory numerical performances, including accuracy and computational cost. On the other hand, the larger errors are generally appear near to the boundary $\partial\Omega$. This phenomenon may be due to the effect of singularity of fundamental solution for Laplacian equation at those nodes.

In order to test the influence of the number of internal nodes, we consider the configurations of $15, 27, 79, 136-$internal nodes for three set of surface nodes ($N=16, 32, 64$). The internal nodes locations are shown in
Figure \ref{liu-23-examine-internal-nodes3} from the top view. We compute the density pair $(\tilde \mu,\tilde \psi)$ at these relatively few nodes and then
compute  $u(x,y,z)$ in the whole domain $\Omega$ in terms of the Levi function representation. The results are shown in Table \ref{liu-23-Table-Ex3} for different configurations of boundary and internal nodes, in terms of the error distributions and computational time.
\begin{figure}[htbp]
\centering
% \subfigure[Node location for domain\eqref{nu1}.]{
\subfigure[]{
\includegraphics[scale=0.22]{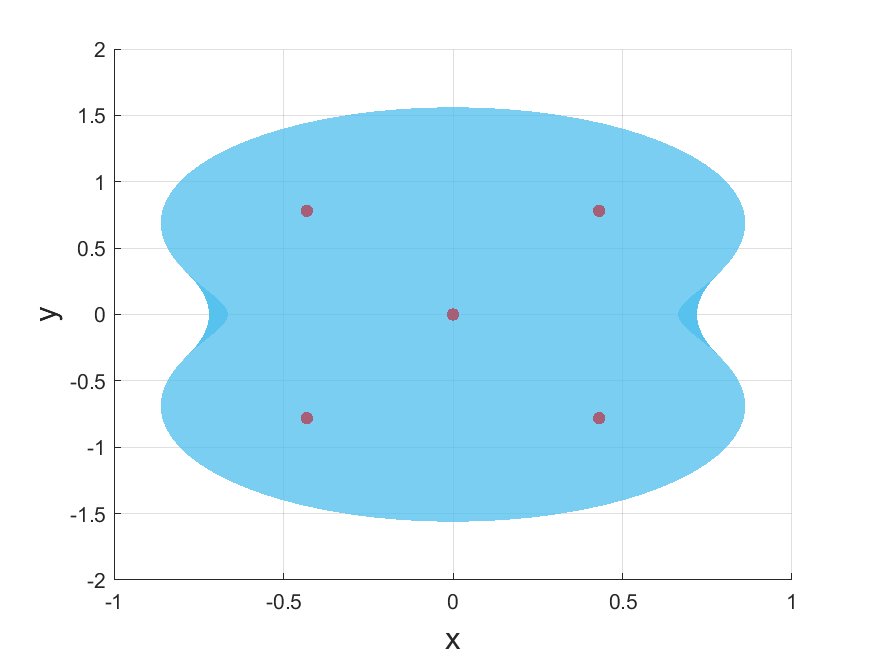}
%\hspace{-0.2in}
}
\centering
\subfigure[]{
\includegraphics[scale=0.22]{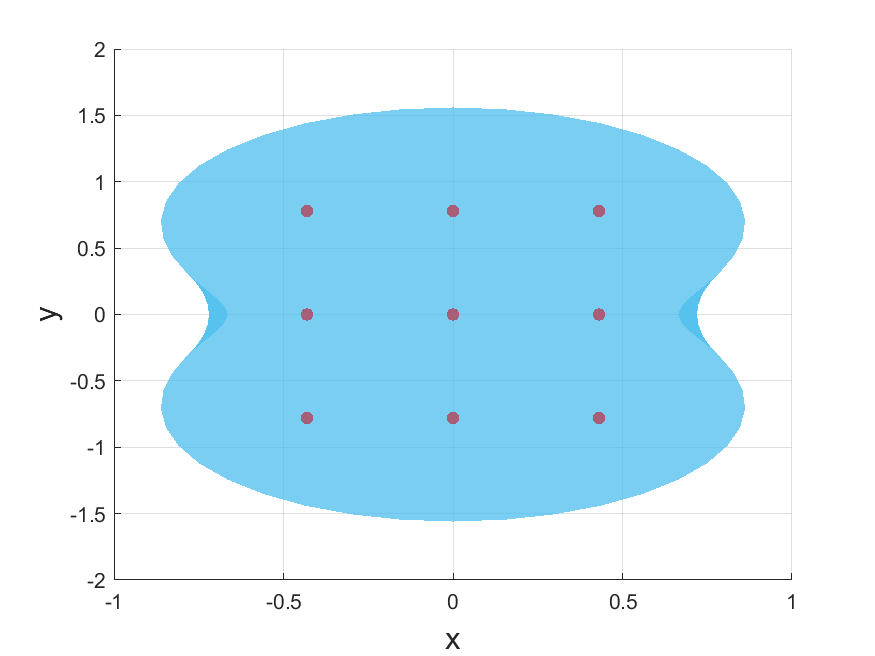}
%\hspace{-0.2in}
}%
\centering
\subfigure[]{
\includegraphics[scale=0.22]{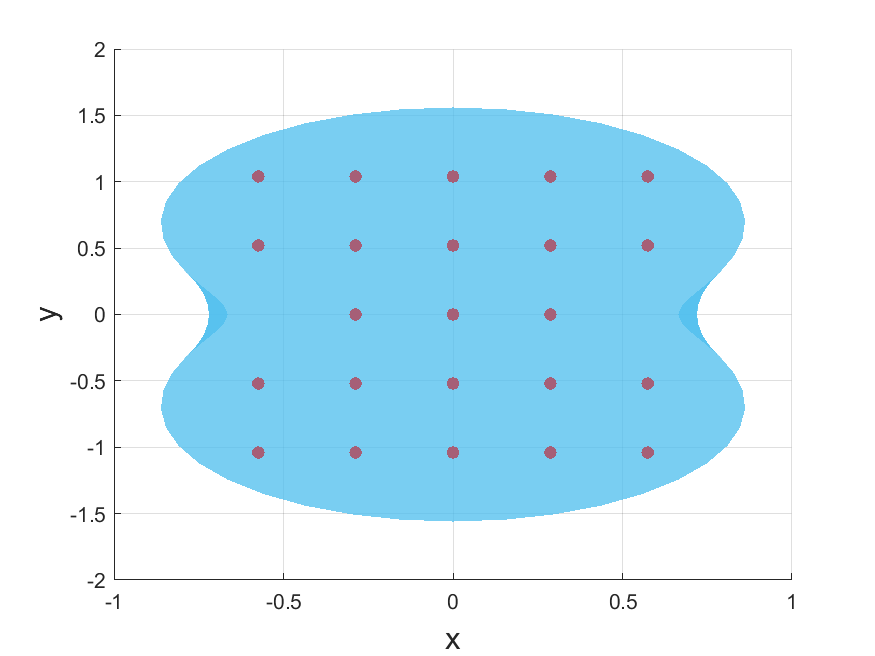}
%\hspace{-0.25in}
}
\centering
\subfigure[]{
\includegraphics[scale=0.22]{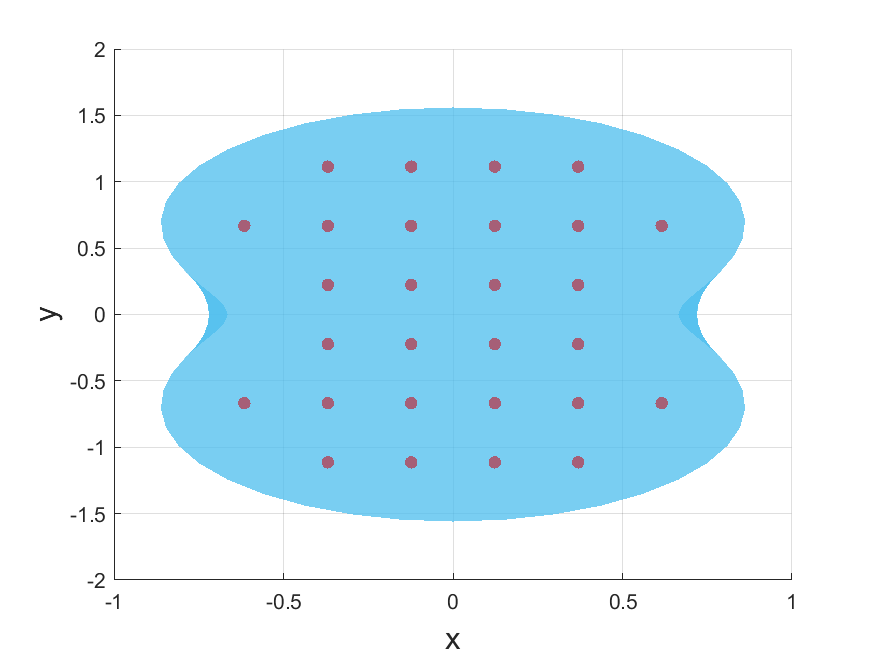}
}%
\centering
\caption{ Different internal nodes from top view: (a) 15 internal nodes; (b) 27 internal nodes; (c) 79 internal nodes; (d) 136 internal nodes.}
\label{liu-23-examine-internal-nodes3}
\end{figure}

\begin{table}[ht]
\begin{center}
\caption{Errors and computational time of Example 2 for different configurations.}
\renewcommand{\arraystretch}{1.2}
\begin{tabular}{|c|c|c|c|c|c|}
\hline
&$Num_I$&15&27&79&136\\
\hline
N=16&$Err_m$ &$1.5424\times 10^{-2}$&$1.2205\times 10^{-2}$&$1.0253
\times 10^{-2}$&$0.9422
\times 10^{-2}$\\
\hline
&$Err_s$&$1.3823
\times 10^{-2}$&$1.3128\times 10^{-2}$&$1.2035\times 10^{-2}$&$1.2364
\times 10^{-2}$\\
\hline
&$Time$&0.9240 s
&0.9737 s&1.1476 s&1.5368 s\\
\hline
N=32&$Err_m$ &$2.7741\times 10^{-3}$&$1.8511\times 10^{-3}$&$5.1004\times 10^{-4}$&$3.5751\times 10^{-4}$\\
\hline
&$Err_s$&$1.5346\times 10^{-3}$&$1.0089\times 10^{-3}$&$3.1122\times 10^{-4}$&$2.4034\times 10^{-4}$\\
\hline
&$Time$&61.285 s
&61.918 s&73.720 s&90.568s\\
\hline
N=64&$Err_m$ &$2.3722\times 10^{-3}$&$1.8523\times 10^{-3}$&$4.9840\times 10^{-4}$&$3.1304\times 10^{-4}$\\
\hline
&$Err_s$&$1.3361\times 10^{-3}$&$1.0025\times 10^{-3}$&$2.8860\times 10^{-4}$&$1.7465\times 10^{-4}$\\
\hline
&$Time$&2701.6 s
&2720.5 s&2853.7 s&2855.8 s\\
\hline
\end{tabular}
\label{liu-23-Table-Ex3}
\end{center}
\end{table}

%\begin{table}[ht]
%\begin{center}
%\caption{\yellowB{Error for} Example 2  }
%\renewcommand{\arraystretch}{1.2}
%\begin{tabular}{|c|c|c|c|c|c|c|}
%\hline
%&$Num_I$&15&27&56&79&136\\
%\hline
%N=16&$Err_m$ &$6.949\times 10^{-3}$&$5.691\times 10^{-4}$&$3.784\times 10^{-4}$&$2.399\times 10^{-4}$&$8.602\times 10^{-5}$\\
%\hline
%&$Err_s$&$4.908\times 10^{-3}$&$4.792\times 10^{-4}$&$3.185\times 10^{-4}$&$1.991\times 10^{-4}$&$7.912\times 10^{-5}$\\
%\hline
%N=32&$Err_m$ &$6.949\times 10^{-3}$&$5.691\times 10^{-4}$&$5.100\times 10^{-4}$&$3.575\times 10^{-4}$&$3.476\times 10^{-4}$\\
%\hline
%&$Err_s$&$4.908\times 10^{-3}$&$4.792\times 10^{-4}$&$3.1122\times 10^{-4}$&$2.403\times 10^{-4}$&$2.203\times 10^{-4}$\\
%\hline
%N=64&$Err_m$ &$6.949\times 10^{-3}$&$5.691\times 10^{-4}$&$4.9840\times 10^{-4}$&$3.1304\times 10^{-4}$&$2.347\times 10^{-4}$\\
%\hline
%&$Err_s$&$4.908\times 10^{-3}$&$4.792\times 10^{-4}$&$2.886\times 10^{-4}$&$1.747\times 10^{-4}$&$1.206\times 10^{-4}$\\
%\hline
%\end{tabular}
%\label{liu-23-Table-Ex3}
%\end{center}
%\end{table}

From Table \ref{liu-23-Table-Ex3}, it can be seen that for $N=16$, the computation time is shortest but the accuracy is the lowest ($O(10^{-2})$). For $N=32,64$, all the numerical results  are very close to exact solution with satisfactory error levels. Even for $15$ internal nodes, the mean absolute error and mean root square relative error are up to $O(10^{-3})$. Comparing the results for $N=32$ and $N=64$, the errors are very close, but the computational time is quite different. Consequently, we can apply relatively few boundary and internal nodes by our DRM scheme to get satisfactory numerical results.

\vskip 0.3cm

{\bf Acknowledgements: } This work is supported by NSFC (No.11971104).


\vskip 0.5cm

\begin{thebibliography}{00}\label{ref:ref}

\bibitem{Evans} L.C. Evans, Partial Differential Equations, Amer. Math. Soc., 2010.

\bibitem{George} G.C. Hsiao, W.L. Wendland, Boundary Integral Equations, Springer-Verlag, Berlin, 2008.

\bibitem{Chap1} R. Chapko, B.T. Johansson, On the numerical solution of a Cauchy problem for the Laplace equation via a direct integral
equation approach, Inverse Probl. Imaging, Vol.6, No.1, 25-38, 2012.

\bibitem{Chap2} R. Chapko, B.T. Johansson, A boundary integral equation method for numerical solution of parabolic and hyperbolic Cauchy problems, Appl. Numer. Math., Vol.129, 104-119, 2018.

\bibitem{Colton1} D. Colton, R. Kress, Inverse Acoustic and Electromagnetic Scattering Theory, Springer-Verlag, Berlin Heidelberg, 1992.

\bibitem{Pomp} A. Pomp, The Boundary-Domain Integral Method for Elliptic Systems with an Application to Shells, Lecture Notes in Mathematics, Vol.1683, Springer, Berlin, 1998.

\bibitem{Mir} C.M. Miranda, Partial Differential Equations of Elliptic Type, Springer, 2nd Ed., 1970.

\bibitem{Beshley} A. Beshley, R. Chapko, B.T. Johansson, An integral equation method for the numerical solution of a Dirichlet problem for second-order elliptic equations with variable coefficients, J. Engrg. Math., Vol.112, 63-73, 2018.

\bibitem{Mikha} S.E. Mikhailov, Analysis of united boundary-domain integro-differential and integral equations for a mixed BVP with variable coefficient, Math. Meth. Appl. Sci., Vol.29, No.6, 715-739, 2006.

\bibitem{DRM_Nardini} D. Nardini and C.A. Brebbia, A new approach for free vibration analysis using boundary elements, in Boundary
Element Methods in Engineering, C.A. Brebbia, ed., Springer, Berlin, 312-326, 1982.

\bibitem{RIM_Gao}X.W. Gao, A boundary element method without internal cells for two-dimensional and three-dimensional
elastoplastic problems, J. Appl. Mech., Vol.69, 154-160, 2002.

%\bibitem{Miranda}C. Miranda, Partial Differential Equations of Elliptic Type, Springer, 1970.

\bibitem{Ang} W.T. Ang, J. Kusuma, D.L. Clements, A boundary element method for a second order elliptic partial differential equation with
variable coefficients, Eng. Anal. Bound. Elem., Vol.18, No.4, 311-316, 1996.

\bibitem{Cle} D.L. Clements, A boundary integral equation method for the numerical solution of a second order elliptic equation with
variable coefficients, J. Aust. Math. Soc., Vol.22(B), 218-228, 1980.

\bibitem{Poz} C. Pozrikidis, Reciprocal identities and integral formulations for diffusive scalar transport and Stokes flow with positiondependent
diffusivity or viscosity, J. Engrg. Math., Vol.96, No.1, 595-114, 2016.

\bibitem{Jaw} M.A. AL-Jawary, L.C. Wrobel, Radial integration boundary integral and integro-differential equation methods for two-dimensional
heat conduction problems with variable coefficients, Eng. Anal. Bound. Elem., Vol.36, No.5, 685-695, 2012.

\bibitem{Chk} O. Chkadua, S.E. Mikhailov, D. Natroshvili, Analysis of direct boundary-domain integral equations for a mixed BVP with
variable coefficient, I: equivalence and invertibility, J. Int. Eqns. Appl., Vol.21, No.4, 499-543, 2009.

\bibitem{Cle2} D.L. Clements, A fundamental solution for linear second-order elliptic systems with variable coefficients, J. Engrg. Math., Vol.49, No.3, 209-216, 2004.


\bibitem{Colton} D. Colton, R. Kress, Integral Equation Methods in Scattering Theory, John Wiley $\&$ Sons, Inc., 1983.

%\bibitem{Kress} R. Kress, Linear Integral Equations, 3rd Ed., Springer, Berlin, 2013.

\bibitem{DRM_book} P.W. Partridge, C.A. Brebbia, and L.C. Wrobel, The Dual Reciprocity Boundary Element Method, Computational
Mechanics Publications, Southampton, 1992.

\bibitem{Colton2} D. Colton, R. Kress, Inverse Acoustic and Electromagnetic Scattering Theory, Springer-Verlag, Berlin Heidelberg, 1992.

\bibitem{DRM_Brunton} I. Brunton, Solving variable coefficient partial differential equations
using the boundary element method, Ph.D. Thesis, The University of
Auckland, Auckland, 1996.

\end{thebibliography}
\end{document}